\documentclass[11pt]{amsart}
\usepackage[left=1.4in,right=1.4in]{geometry}
\usepackage{amsmath,amsthm,amssymb,mathrsfs,hyperref}
\usepackage{xcolor}
\usepackage{pict2e}
\usepackage{dirtytalk}
\usepackage{ulem}
\definecolor{paperblue}{cmyk}{1,0,0,0}
\hypersetup{
  colorlinks=true,
  linkcolor=paperblue,
  citecolor=magenta!80!black,
  urlcolor=paperblue
}
\makeatletter

\renewcommand*\l@section{\@tocline{1}{8pt plus 2pt}{0pt}{2em}{\bfseries}}
\renewcommand*\l@subsection{\@tocline{2}{2pt}{1.5em}{3em}{}}
\renewcommand*\l@subsubsection{\@tocline{3}{1pt}{3em}{4em}{}}
\makeatother
\numberwithin{equation}{section}
\newtheorem{theorem}{Theorem}[section]
\newtheorem{lemma}[theorem]{Lemma}
\newtheorem{proposition}[theorem]{Proposition}
\newtheorem{corollary}[theorem]{Corollary}

\newtheorem{definition}[theorem]{Definition}
\newtheorem{example}[theorem]{Example}

\newcommand{\bC}{\mathbb{C}}

\newcommand{\bZ}{\mathbb{Z}}
\newcommand{\bQ}{\mathbb{Q}}
\newcommand{\bR}{\mathbb{R}}
\newcommand{\bN}{\mathbb{N}}
\newcommand{\bP}{\mathbb{P}}
\newcommand{\cA}{\mathcal{A}}
\newcommand{\Z}{\mathbb{Z}}

\newcommand{\cT}{\mathcal{T}}

\newcommand{\cO}{\mathcal{O}}

\newcommand{\bfB}{\mathbf{B}}
\newcommand{\bfT}{\mathbf{T}}
\newcommand{\bfU}{\mathbf{U}}
\newcommand{\bfG}{\mathbf{G}}

\newcommand{\mfp}{\mathfrak{p}}

\newcommand{\SL}{\operatorname{SL}}
\newcommand{\SA}{\operatorname{SA}}
\newcommand{\UCP}{\operatorname{UCP}}
\newcommand{\rk}{\operatorname{rk}}
\newcommand{\Sub}{\operatorname{Sub}}
\newcommand{\GL}{\operatorname{GL}}
\newcommand{\PGL}{\operatorname{PGL}}
\newcommand{\IRS}{\operatorname{IRS}}

\newcommand{\PSL}{\operatorname{PSL}}
\newcommand{\diag}{\operatorname{diag}}
\newcommand{\Tr}{\operatorname{Tr}}
\newcommand{\trivgp}{\langle 1 \rangle}

\newcommand{\E}{\mathbb{E}}

\newcommand{\supp}{\operatorname{supp}}

\title[Boomerang subalgebras \& IRM Rigidity]{Invariant random maps and boomerang subalgebras of group von Neumann algebras}
\author[Amrutam]{Tattwamasi Amrutam}
\address{Institute of Mathematics of the Polish Academy of Sciences, ul. Sniadeckich 8, 00-656, Warszawa, Poland}
\email{tattwamasiamrutam@gmail.com}
\author[Glasner]{Yair Glasner}
\address{Ben Gurion University of the Negev.
	Department of Mathematics.
	Be'er Sheva, 8410501, Israel.
}
\email{yairgl@bgu.ac.il}
\author[Hartman]{Yair Hartman}
\address{Ben Gurion University of the Negev.
	Department of Mathematics.
	Be'er Sheva, 8410501, Israel.
}
\email{hartmany@bgu.ac.il}
\author[Jiang]{Yongle Jiang}
\address{School of Mathematical Sciences, Dalian University of Technology, Dalian, 116024, China}
\email{yonglejiang@dlut.edu.cn}
\date{\today}
\begin{document}
\begin{abstract}
With a countable discrete group one associates a few topological dynamical systems and natural $\Gamma$-maps between them 
\[\Sub(\Gamma) \stackrel{L}{\longrightarrow} \SA(L(\Gamma)) \stackrel{\E}{\longrightarrow} \UCP_{\tau_0}(L(\Gamma))\]
These are, respectively, the space of subgroups, of subalgebras of the von-Neumann algebra $L(\Gamma)$ and of normal trace-preserving unital
completely positive (UCP) self maps on $L(\Gamma)$. Invariant probability measures on these spaces, are known as invariant random subgroups (IRSs), invariant random algebras (IRAs) and invariant random maps (IRMs). The latter are introduced here for the first time. By Poincar\'{e} recurrence, all such measures are supported on boomerang elements - elements that return arbitrarily close to their original position under successive applications of any group element $\gamma \in \Gamma$. 

A group is called IRM rigid if every IRM is supported on Fourier multipliers. These are the maps $T \in \UCP_{\tau_0}(L(\Gamma))$ of the form $T(\lambda(\gamma))= \phi(\gamma)\lambda(\gamma)$ for a scalar UCP map $\phi:L(\Gamma) \rightarrow \bC$.
We establish IRM rigidity for many classes of countable groups. Including acylindrically hyperbolic groups with trivial
amenable radical, nontrivial
mixed identity-free groups and finite products of such groups. In fact we prove the stronger statement that every boomerang $T \in \UCP_{\tau_0}(L(\Gamma))$ is a Fourier multiplier. 
As a consequence we deduce that every IRA in such groups is of the form $L_*(\mu)$ for some $\mu \in \IRS(\Gamma)$. We establish the same conclusion for adjoint
$K$-split higher rank $S$-arithmetic groups.
\end{abstract}
\maketitle

\setcounter{tocdepth}{1}
\tableofcontents
\section{Introduction}
Let $\Gamma$ be a countable group. We denote by $\Sub(\Gamma)$ the
space of its subgroups, endowed with the compact metrizable Chabauty
topology and the conjugation action of $\Gamma$. Ab\'ert, Glasner and
Vir\'ag~\cite{abert2014kesten} introduced the notion of invariant random
subgroup (IRS) (which is a $\Gamma$-invariant Borel probability measure on
$\Sub(\Gamma)$), and extended Kesten's spectral radius criterion for
amenability from normal subgroups to IRSs. The study of IRSs has since
become a fruitful direction in its own right, with applications to
the geometry of locally symmetric and $\operatorname{CAT}(0)$
spaces~\cite{abert2017growth,duchesne2015geometric}, to stabilizers of
measure-preserving actions and measured equivalence
relations~\cite{creutz2017stabilizers,bowen2015irs}, and to amenability
and permutation stability~\cite{bader2016amenable,becker2019stability}.

Amrutam, Hartman and Oppelmayer~\cite{amrutam2309amenable} took up
the analogous question for group von Neumann algebras. Let
$L(\Gamma)$ be the group von Neumann algebra of $\Gamma$, with
canonical trace $\tau_0$ and canonical unitaries
$\{\lambda(g)\}_{g\in\Gamma}$. We denote by $\SA(L(\Gamma))$ the space
of its unital von Neumann subalgebras, endowed with the
Effros--Mar\'echal topology~\cite{Effros, marechal1973topologie, haagerup1998effros} (also see~\cite{FLMP:ucp}) and the conjugation action of $\Gamma$.
Invariant random algebras (IRAs) are the $\Gamma$-invariant Borel
probability measures on this space. The study of IRAs has also been extended to crossed products.
Zhou~\cite{zhou2024noncommutative} studied amenable invariant random
intermediate subalgebras in this setting. As mentioned above, the notion of IRS and the measure theoretic methods that it brought with it into group theory turned out to be very fruitful, and the hope is that IRAs might play a similar role in the von Neumann algebra setting.

A subgroup $\Delta\leq\Gamma$ determines a subalgebra $L(\Delta)$,
and a subalgebra $M\leq L(\Gamma)$ in turn determines its
trace-preserving conditional expectation $\E_M$. The latter is a
normal trace-preserving unital completely positive (UCP) map on
$L(\Gamma)$, and is, in fact, a projection onto $M$. It turns out that the assignments $\Delta\mapsto
L(\Delta)$ and $M\mapsto\E_M$ are continuous and equivariant
embeddings, where the space $\UCP_{\tau_0}(L(\Gamma))$ of
such maps is endowed with the point ultraweak topology (see, for example, Lemma~\ref{lem:diagram-inclusions}). Thus an
IRS pushes forward to an IRA, and an IRA pushes forward to an
invariant probability measure on the space of UCP maps.

This suggests taking one more step and considering the invariant
probability measures on the whole space of UCP maps. We call them
invariant random maps (IRMs) (see Definition~\ref{def:irm}). The action on maps is given by the two sided conjugation action. One way to obtain an IRM is to start with a $\Gamma$-invariant
probability measure on the space $\operatorname{PD}_1(\Gamma)$ of
normalized positive-definite functions. Such measures are the
distributions of the invariant random positive-definite functions of
Alekseev and Brugger~\cite[Definition~3.1]{AB:irpdf}. Each
$\phi\in\operatorname{PD}_1(\Gamma)$ gives a Fourier multiplier
$T_\phi$ with $T_\phi(\lambda(g))=\phi(g)\lambda(g)$. The map
$\mathfrak{M}_\Gamma\colon\phi\mapsto T_\phi$ is continuous and
$\Gamma$-equivariant by Lemma~\ref{lem:fourier-multiplier-map}, so
pushforward under $\mathfrak{M}_\Gamma$ gives an IRM.
We may then ask when every IRM is obtained in this way. We call $\Gamma$
IRM rigid if every IRM is supported on Fourier multipliers.

There is also corresponding question for IRAs. The map $L$ sends
an IRS $\nu$ to the IRA $L_*\nu$. In the other direction we use
the equivariant Borel map $J$ given by
$J(M)=\Gamma_M:=\{\gamma\mid\lambda(\gamma)\in M\}$. Clearly
$JL(\Delta)=\Delta$ and $LJ(M)\leq M$. We call $\Gamma$ IRA rigid
if every IRA is of the form $L_*\nu$ for an IRS $\nu$.

Invariant subalgebras are the simplest examples of IRAs. The Dirac
measure $\delta_M$ is an IRA exactly when $M$ is $\Gamma$-invariant.
Thus IRA rigidity implies the ISR property, which says that every
invariant subalgebra of $L(\Gamma)$ is $L(N)$ for some normal subgroup
$N$ of $\Gamma$. Kalantar and
Panagopoulos~\cite[Theorem~1.1]{kalantar2023invariant} established this rigidity for irreducible lattices in connected semisimple Lie
groups with trivial center and no nontrivial compact factors, all
of whose simple factors have real rank at least two. Since then, a series of recent works 
\cite{amrutam2023invariant}, 
\cite{CDS2023}, \cite{jiang2024example}, \cite{dudko2024character} have established the ISR property in various settings.
The study of invariant subalgebras has developed into a theory of its own; see also \cite{amrutam2025invariant,amrutam-relative-isr}.

In this paper, we are interested in both IRM and IRA rigidity. As it turns out, IRM rigidity implies IRA rigidity. Indeed, push an IRA $\mu$ forward
under $M\mapsto\E_M$. If $\Gamma$ is IRM rigid, then $\E_M$ is a Fourier
multiplier for $\mu$-almost every $M$. Since $\E_M$ is an idempotent, it follows that
$M=L(J(M))$. Thus $\mu=L_*\nu$ for $\nu=J_*\mu$, and
$JL=\operatorname{id}$ gives uniqueness.

Our first result established IRM rigidity for products of groups with rank one behaviour.
In particular, we show that acylindrically hyperbolic groups with
trivial amenable radical, as well as finite products of such groups are IRM rigid, hence IRA rigid. 
\begin{theorem}\label{thm:hyperbolic_tf}
Let $n\geq1$ and let $\Gamma_i$ be countable acylindrically hyperbolic
groups with trivial amenable radical. Let
$\Gamma=\oplus_{1\leq i\leq n}\Gamma_i$. Then, $\Gamma$ is IRM rigid. Consequently, every IRA $\mu$ of $\Gamma$ is of the form $\mu=L_*\nu$
for a unique IRS $\nu$ of $\Gamma$.
\end{theorem}

Our motivation for the proof comes from boomerang subgroups, introduced
by Glasner and Lederle~\cite{glasner2024boomerang},  \cite{G:lin_IRS}. Inspired by these, in a similar spirit, we introduce the notion of boomerang UCP
maps and boomerang subalgebras. These are recurrent points for
conjugation by every element of $\Gamma$ separately, in the topology
of the corresponding space. By Poincar\'e recurrence, every IRM
gives full measure to the boomerang UCP maps. The same holds for
IRAs in its respective space
(see \cite[Example~2.9(5)]{glasner2024boomerang}).
This is even more interesting in the rank one case, since there are many boomerang subgroups that do not come
from IRSs; see, for example,  \cite{Lederle2026,GHL:ec}.

To study boomerang UCP maps, we use the dictionary of positive definite functions. In particular, we first identify the space of
UCP maps with a subclass of normalized positive-definite functions (see Proposition~\ref{prop:ira-ucp-correspondence}). For the groups in
Theorem~\ref{thm:hyperbolic_tf}, we use recurrence and the geometry of
acylindrical actions to show that every boomerang UCP is a Fourier multiplier (see~Theorem~\ref{thm:boomerang-ucp-ah}). Applying the same result to conditional expectations gives a
statement about every boomerang subalgebra.
\begin{theorem}\label{thm:hyperbolic_tf_boom}
Let $n\geq1$ and let $\Gamma_i$ be countable acylindrically hyperbolic
groups with trivial amenable radical. Let
$\Gamma=\oplus_{1\leq i\leq n}\Gamma_i$. Then every boomerang
subalgebra $M\in\SA(L(\Gamma))$ is of the form $L(\Delta)$ for a
unique boomerang subgroup $\Delta\in\Sub(\Gamma)$.
\end{theorem}

We do not know whether the ISR property, or the property that every
boomerang subalgebra is a group von Neumann algebra, is preserved
under finite direct products. Thus the product statement in
Theorem~\ref{thm:hyperbolic_tf_boom} requires a separate argument.

It is natural to ask whether there is a dynamical characterization
of IRM rigid groups. We take a first step in this direction.
For every nontrivial countable group $\Gamma$, the free product
$\Gamma*\langle t\rangle$, where $\langle t\rangle\cong\mathbb Z$,
is acylindrically hyperbolic (see, for example, Minasyan and
Osin~\cite[Corollary~2.2 and the paragraph following it]{minasyan-osin2015trees}). Using \cite[Corollary~7.3(a)]{osin2016acylindrically}, its amenable
radical is trivial. Thus $\Gamma*\langle t\rangle$ is IRM rigid by
Theorem~\ref{thm:hyperbolic_tf}. To obtain IRM rigidity
for $\Gamma$ from this larger group, we use mixed identity-freeness.

Recall that a nontrivial group $\Gamma$ is mixed identity-free if
every nonidentity element of $\Gamma*\langle t\rangle$ survives
under some homomorphism $\Gamma*\langle t\rangle\to\Gamma$ fixing
$\Gamma$ pointwise, where $\langle t\rangle\cong\mathbb Z$
(see Hull and Osin~\cite{hull-osin2016}).
\begin{theorem}\label{thm:ira-mif-intro}
Every finite product of countable nontrivial mixed identity-free
groups is IRM rigid. In particular, it is IRA rigid.
\end{theorem}

For a mixed identity-free group $\Gamma$, we use the positive-definite
function dictionary (see~Proposition~\ref{prop:ira-ucp-correspondence}) to transfer an IRM from $\Gamma$ to
$\Gamma*\langle t\rangle$. We then apply
Theorem~\ref{thm:hyperbolic_tf} to the larger group and
recover the original invariant probability measure by restricting
the functions to $\Gamma\times\Gamma$. For finite products we use
this construction in each factor.

As a consequence, we obtain IRM rigidity for finite products of
groups admitting the following actions.
\begin{theorem}\label{thm:ira-boundaries-trees-intro}
Let $n\geq1$ and let $\Gamma_i$ be countable groups for
$1\leq i\leq n$. Suppose that each $\Gamma_i$ satisfies one of the
following conditions.
\begin{enumerate}
\item $\Gamma_i$ admits a topologically free minimal extremely
proximal action on a compact Hausdorff space with more than two points.
\item $\Gamma_i$ admits a faithful minimal action of general type
on a tree, and the induced action on the boundary of the tree is
topologically free.
\end{enumerate}
Then $\Gamma=\oplus_{1\leq i\leq n}\Gamma_i$ is IRM rigid. In particular, the Baumslag--Solitar groups $\operatorname{BS}(m,n)$
with $|m|,|n|\geq2$ and $|m|\neq|n|$, and their finite products, are
IRM rigid
\end{theorem}
Theorem~\ref{thm:ira-mif-intro} also gives amenable examples. There
exists a two-generated infinite elementary amenable IRM rigid group
and, for each prime $p$, an infinite locally finite IRM rigid
$p$-group. We also obtain continuum many pairwise nonisomorphic
finitely generated nonamenable IRM rigid groups which contain no
nonabelian free subgroup and have nontrivial amenable radical
(see Corollary~\ref{cor:jacobson}). Moreover, every countable amenable
group embeds into a finitely generated amenable IRM rigid group,
and the overgroup can be chosen elementary amenable when the given
group is elementary amenable (see Corollary~\ref{cor:embedding}).

The conclusion of Theorem~\ref{thm:hyperbolic_tf_boom} can fail for
mixed identity-free groups (see Example~\ref{ex:mif-boomerang}). Our Theorem~\ref{thm:ira-mif-intro} also applies to linear groups.
Zariski dense subgroups of $\PGL_n(k)$ over an algebraically closed
field $k$ are mixed identity-free for $n\geq2$. In particular,
$\PSL_n(\bZ)$ is IRM rigid for every $n\geq2$ (see
Corollary~\ref{cor:linear-irm}).

If $\Gamma$ is IRM rigid and has property $\mathrm{CR}_{\mathrm{fin}}$
(Definition~\ref{def:cr-fin}), we can describe the maps that occur in an IRM.
By Theorem~\ref{prop:irm-char}, almost every map is a convex combination
of the trace map and a Fourier multiplier $T_\phi$, where $\phi$ is a
uniform limit of positive-definite functions factoring via finite
quotients. If the IRM is further ergodic, the description becomes much nicer.
In particular, this describes the IRMs of $\PSL_n(\bZ)$ for $n\geq3$
(see Corollary~\ref{cor:irm-pgl} for the ergodic case).

Not every higher rank lattice is mixed identity-free, and for these groups,
we need a different argument. We would like to classify
boomerang UCP maps associated with higher rank lattices as well, but our proof falls short of such a
classification. Boomerangs still play a crucial role. We use recurrence
to restrict the Fourier support of conditional expectations and combine
this with the Nevo--Stuck--Zimmer theorem (Theorem~\ref{thm:NSZ}) and
character rigidity (Theorem~\ref{thm:char}) to classify the IRAs.
\begin{theorem}\label{thm:ksplit}
Let $K$ be a number field, let $S$ be a finite set of places containing
all Archimedean places, and let $\bfG$ be a connected, absolutely
simple, adjoint, $K$-split algebraic group with
$\sum_{v\in S}\operatorname{rank}_{K_v}(\bfG)\geq2$.
Let $\Gamma<\bfG(K)$ be any subgroup commensurable with
$\bfG(\mathcal O_S)$. Then every ergodic IRA $\mu$ of $\Gamma$ is
of the form $\mu=L_*\nu$ for an ergodic IRS $\nu$ of $\Gamma$.
\end{theorem}
By ergodic decomposition, and since $L(\Sub(\Gamma))$ is closed,
Theorem~\ref{thm:ksplit} implies that $\Gamma$ is IRA rigid. The proof is different in nature from the proof for products of
acylindrically hyperbolic groups. Boomerangs still are useful as mentioned above, but we also use the classification of ergodic IRSs given by
the Nevo--Stuck--Zimmer
theorem~\cite{stuck1994stabilizers,NSZ:intermediate_factor}
\cite[Corollary~1.4]{DGGHL:char_rigid} and the classification of
extremal characters \cite{Bekka,PT:char_rigidity,BBHP,DGGHL:char_rigid}.
Together with the IRS classification, Theorem~\ref{thm:ksplit}
describes the ergodic IRAs explicitly.

We also spend some time discussing the structural properties of
boomerang subalgebras, in particular finite dimensional and amenable
ones. Many of the groups considered here are $C^*$-simple and admit
no nontrivial amenable normal subgroup~\cite{breuillard2017c}. It is
natural to ask when the same conclusion holds for amenable boomerang
subgroups, and more generally for amenable boomerang subalgebras.

We answer this question under some natural assumptions on the group
or its actions. In particular, we prove that every amenable
boomerang subalgebra is trivial when $\Gamma$ has property
$P_{\mathrm{nai}}$, admits a faithful extremely proximal action on
a compact Hausdorff space with more than two points, or admits an
action with property $P_{\mathrm{geo}}$ (see
Theorem~\ref{prop:amenable-boomerang}). Bekka, Cowling and de la
Harpe~\cite{BCD} introduced the properties $P_{\mathrm{nai}}$
and $P_{\mathrm{geo}}$. These conclusions strengthen the
corresponding results for invariant amenable subalgebras and
amenable IRAs in \cite[Theorems~A and~B]{amrutam2309amenable}.

\subsection*{Organization of the paper} In addition to this introduction, this paper has five other sections.
In Section~\ref{sec:prelim}, we establish the correspondence between
trace-preserving UCP maps and regular positive-definite functions on
$\Gamma\times\Gamma$ (Proposition~\ref{prop:ira-ucp-correspondence}).
We also introduce the notion of invariant random maps and define boomerang UCP maps and
boomerang positive-definite functions while we are here. In Section~\ref{sec:prep}, we study basic properties of boomerang
subalgebras, their associated positive-definite functions, and finite
dimensional and amenable boomerang subalgebras. In particular, we
prove Theorem~\ref{prop:amenable-boomerang} here. Section~\ref{sec:higherrank} is devoted to classifying the IRAs of
$K$-split higher rank lattices. It should not come as a surprise then
that we prove Theorem~\ref{thm:ksplit} in this section. In Section~\ref{sec:boomerang-hyperbolic}, we prove
Theorems~\ref{thm:hyperbolic_tf} and~\ref{thm:hyperbolic_tf_boom}, as well as
rigidity for invariant random maps. Finally, in Section~\ref{sec:ira-mif}, we prove
Theorems~\ref{thm:ira-mif-intro} and~\ref{thm:ira-boundaries-trees-intro}, and
obtain Corollaries~\ref{cor:jacobson} and~\ref{cor:embedding}.
We also apply these results to linear groups (Corollary~\ref{cor:linear-irm})
and describe the IRMs of IRM rigid groups with property $\mathrm{CR}_{\mathrm{fin}}$
(Theorem~\ref{prop:irm-char}), including the $\PGL_n$ lattices in
Corollary~\ref{cor:irm-pgl}.
\subsection*{Recent Developments}
While this draft was being readied, there were two recent developments
that deserve particular mention.

The first is joint work of Amrutam, Klisse and
Pagliero~\cite{AKP:inner-selfless}, in which the relative ISR
property was first established for all countable nontrivial mixed
identity-free groups. In particular, this includes groups admitting
topologically free minimal extremely proximal actions on compact
Hausdorff spaces with more than two points, and the Baumslag--Solitar
groups $\operatorname{BS}(m,n)$ with $|m|,|n|\geq2$ and $|m|\neq|n|$.
The Baumslag--Solitar result answers the corresponding part of a
question of Dudko and Jiang~\cite[Question~6.1]{dudko2024character}.
Their proof uses tracial ultrapowers, and that argument does not
directly extend to invariant random algebras.

The second is the recent work of Fima, Oppelmayer and
Tao~\cite{FOT2026}, who introduce boomerang $C^*$-subalgebras using
the Wijsman topology. Among other things, their Theorems~2.13 and~6.8
show that every boomerang subalgebra of $C_r^*(\Gamma)$ comes from
a boomerang subgroup when $\Gamma$ is a torsion-free non-amenable
hyperbolic group. Their proofs use techniques for reduced group
$C^*$-algebras.

\subsection*{AI Declaration}
The mathematical ideas, results, and proofs are due to
the authors. We were aided by ChatGPT (OpenAI) for the proof of Proposition \ref{prop:mat_conf}, as well as in grammatical corrections,
assistance with exposition and references, and feedback on some
of the authors' arguments. The authors take full responsibility
for the manuscript.

\subsection*{Acknowledgments}
This work began in 2024 when T.A. was supported by the European
Research Council (ERC) under the European Union’s Horizon Europe
research and innovation programme, Grant agreement No. 101078193
and Israel Science Foundation grant 1175/18. Much water has flowed
under the bridge since then. T.A. is currently supported by National
Science Centre, Poland Sonata, grant number 2025/59/D/ST1/03117.
T.A. thanks Mehrdad Kalantar for asking him a question about UCP maps in 2025.
YG was partially funded by ISF grant 3187/24 as well as by BSF Grant 2024317. 
\section{Boomerang objects}
\label{sec:prelim}
We begin by setting up the notation for later use. We then establish
the correspondence between trace-preserving UCP maps and regular
positive-definite functions on $\Gamma\times\Gamma$. We also introduce
invariant random maps and define boomerang UCP maps and boomerang
positive-definite functions in this section.

Throughout this paper, $\Gamma$ is a countable group, and $\tau_0$ is the canonical
trace on $L(\Gamma)$. For any $x\in L(\Gamma)$, we write
$\operatorname{supp}(x)=\{g\in\Gamma:\tau_0(x\lambda(g)^*)\neq0\}$. Also, we shall use $\|x\|_2=\tau(x^*x)^{1/2}$, the $\|\cdot\|_2$-norm associated with the canonical trace. We also write $\tau=\tau_0$ and use $\lambda(g)$ for the canonical group unitaries.
All subalgebras in this paper are unital von Neumann subalgebras.  
\subsection{Algebras}\label{subsec:algebras}
For a subalgebra $\cA\leq L(\Gamma)$, let
$\E_\cA\colon L(\Gamma)\to\cA$ be the unique $\tau$-preserving
conditional expectation \cite[Chapter~1]{brown2008cstar}. We equip $\SA(L(\Gamma))$ with the Effros--Mar\'echal topology. This topology is Polish~\cite{marechal1973topologie, FLMP:marechal}; it is not compact when $\Gamma$ is infinite~\cite[Theorem~A]{AFJ:uniformly-recurrent}.
In the finite tracial setting, $\cA_i\to\cA$ in the
Effros--Mar\'echal topology if and only if
$\|\E_{\cA_i}(x)-\E_\cA(x)\|_2\to0$ for every $x\in L(\Gamma)$;
see \cite[Corollary~2.12]{haagerup1998effros} and
\cite[Proposition~2.2]{amrutam2026confined}.
By \cite[Lemma~4.5 and Theorem~4.7]{FLMP:ucp}, this is also equivalent
to point ultraweak convergence of the trace-preserving conditional expectations.

The following fact is well known to experts. A proof can be found
in \cite[Lemma~2.3]{amrutam2026confined}.
\begin{lemma}\label{lem:conjugate-expectation}
Let $\cA\leq L(\Gamma)$ be a subalgebra and $t\in\Gamma$.
For every $x\in L(\Gamma)$,
$\E_{\lambda(t)\cA\lambda(t)^*}(x) =\lambda(t)\E_\cA\bigl(\lambda(t)^*x\lambda(t)\bigr)\lambda(t)^*$.
\end{lemma}

\subsection{Invariant random algebras and positive-definite functions}
\label{subsec:char}
A function $\phi\colon\Gamma\to\mathbb C$ is positive-definite if
$\sum_{i,j=1}^n c_i\overline{c_j}\phi(g_j^{-1}g_i)\geq0$
for every $g_1,\ldots,g_n\in\Gamma$ and $c_1,\ldots,c_n\in\mathbb C$.
It is normalized if $\phi(e)=1$. A normalized positive-definite
function which is conjugation invariant is called a character.
Let $\operatorname{PD}_1(\Gamma)$ denote the compact metrizable space of
normalized positive-definite functions on $\Gamma$, with the pointwise
topology and the action $(s\cdot\phi)(g)=\phi(s^{-1}gs)$.

We will use the following averaging lemma to prove the measurability
of the finite-quotient part of a positive-definite function.
\begin{lemma}\label{lem:averaging-operator}
Let $\Gamma$ be a countable group, let $\mu$ be a symmetric probability
measure on $\Gamma$ with full support, and let $\rho$ be a unitary
representation of $\Gamma$ on a Hilbert space $\mathcal K$. Let $Q$ be
the orthogonal projection onto the $\rho(\Gamma)$-invariant vectors.
Then the averaging operator $\rho(\mu)=\sum_{g\in\Gamma}\mu(g)\rho(g)$
satisfies $\rho(\mu)^k\to Q$ in the strong operator topology. In particular,
for every $\zeta\in\mathcal K$, if $\beta(g)=\langle\rho(g)\zeta,\zeta\rangle$,
then
\[
\lim_{k\to\infty}\sum_{g\in\Gamma}\mu^{*k}(g)\beta(g)
=\langle Q\zeta,\zeta\rangle=\|Q\zeta\|^2.
\]
\end{lemma}

\begin{proof}
Since $\mu$ is symmetric, the operator $\rho(\mu)$ is a self-adjoint
contraction. Since $\|\rho(\mu)-\mu(e)I\|\leq 1-\mu(e)$ and $\mu(e)>0$,
its spectrum is contained in $[2\mu(e)-1,1]$. Moreover, for every
$\eta\in\mathcal K$,
\[
\langle (I-\rho(\mu))\eta,\eta\rangle
=\tfrac12\sum_{g\in\Gamma}\mu(g)\|\rho(g)\eta-\eta\|^2.
\]
Since $\mu$ has full support, the $1$-eigenspace of $\rho(\mu)$ is
precisely the space of $\rho(\Gamma)$-invariant vectors. The spectral
theorem for bounded self-adjoint operators therefore gives
$\rho(\mu)^k\to Q$ strongly. Since $\rho(\mu)^k=\rho(\mu^{*k})$, we have
\[
\sum_{g\in\Gamma}\mu^{*k}(g)\beta(g)
=\langle\rho(\mu)^k\zeta,\zeta\rangle
\longrightarrow\langle Q\zeta,\zeta\rangle=\|Q\zeta\|^2.
\]
\end{proof}
Given a decreasing
sequence of finite-index normal subgroups $(N_j)$ of $\Gamma$, we can associate two maps at the level of $\operatorname{PD}_1(\Gamma)$.

By an abuse of notation, we write $Q$ for
the sequence of quotient maps $(\Gamma\to\Gamma/N_j)_j$. For
$\phi\in\operatorname{PD}_1(\Gamma)$, let $(\rho,\mathcal H,\xi)$ be its
GNS triple (see~\cite[Construction~1.B.5]{bekka-de-la-harpe2020}). Let
$P_j$ be the orthogonal projection onto the $\rho(N_j)$-invariant
vectors and let $P$ be the strong limit of the increasing sequence
$(P_j)$. Then $P$ is the orthogonal projection onto
$\overline{\bigcup_j P_j\mathcal H}$. For $g\in\Gamma$, set
$\phi_j(g)=\langle\rho(g)P_j\xi,\xi\rangle$ and
$\phi_\infty(g)=\langle\rho(g)P\xi,\xi\rangle$. Since each $N_j$ is
normal, $P_j$ commutes with $\rho(\Gamma)$, and so does $P$. Hence
$\phi_\infty(g)=\langle\rho(g)P\xi,P\xi\rangle$ is positive definite.
Moreover, since $P$ is an orthogonal projection, $P\xi$ is orthogonal
to $(1-P)\xi$. Thus $\phi_\infty(e)=\langle P\xi,\xi\rangle=\|P\xi\|^2$,
which belongs to $[0,1]$ since $\|\xi\|=1$.

We can now make the maps precise. Define
$a\colon\operatorname{PD}_1(\Gamma)\to[0,1]$ by $a(\phi)=\phi_\infty(e)$.
We also define the normalized finite-quotient part
$F_Q\colon\operatorname{PD}_1(\Gamma)\to\operatorname{PD}_1(\Gamma)$ by
\[
F_Q(\phi)=
\begin{cases}
\phi_\infty/a(\phi), & a(\phi)>0,\\
1_\Gamma, & a(\phi)=0,
\end{cases}
\]
where $1_\Gamma$ is the constant function with value $1$.

\begin{proposition}\label{prop:finite-quotient-part}
Let $\Gamma$ be a countable group and let $(N_j)$ be a decreasing
sequence of finite-index normal subgroups of $\Gamma$. Then the
following hold.
\begin{enumerate}
\item For every $\phi\in\operatorname{PD}_1(\Gamma)$, the function
$\phi-\phi_\infty$ is positive definite.
\item $F_Q$ is Borel for the pointwise topology and $\Gamma$-equivariant
for the conjugation action.
\item $a$ is Borel and $\Gamma$-invariant.
\item For every $\phi\in\operatorname{PD}_1(\Gamma)$, $F_Q(\phi)$
is a uniform limit of normalized positive-definite functions factoring
through finite quotients of $\Gamma$.
\end{enumerate}
\end{proposition}

\begin{proof}
Since $P$ commutes with $\rho(\Gamma)$, we have
$(\phi-\phi_\infty)(g)=\langle\rho(g)(1-P)\xi,(1-P)\xi\rangle$ for
every $g\in\Gamma$. This proves (1).

For (2) and (3), fix $j$ and choose a symmetric probability measure
$\mu_j$ on $N_j$ with full support. By Lemma~\ref{lem:averaging-operator}
applied to $\rho|_{N_j}$, the operators $\rho(\mu_j)^k$ converge
strongly to $P_j$. Since $\rho(\mu_j)^k=\rho(\mu_j^{*k})$, for each
$g\in\Gamma$ we have
\[
\begin{aligned}
\sum_{n\in N_j}\mu_j^{*k}(n)\phi(ng)
=\langle\rho(\mu_j)^k\rho(g)\xi,\xi\rangle\longrightarrow\langle P_j\rho(g)\xi,\xi\rangle=\phi_j(g).
\end{aligned}
\]
First, observe that the map $\phi\mapsto\phi_j(g)$ is Borel. Indeed, For fixed $k$ and $g$, the sum is continuous in $\phi$. Moreover, it is
a uniform limit of finite sums of point evaluations, since
$|\phi(ng)|\leq1$ and $\mu_j^{*k}$ is a probability measure. Thus $\phi\mapsto\phi_j(g)$ is Borel, and taking $j\to\infty$ shows that
$\phi\mapsto\phi_\infty(g)$ is Borel. Taking $g=e$ gives that $a$ is
Borel. On the Borel set $\{\phi:a(\phi)>0\}$, the coordinate
$F_Q(\phi)(g)$ is the quotient of the Borel functions
$\phi\mapsto\phi_\infty(g)$ and $a$. On its complement the coordinate
is constant. Thus each coordinate of $F_Q$ is Borel. Since $\Gamma$
is countable, the coordinate evaluations generate the Borel
$\sigma$-algebra of $\operatorname{PD}_1(\Gamma)$, so $F_Q$ is Borel. Let us now prove equivariance. For $s\in\Gamma$, a GNS triple for $s\cdot\phi$ is
$(\rho,\mathcal H,\rho(s)\xi)$. The projection $P$ depends only on the
representation and $(N_j)$, and it commutes with $\rho(s)$. Hence
\[
(s\cdot\phi)_\infty(g)
=\langle\rho(g)P\rho(s)\xi,\rho(s)\xi\rangle
=\langle\rho(s^{-1}gs)P\xi,\xi\rangle
=(s\cdot\phi_\infty)(g).
\]
Evaluating at $e$ gives $a(s\cdot\phi)=a(\phi)$. If $a(\phi)>0$,
dividing by this common value gives
$F_Q(s\cdot\phi)=s\cdot F_Q(\phi)$. If $a(\phi)=0$, the same identity
holds since $1_\Gamma$ is conjugation invariant. This proves (2) and (3).

For (4), since $P_j$ commutes with $\rho(\Gamma)$, we have
$\phi_j(g)=\langle\rho(g)P_j\xi,P_j\xi\rangle$. Thus $\phi_j$ is a
coefficient of the representation on $P_j\mathcal H$, which factors
through $\Gamma/N_j$. Moreover, since
$|\phi_j(g)-\phi_\infty(g)|\leq\|(P_j-P)\xi\|$ for every $g\in\Gamma$, uniform convergence to $\phi_\infty$ follows
(cf.~\cite[Theorem~4.C.5]{bekka-de-la-harpe2020}).
If $a(\phi)>0$, then $\phi_j(e)\to a(\phi)$, so the normalized
positive-definite functions $\phi_j/\phi_j(e)$ are defined for all
sufficiently large $j$ and converge uniformly to $F_Q(\phi)$.
If $a(\phi)=0$, then $F_Q(\phi)=1_\Gamma$ factors through every
$\Gamma/N_j$.
\end{proof}
Given a von Neumann subalgebra $\cA\leq L(\Gamma)$, we associate to it the
normalized positive-definite function (see \cite[Proposition~3.2]{jiang2024example}
and \cite[Remark~4.2(5)]{amrutam2026confined}) defined by
$$\phi_\cA(g)=\tau_0\bigl(\E_\cA(\lambda(g))\lambda(g)^*\bigr) =\|\E_\cA(\lambda(g))\|_2^2, ~g\in\Gamma.$$
The second equality follows because $\E_\cA$ is the orthogonal
projection onto $L^2(\cA,\tau_0)$. This function has been studied extensively
in \cite{jiang2024example,dudko2024character,amrutam2026confined} and has
been particularly useful in the study of the associated subalgebra,
especially when it is invariant.
If $\cA$ is $\Gamma$-invariant, then $\phi_\cA$ is a character.

Recall that an invariant random von Neumann subalgebra (IRA) is a $\Gamma$-invariant
Borel probability measure on $\SA(L(\Gamma))$. Given an IRA $\mu$,
the associated function $\phi_\mu$ defined by
$$\phi_\mu(g)=\int_{\SA(L(\Gamma))}\phi_\cA(g)\,d\mu(\cA), ~g\in\Gamma$$
is a character. Indeed, normalization and positive definiteness
are preserved by integration. For $s,g\in\Gamma$,
Lemma~\ref{lem:conjugate-expectation} gives
$\phi_\cA(sgs^{-1})=\phi_{\lambda(s^{-1})\cA\lambda(s)}(g)$.
Integrating this equality and using the invariance of $\mu$ gives
$\phi_\mu(sgs^{-1})=\phi_\mu(g)$.
\subsection{Positive-definite functions and UCP maps}\label{subsec:ira-ucp-correspondence}
Given a subalgebra $\mathcal{A}\le L(\Gamma)$, the associated conditional expectation $\E_\cA$ is a normal trace-preserving unital
completely positive (UCP) map (see \cite[Chapter~1]{brown2008cstar}).
The function $\phi_\cA$ defined above records the coefficient of $\lambda(g)$ in
$\E_\cA(\lambda(g))$ for each $g\in\Gamma$. For our purposes, we want to keep a track of all Fourier coefficients. This
leads to a correspondence between trace-preserving UCP maps on $L(\Gamma)$
and positive-definite functions on $\Gamma\times\Gamma$.

On $\operatorname{PD}_1(\Gamma\times\Gamma)$ we use the diagonal conjugation
action $(s\cdot\psi)(g,h)=\psi(s^{-1}gs,s^{-1}hs)$. We write $\UCP_{\tau_0}(L(\Gamma))$ for the space of normal
trace-preserving UCP maps on $L(\Gamma)$ equipped with the point ultraweak topology.
And the action on this space is given by
$(s\cdot T)(x)=\lambda(s)T(\lambda(s)^*x\lambda(s))\lambda(s)^*$.

For the correspondence with UCP maps we need a condition on the two
marginals. This is natural, since for the coefficients $\tau_0(T(\lambda(g))\lambda(h)^*)$,
unitality determines the marginal at $g=e$, and trace preservation
determines the marginal at $h=e$.
\begin{definition}[Regular positive-definite function]
A normalized positive-definite function $\psi$ on $\Gamma\times\Gamma$
is called regular if both marginals are the regular character, that is,
$\psi(g,e)=\psi(e,g)=\delta_{g,e}$ for every $g\in\Gamma$.
\end{definition}
We denote the set of such functions by $\operatorname{PD}_1^{\mathrm R}(\Gamma\times\Gamma)$,
with the topology and action inherited from $\operatorname{PD}_1(\Gamma\times\Gamma)$.

The following proposition is the bridge between UCP maps and
regular positive-definite functions on $\Gamma\times\Gamma$. It allows us to navigate the two worlds seamlessly by carrying information from one to the other. We follow
the bimodule construction in \cite[Section~13.1.2 and Lemma~13.1.8]{AP:II1}.

\begin{proposition}\label{prop:ira-ucp-correspondence}
Let $\Gamma$ be a countable group and let $\psi\in\operatorname{PD}_1^{\mathrm R}(\Gamma\times\Gamma)$.
There is a unique $T_\psi\in\UCP_{\tau_0}(L(\Gamma))$ such that
$\tau_0(T_\psi(\lambda(g))\lambda(h)^*)=\psi(g,h)$
for every $g,h\in\Gamma$.
Conversely, every $T\in\UCP_{\tau_0}(L(\Gamma))$ arises in this way.
The correspondence $\mathfrak{T}_\Gamma\colon\psi\mapsto T_\psi$ is an equivariant
homeomorphism from $\operatorname{PD}_1^{\mathrm R}(\Gamma\times\Gamma)$ onto
$\UCP_{\tau_0}(L(\Gamma))$.
\end{proposition}

\begin{proof}
Let $\psi\in\operatorname{PD}_1^{\mathrm R}(\Gamma\times\Gamma)$. We equip $\mathbb C[\Gamma\times\Gamma]$ with the positive semidefinite form
$\langle[g,h],[k,l]\rangle=\psi(k^{-1}g,l^{-1}h)$.
After the usual completion procedure, we obtain the associated Hilbert space $\mathcal H$. We define a unitary representation of $\Gamma\times\Gamma$ by
$\pi(s,t)[g,h]=[sg,th]$ for $(s,t)\in\Gamma\times\Gamma$.
For $\xi=[e,e]$, one has $\psi(g,h)=\langle\pi(g,h)\xi,\xi\rangle$. Moreover,  the set $\{\pi(g,h)\xi:(g,h)\in\Gamma\times\Gamma\}$ spans a dense subspace of $\mathcal H$.
The two coordinate representations commute.
Let $A=\{\pi(g,e):g\in\Gamma\}''$ and $\varphi(a)=\langle a\xi,\xi\rangle$ for $a\in A$.
Since $\varphi(\pi(g,e))=\delta_{g,e}$, the state $\varphi$ is tracial on the group algebra, hence on $A$ (by ultraweak continuity).

We claim that $\varphi$ is faithful. Indeed, if $\varphi(a^*a)=0$, then
$$\|a\pi(g,e)\xi\|^2 =\varphi(\pi(g,e)^*a^*a\pi(g,e)) =\varphi(a^*a)=0.$$
Since $a$ commutes with $\pi(e,h)$, this gives $a\pi(g,h)\xi=0$ for all $g,h\in\Gamma$, and therefore $a=0$.
Thus $(A,\varphi)$ is a finite tracial von Neumann algebra.
Also, the map $\delta_g\mapsto\widehat{\pi(g,e)}$ extends to a unitary $\ell^2(\Gamma)\to L^2(A,\varphi)$ intertwining the left regular representation with left multiplication by $\pi(g,e)$.
It induces a unital $*$-isomorphism $\pi_1:L(\Gamma)\to A\subseteq B(\mathcal H)$ satisfying $\pi_1(\lambda(g))=\pi(g,e)$ and $\varphi\circ\pi_1=\tau_0$. Note that  $\pi_1$ is normal. Indeed, if $0\leq x_i\uparrow x$, then
$\varphi\left(\pi_1(x)-\sup_i\pi_1(x_i)\right) =\tau_0(x)-\sup_i\tau_0(x_i)=0$,
so that (using the faithfulness of $\varphi$) $\pi_1(x)=\sup_i\pi_1(x_i)$.

Define $V:\ell^2(\Gamma)\to\mathcal H$ by $V\delta_g=\pi(e,g^{-1})\xi$.
Since
$\langle V\delta_g,V\delta_h\rangle =\psi(e,hg^{-1})=\delta_{g,h}$,
it follows that $V$ is an isometry.
Define, initially with values in $B(\ell^2(\Gamma))$, $T_\psi(x)=V^*\pi_1(x)V$.
This map is normal and UCP, since $\pi_1$ is normal and $V$ is an isometry.
Let $R_t\delta_g=\delta_{gt^{-1}}$ be the right regular unitaries.
For $s,t,g,x\in\Gamma$, we have
\[
\begin{aligned}
\langle T_\psi(\lambda(s))R_t\delta_g,\delta_x\rangle
&=\langle V^*\pi(s,tg^{-1})\xi,\delta_x\rangle\\
&=\langle\pi(s,tg^{-1})\xi,\pi(e,x^{-1})\xi\rangle\\
&=\psi(s,xtg^{-1}),\\
\langle R_tT_\psi(\lambda(s))\delta_g,\delta_x\rangle
&=\langle V^*\pi(s,g^{-1})\xi,R_t^*\delta_x\rangle\\
&=\langle\pi(s,g^{-1})\xi,\pi(e,t^{-1}x^{-1})\xi\rangle\\
&=\psi(s,xtg^{-1}).
\end{aligned}
\]
Thus $T_\psi(\lambda(s))R_t=R_tT_\psi(\lambda(s))$.
Normality of $T_\psi$ and ultraweak density of $\mathbb C\Gamma$ force $T_\psi(a)R_t=R_tT_\psi(a)$ for every $a\in L(\Gamma)$.
Hence $T_\psi(a)\in R(\Gamma)'=L(\Gamma)$, where $R(\Gamma)=\{R_t:t\in\Gamma\}''$.
Moreover,
$\tau_0(T_\psi(x)) =\langle\pi_1(x)\xi,\xi\rangle =\varphi(\pi_1(x))=\tau_0(x)$.
Finally,
$\tau_0(T_\psi(\lambda(g))\lambda(h)^*) =\langle\pi(g,e)\xi,\pi(e,h^{-1})\xi\rangle =\psi(g,h)$,
which proves the claim. 

Conversely, let $T\in\UCP_{\tau_0}(L(\Gamma))$. We associate with it is the map
$\psi_T\colon\Gamma\times\Gamma\to\mathbb C$ defined by
$\psi_T(g,h)=\tau_0(T(\lambda(g))\lambda(h)^*)$.
For $g_i,h_i\in\Gamma$ and $c_i\in\mathbb C$, complete positivity gives
$$\sum_{i,j}\overline{c_i}c_j\psi_T(g_i^{-1}g_j,h_i^{-1}h_j) =\tau_0\left(\sum_{i,j}\overline{c_i}c_j \lambda(h_i)T(\lambda(g_i)^*\lambda(g_j))\lambda(h_j)^*\right) \geq0.$$
Since $\psi$ is unital and trace preserving, it follows that $\psi_T(g,e)=\psi_T(e,g)=\delta_{g,e}$.
Thus $\psi_T\in\operatorname{PD}_1^{\mathrm R}(\Gamma\times\Gamma)$,
and by uniqueness, we obtain that $T=T_{\psi_T}$. It follows as a consequence that $\mathfrak{T}_\Gamma^{-1}(T)=\psi_T$. 

For fixed
$g,h\in\Gamma$, the functional $x\mapsto\tau_0(x\lambda(h)^*)$ is normal,
so $T\mapsto\psi_T(g,h)$ is continuous in the point ultraweak topology.
Thus $\mathfrak{T}_\Gamma^{-1}$ is continuous for the pointwise topology
on $\operatorname{PD}_1^{\mathrm R}(\Gamma\times\Gamma)$.
Since $\UCP_{\tau_0}(L(\Gamma))$ is compact
\cite[Remark~4.6]{FLMP:ucp} and $\operatorname{PD}_1^{\mathrm R}(\Gamma\times\Gamma)$ is Hausdorff, it follows that
$\mathfrak{T}_\Gamma^{-1}$ is a homeomorphism.
Finally, 
$$\psi_{s\cdot T}(g,h)=\tau_0(T(\lambda(s^{-1}gs))\lambda(s^{-1}hs)^*)=(s\cdot\psi_T)(g,h).$$
Hence equivariance follows.
\end{proof}
For a large class of groups, we will show that the positive-definite
function associated with a boomerang UCP map vanishes off the diagonal.
These are the Fourier multipliers from the introduction.
\begin{definition}[Fourier multiplier]
A map $T\in\UCP_{\tau_0}(L(\Gamma))$ is called a Fourier multiplier if
$T(\lambda(g))\in\mathbb C\lambda(g)$ for every $g\in\Gamma$.
\end{definition}
We now relate Fourier multipliers to normalized positive-definite
functions on $\Gamma$. This has been studied before by Alekseev-Brugger~\cite{AB:irpdf}.
\begin{lemma}\label{lem:fourier-multiplier-map}
Let $\Gamma$ be a countable group. The map
$\mathfrak{M}_\Gamma\colon\operatorname{PD}_1(\Gamma)\to
\UCP_{\tau_0}(L(\Gamma))$ given by
$\mathfrak{M}_\Gamma(\phi)=T_\phi$, where
$T_\phi(\lambda(g))=\phi(g)\lambda(g)$ for every $g\in\Gamma$, is a
$\Gamma$-equivariant homeomorphism onto its image.
\end{lemma}
\begin{proof}
Let $\phi\in\operatorname{PD}_1(\Gamma)$ be given. We define a function $\psi$ on $\Gamma\times\Gamma$ by
$\psi(g,h)=\delta_{g,h}\phi(g)$. It follows from
\cite[Proposition~1.F.9]{bekka-de-la-harpe2020} that $\psi$ is positive definite.
Since $\phi(e)=1$, we have
$\psi\in\operatorname{PD}_1^{\mathrm R}(\Gamma\times\Gamma)$.
It also follows from
Proposition~\ref{prop:ira-ucp-correspondence} that
$T_\psi(\lambda(g))=\phi(g)\lambda(g)$, so that $T_\phi=T_\psi$.  The map $\mathfrak p:\operatorname{PD}_1(\Gamma)\to \operatorname{PD}_1(\Gamma\times\Gamma),~\phi\mapsto\psi$ is continuous in the pointwise
topology. It is obviously equivariant. Since $\mathfrak{M}_\Gamma=\mathfrak{T}_\Gamma\circ\mathfrak{p}$, the claim follows.

Conversely, if $T$ is a Fourier multiplier, then $\psi_T(g,h)=0$
for $g\ne h$. Moreover, $\phi(g)=\psi_T(g,g)$ belongs to $\operatorname{PD}_1(\Gamma)$, and
$\mathfrak{M}_\Gamma(\phi)=T$. This also gives injectivity.
Since $\operatorname{PD}_1(\Gamma)$ is compact and Hausdorff, $\mathfrak{M}_\Gamma$ is a homeomorphism onto
the Fourier multipliers.
\end{proof}

Let us now discuss the relationship of IRAs with these maps. Via the map $M\mapsto\E_M$, we can view an IRA as a $\Gamma$-invariant
probability measure on $\operatorname{CE}_{\tau_0}(L(\Gamma))
=\{\E_M:M\leq L(\Gamma)\}$. This leads us to consider invariant probability
measures on the entire space $\UCP_{\tau_0}(L(\Gamma))$.
\begin{definition}[Invariant random map]
\label{def:irm}
An invariant random map (IRM for short) of $\Gamma$ is a
$\Gamma$-invariant Borel probability measure on
$\UCP_{\tau_0}(L(\Gamma))$ for the conjugation action.
We denote the set of these measures by $\operatorname{IRM}(\Gamma)$.
\end{definition}
The space $\UCP_{\tau_0}(L(\Gamma))$ carries the information about all the relevant objects we have discussed so far. 
\begin{lemma}\label{lem:diagram-inclusions}
Equip $\Sub(\Gamma)$, $\SA(L(\Gamma))$, and $\UCP_{\tau_0}(L(\Gamma))$
with the Chabauty, Effros--Mar\'echal, and point ultraweak topologies,
respectively. Then $H\mapsto L(H)$ and $M\mapsto\E_M$ are
$\Gamma$-equivariant topological embeddings.
\end{lemma}
\begin{proof}
By \cite[Proposition~4.1]{amrutam2309amenable}, the map $H\mapsto L(H)$
is continuous and equivariant. It is injective, and since $\Sub(\Gamma)$
is compact and $\SA(L(\Gamma))$ is Hausdorff, it is a topological
embedding.

The map $M\mapsto\E_M$ is injective since $\E_M$ has range $M$.
By \cite[Lemma~4.5 and Theorem~4.7]{FLMP:ucp}, convergence in the
Effros--Mar\'echal topology is equivalent to point ultraweak convergence
of the trace-preserving conditional expectations. Thus $M\mapsto\E_M$
is a topological embedding. Equivariance follows from
Lemma~\ref{lem:conjugate-expectation}.
\end{proof}

In the diagram below we identify $\Sub(\Gamma)$ with
$\{\E_{L(H)}:H\in\Sub(\Gamma)\}$. By Lemma~\ref{lem:diagram-inclusions},
the inclusions of the underlying spaces induce the corresponding inclusions
of invariant probability measures by pushforward. We write $\operatorname{IRA}(\Gamma)$ for the
set of IRAs of $\Gamma$.
\begin{center}
\begin{tabular}{@{}lccccc@{}}
Invariant measures
& $\IRS(\Gamma)$ & $\hookrightarrow$
& $\operatorname{IRA}(\Gamma)$ & $\hookrightarrow$
& $\operatorname{IRM}(\Gamma)$\\
\noalign{\vskip 2pt}
\hline
\noalign{\vskip 2pt}
Underlying spaces
& $\Sub(\Gamma)$ & $\subseteq$
& $\operatorname{CE}_{\tau_0}(L(\Gamma))$ & $\subseteq$
& $\UCP_{\tau_0}(L(\Gamma))$
\end{tabular}
\end{center}

\begin{definition}[IRM rigidity]
We call $\Gamma$ IRM rigid if every IRM of $\Gamma$ is supported on
Fourier multipliers.
\end{definition}
An IRM is a probability measure. To study such measures through individual
maps, we use the notion of boomerang points (as mentioned in the introduction) introduced by Glasner and Lederle
\cite[Definition~2.1]{glasner2024boomerang}. Poincar\'e recurrence ensures
that every IRM gives full measure to maps that are recurrent under
conjugation by each group element separately.

\begin{definition}[Boomerang UCP map]\label{def:boomerang-ucp}
Let $\Gamma$ be a countable group. A map $T\in\UCP_{\tau_0}(L(\Gamma))$ is called
boomerang
if for every $s\in\Gamma$ there is a strictly increasing sequence
$(n_i)\subseteq\bN$ such that
$\operatorname{Ad}(\lambda(s^{n_i}))\circ T\circ\operatorname{Ad}(\lambda(s^{-n_i}))\to T$ in the point ultraweak topology.
\end{definition}

We also consider recurrence for the conjugation action on $\operatorname{PD}_1(\Gamma)$.
\begin{definition}\label{def:boomerang-function}
A normalized positive-definite function $\phi$ on $\Gamma$ is called boomerang
if for every $s\in\Gamma$ there is a strictly increasing sequence
$(n_i)\subseteq\bN$ such that $\phi(s^{n_i}gs^{-n_i})\to\phi(g)$ for every
$g\in\Gamma$. 
\end{definition}
By an abuse of name calling, we will call these functions Boomerang characters.

We finish this subsection with two elementary observations. First, we observe that being boomerang leads to certain restrictions on the Fourier support of a ucp map.
For $s\in\Gamma$, let
$R(s):=\bigcap_{t\in C_\Gamma(s)}\ \bigcup_{k\geq1}C_\Gamma(t^k)$.
\begin{proposition}\label{prop:support}
Let $\Gamma$ be a countable discrete group and let $T\in\UCP_{\tau_0}(L(\Gamma))$ be boomerang. Then $\supp(T(\lambda(s)))\subseteq R(s)$ for every $s\neq e$.
\end{proposition}

\begin{proof}
Let us write $T(\lambda(s))=\sum_{h\in\Gamma}c_h\lambda(h)$ in terms of its Fourier
expansion. Fix $g\notin R(s)$. By the definition of $R(s)$, there is
$t\in C_\Gamma(s)$ such that $t^kgt^{-k}\neq g$ for every $k\geq1$.
Since $T$ is boomerang, we can find a strictly increasing sequence $(n_i)$
such that $t^{n_i}\cdot T\to T$ in the point ultraweak topology.
Since $t$ commutes with $s$, the Fourier coefficient of
$(t^{n_i}\cdot T)(\lambda(s))$ at $g$ is $c_{t^{-n_i}gt^{n_i}}$.
It follows from the point ultraweak convergence that  $c_{t^{-n_i}gt^{n_i}}\to c_g$.
The elements $t^{-n_i}gt^{n_i}$ are pairwise distinct; indeed, equality
for $i<j$ would give $t^{n_j-n_i}gt^{-(n_j-n_i)}=g$.
Since $(c_h)_h\in\ell^2(\Gamma)$, we have $c_{t^{-n_i}gt^{n_i}}\to0$.
Hence $c_g=0$. The claim follows.
\end{proof}
Since $s\in C_\Gamma(s)$, taking $t=s$ in the intersection defining $R(s)$ gives
$\supp(T(\lambda(s)))\subseteq R(s)\subseteq\bigcup_{k\geq1}C_\Gamma(s^k)$. This will come in handy when we deal with acylindrically hyperbolic groups.

It is also important for us to know if $\supp(T(\lambda(s)))$ remains non-empty under infinitely many conjugates. For example, if $T=\E_{L(\Delta)}$, the condition $T(\lambda(t))\ne0$ means that
$t\in\Delta$ and hence all of its conjugates under $\Delta$ remain in the support. The next lemma is the analogue of this principle (compare it with \cite[Lemma~2.14]{glasner2024boomerang}).
\begin{lemma}\label{notzerounderconditionalexpec}
Let $T\in\UCP_{\tau_0}(L(\Gamma))$ be boomerang and suppose that
$T(\lambda(t))\neq0$. For every $s\in\Gamma$, there are infinitely many
$n\geq1$ such that $T(\lambda(s^{-n}ts^n))\neq0$.
\end{lemma}
\begin{proof}
Since $T$ is boomerang, choose a strictly increasing sequence
$(n_i)\subseteq\mathbb N$ such that $s^{n_i}\cdot T\to T$ point ultraweakly.
Choose $h\in\Gamma$ such that $\tau_0(T(\lambda(t))\lambda(h)^*)\neq0$.
Then $\tau_0((s^{n_i}\cdot T)(\lambda(t))\lambda(h)^*)\to
\tau_0(T(\lambda(t))\lambda(h)^*)$, so
$(s^{n_i}\cdot T)(\lambda(t))\neq0$ for all sufficiently large $i$.
Since $(s^{n_i}\cdot T)(\lambda(t))=
\lambda(s^{n_i})T(\lambda(s^{-n_i}ts^{n_i}))\lambda(s^{-n_i})$,
it follows that $T(\lambda(s^{-n_i}ts^{n_i}))\neq0$ for all such $i$.
\end{proof}

\subsection{Thickness and vanishing}\label{subsec:thickness-vanishing} We recall the thickness and vanishing principle from \cite{amrutam-relative-isr} (which itself was motivated by \cite{dudko2024character}).
Let $\phi\colon\Gamma\to\mathbb C$ be a normalized positive-definite
function. A set $K\subseteq\Gamma$ is \emph{$\phi$-thick} if, for every
$g\in\Gamma\setminus\{e\}$, there is a sequence $(h_i)\subseteq\Gamma$
such that
$h_jg^{-1}h_j^{-1}h_igh_i^{-1}\in K,~i\neq j$, and $\phi(h_igh_i^{-1})\longrightarrow\phi(g)$ as $i\to\infty$.

We will need the following pointwise GNS vanishing argument of \cite[Lemma~3.6]{amrutam-relative-isr}.
We include the proof for completeness.
\begin{lemma}\label{lem:thickness-vanishing}
Let $\phi$ be a normalized positive-definite function on $\Gamma$,
and fix $g\in\Gamma\setminus\{e\}$. Suppose that there is a sequence
$(h_i)\subseteq\Gamma$ such that $\phi(h_igh_i^{-1})\to\phi(g)$ and
$\phi(h_jg^{-1}h_j^{-1}h_igh_i^{-1})=0$ for all $i\ne j$.
Then $\phi(g)=0$.
\end{lemma}
\begin{proof}
Let $(\pi,\mathcal H,\xi)$ be the GNS representation of $\phi$,
so that $\|\xi\|=1$ and
$\phi(t)=\langle\pi(t)\xi,\xi\rangle$ for every $t\in\Gamma$.
Let $v_i=\pi(h_igh_i^{-1})\xi$.
Since $\pi(h_igh_i^{-1})$ is unitary, $\|v_i\|=1$. If $i\neq j$,
then
\begin{align*}
 \langle v_i,v_j\rangle
 &=\langle\pi(h_igh_i^{-1})\xi,
                \pi(h_jgh_j^{-1})\xi\rangle\\
 &=\langle\pi(h_jg^{-1}h_j^{-1}h_igh_i^{-1})\xi,\xi\rangle\\
 &=\phi(h_jg^{-1}h_j^{-1}h_igh_i^{-1})=0,
\end{align*}
by hypothesis. Thus $(v_i)$ is an orthonormal sequence. For every
$r\geq1$, Bessel's inequality now gives
$\sum_{i=1}^r|\phi(h_igh_i^{-1})|^2
=\sum_{i=1}^r|\langle v_i,\xi\rangle|^2\leq\|\xi\|^2=1$.
If $\phi(g)\neq0$, choose $i_0$ such that
$|\phi(h_igh_i^{-1})|\geq|\phi(g)|/2$ for every $i\geq i_0$.
For $r\geq i_0$, the preceding bound would imply
$(r-i_0+1)|\phi(g)|^2/4\leq1$, which is impossible as
$r\to\infty$. Therefore $\phi(g)=0$.
\end{proof}
\subsection{\texorpdfstring{$K$}{}-Split higher rank lattices} \label{ksplit}
We will be interested in the case where $\Gamma$ is a split, higher rank $S$-arithmetic group. To set the notation let $K$ be a number field with ring of integers $\cO$ and $S$ a finite set of places of $K$ containing all of the Archimedean ones. Let $\bfG$ be a connected  absolutely simple adjoint algebraic group that is defined and split over $K$. Assume that $\bfG$ is of higher rank in the sense that 
$\sum_{v \in S} \rk_{K_v}(\bfG) \ge 2$.
We fix a faithful $K$-representation $\bfG \rightarrow \GL_n$ and hence identify $\bfG$ with its image under this representation. Let $\Gamma < \bfG(K)$ be any subgroup that is commensurable to $\bfG(\cO_S) = \bfG(K) \cap \GL_n(\cO_S)$. 

To fix notation, let $\bfB$ be a Borel subgroup, $\bfT< \bfB$ a maximal split torus and $\bfU$ the unipotent radical of $\bfB$. We denote by $\Phi$ the corresponding choice of root system, by $\Psi \subset \Phi$ the positive basis corresponding to the chosen Borel subgroup and by $\theta \in \Phi$ the highest root with respect to this basis. Choose root parametrizations $x_\alpha\colon\bfG_a\to\bfU_\alpha$
so that $x_\theta$ and $x_{-\theta}$ extend to the homomorphism
$\iota: \SL_2 \rightarrow \bfG, \qquad \iota \begin{pmatrix}1&t\\0&1\end{pmatrix} = x_{\theta}(t), \qquad \iota \begin{pmatrix}1&0\\t&1\end{pmatrix} = x_{-\theta}(t)$.
We denote the image of $\iota$ by $\bfG_\theta$.
The kernel of $\iota$ is contained in $\{\pm I\}$, so $\bfG_\theta$ is isomorphic to $\SL_2$ or $\PGL_2$.
A main feature which considerably simplifies our proof is that we often work inside this rank one group, rather than in the whole group. We denote
$\Gamma_{\theta} = \Gamma \cap \bfG_{\theta}(K) \qquad \Upsilon_{\theta} = \Gamma \cap \bfU_\theta(K)$.

For better clarity, we will, throughout, exemplify our main proof by following this special case. Here $K = \bQ$, $\cO = \bZ$, $S = \{\infty\}$, $\bfG = \PGL_n$. The higher rank assumption here means that $n \ge 3$. The Borel group $\bfB$ is the image of the upper triangular matrices, $\bfT$ is the image of the diagonal matrices, and $\bfU$ is the image of the unipotent upper triangular matrices. The roots are $\alpha_{i,j}(\overline{\diag(t_1,\ldots,t_n)})=t_i/t_j$ for $i\neq j$. The maps $x_{i,j}(t)$ give the elementary unipotent matrices with $1$ along the diagonal and $t$ in the $(i,j)$-entry. The highest root here is $\theta=\alpha_{1,n}$ and $-\theta = \alpha_{n,1}$. Here $\Gamma$ is any group commensurable to the arithmetic lattice $\PSL_n(\bZ)$. 

When referring to this example we use the overline notation $\overline{A}$ to denote the image of the matrix $A \in \SL_n(\bZ)$ modulo the center of $\SL_n(\bZ)$. Similarly for a nonzero vector $x \in \bR^n \setminus \{0\}$ we denote by $\overline{x}$ its image in the projective space $\bP \bR^n$.

Our proof relies heavily on strong rigidity results for these groups, which give a complete classification of their ergodic IRSs and extremal characters.
We now recall these results.
\begin{theorem}[Nevo--Stuck--Zimmer \cite{stuck1994stabilizers,NSZ:intermediate_factor}; \cite{DGGHL:char_rigid}, Corollary~1.4 and the paragraph following it] \label{thm:NSZ}
Let $\Gamma$ be as in Subsection \ref{ksplit}. Then every ergodic IRS on $\Sub(\Gamma)$ is either $\nu=\delta_{\{e\}}$ or $\nu=\nu_\Sigma$ for some finite-index subgroup $\Sigma<\Gamma$, where
$\nu_\Sigma=\frac{1}{[\Gamma:N_\Gamma(\Sigma)]}\sum_{gN_\Gamma(\Sigma)\in\Gamma/N_\Gamma(\Sigma)}\delta_{g\Sigma g^{-1}}$.
\end{theorem}
We will use the following form of character rigidity.
\begin{definition}\label{def:cr-fin}
A countable group $\Gamma$ has property $\mathrm{CR}_{\mathrm{fin}}$ if
the following hold.
\begin{enumerate}
\item Up to unitary equivalence, $\Gamma$ has at most countably many
irreducible finite-dimensional unitary representations.
\item Every irreducible finite-dimensional unitary representation of
$\Gamma$ factors via a finite quotient of $\Gamma$.
\item Every extremal character of $\Gamma$ is either $\delta_e$ or the
normalized character of an irreducible finite-dimensional unitary
representation of $\Gamma$.
\end{enumerate}
\end{definition}
The following character rigidity theorem was first proved for $\PSL_n(\bZ)$
by Bekka~\cite{Bekka} and later generalized in
\cite{PT:char_rigidity,BBHP,DGGHL:char_rigid}.
\begin{theorem}[{\cite[Theorem~1.3 and Remark~6.6]{DGGHL:char_rigid}, \cite[Proposition~7.1]{BBHP}}]
\label{thm:char}
Every group $\Gamma$ as in Subsection~\ref{ksplit} has property
$\mathrm{CR}_{\mathrm{fin}}$.
\end{theorem}
For the next proposition, $\Gamma$ is a $K$-split higher rank lattice
as in Subsection~\ref{ksplit}. It relates the vanishing of the
conditional expectations to the associated character. It essentially says that if an IRA is supported on non-trivial subalgebras, then the associated conditional expectation cannot vanish on a group element and all its powers. 
\begin{proposition}\label{prop:onedim}
Let $\mu$ be an IRA on $\SA(L(\Gamma))$ and let $s\in\Gamma$ have infinite order. Assume that for $\mu$-a.e.\ $M$ one has $\E_M(\lambda(s^k))=0$ for all $k\neq0$.  Then $\mu=\delta_{\bC}$, i.e.\ $M=\bC1$ for $\mu$-a.e.\ $M$.
\end{proposition}

\begin{proof}
For $g\in\Gamma$ set
$\varphi(g)=\int_{\SA(L(\Gamma))}\|\E_M(\lambda(g))\|_2^2\,d\mu(M)$.
Then $\varphi$ is a character of $\Gamma$ (see Subsection~\ref{subsec:char}, also see \cite[Lemma 2.2]{JL}).
By Theorem~\ref{thm:char} (since characters are traces on $C^*(\Gamma)$, they form a Choquet simplex, whose extreme points are countable) we may write $\varphi=c_0\delta_e+\sum_{i\geq1}c_i\chi_i$ with $c_i\geq0$, $\sum_{i\geq0}c_i=1$, where $\{\delta_e,\chi_i:i\geq1\}$ lists the countably many extremal characters of $\Gamma$; in particular each $\chi_i$ arises from an irreducible finite-dimensional representation $\pi_i$ of a finite quotient of $\Gamma$.

Let $H=\langle s\rangle\cong\Z$.  By the Herglotz--Bochner theorem, $\varphi|_H(s^m)=\int_\mathbb{T} z^m\,d\sigma(z)$ for a probability measure $\sigma$ on $\mathbb{T}$.  Since $\E_M(\lambda(s^k))=0$ for all $k\neq0$ and $\mu$-a.e.\ $M$, we get $\varphi(s^k)=0$ for all $k\neq0$, so all nonzero Fourier coefficients of $\sigma$ vanish and $\sigma=m_\mathbb{T}$, the Haar measure.  Likewise $\delta_e|_H$ has Herglotz measure $m_\mathbb{T}$, and writing $\chi_i(s^m)=\int_\mathbb{T} z^m\,d\sigma_i(z)$, uniqueness of Herglotz measures gives
$m_\mathbb{T}=c_0\,m_\mathbb{T}+\sum_{i\geq1}c_i\,\sigma_i$ on $\mathbb{T}$.
Fix $i\geq1$ and let $\bar s$ be the image of $s$ in the finite quotient defining $\pi_i$. The unitary $\pi_i(\bar s)$ has finite order, hence is diagonalizable with eigenvalues that are roots of unity, say $\omega_1,\dots,\omega_{d_i}$; then $\chi_i(s^m)=\frac1{d_i}\sum_j\omega_j^m$, so $\sigma_i=\frac1{d_i}\sum_j\delta_{\omega_j}$ is atomic, supported on roots of unity.  If $c_0<1$, then $(1-c_0)m_\mathbb{T}=\sum_ic_i\sigma_i$ would be a nonzero atomless measure equal to a countable sum of atomic measures, which is impossible.  Hence $c_0=1$ and $\varphi=\delta_e$, i.e.\ for every $g\neq e$, $\E_M(\lambda(g))=0$ for $\mu$-a.e.\ $M$; discarding countably many null sets, for $\mu$-a.e.\ $M$ we have $\E_M(\lambda(g))=0$ for all $g\neq e$.  For such $M$ and any $x\in M$, the Fourier expansion gives $x=\E_M(x)=\tau(x)1$, so $M=\bC1$.  Hence $\mu=\delta_\bC$.
\end{proof}
We would like to remark that the above proposition is a higher rank phenomenon. For example, if we consider $\Gamma=\mathbb F_2=\langle a,b\rangle$, and  $N=\ker(\Gamma\to\mathbb Z)$, where $a\mapsto1$ and $b\mapsto0$. Then $N$ is nontrivial and normal, but $a^k\notin N$ for every $k\neq0$. Consequently, $\mathbb E_{L(N)}(\lambda(a^k))=0$ for every $k\neq0$, while $\mu=\delta_{L(N)}$ is an IRA different from $\delta_{\mathbb C1}$.
\subsection{Acylindrically hyperbolic groups}
\label{subsec:acylindrically-hyperbolic}
We recall the terminology used in
\cite[Subsection~2.1]{amrutam-relative-isr} (which was taken from
\cite{osin2016acylindrically,DGO}).
An isometric action $\Gamma\curvearrowright(S,d)$ is
\emph{acylindrical} if, for every $\epsilon>0$, there are $R,N>0$
such that
$\bigl|\{g\in\Gamma:d(x,gx)\leq\epsilon, \ d(y,gy)\leq\epsilon\}\bigr|\leq N$ whenever $d(x,y)\geq R$.
For an action on a hyperbolic space, $s\in\Gamma$ is
\emph{loxodromic} if $n\mapsto s^nx$ is a quasi-isometric embedding
of $\mathbb Z$ into $S$ for some $x\in S$.
Such an element has two fixed points $x_s^+,x_s^-\in\partial S$.
Two loxodromic elements are \emph{independent} if their fixed-point
pairs are disjoint. The action is \emph{non-elementary} if it
contains two independent loxodromic elements. A group is
\emph{acylindrically hyperbolic} if it admits a non-elementary
acylindrical action on a hyperbolic space.

An element $s$ satisfies \emph{weak proper discontinuity} (WPD) if, for every
$x\in S$ and $\epsilon>0$, there is $m\geq1$ such that
$\{g\in\Gamma:d(x,gx)\leq\epsilon, \ d(s^mx,gs^mx)\leq\epsilon\}$ is finite.
Every loxodromic element of an acylindrical action satisfies WPD.
An element is \emph{generalized loxodromic} if it is loxodromic
for some acylindrical action of $\Gamma$ on a hyperbolic space.
Equivalently, it is loxodromic WPD for some action on a hyperbolic
space \cite[Theorem~1.4]{osin2016acylindrically}.
We denote the set of these elements by $K_{\mathrm{lox}}$.

For $s\in K_{\mathrm{lox}}$, there is a unique maximal virtually
cyclic subgroup $E_\Gamma(s)$ containing $s$, and
$E_\Gamma(s)= \{g\in\Gamma:gs^mg^{-1}=s^{\pm m} \text{ for some }m\geq1\}$.
It is the stabilizer of $\{x_s^+,x_s^-\}$ in any action for which
$s$ is loxodromic WPD, and it is hyperbolically embedded in $\Gamma$;
see \cite[Lemma~6.5, Corollaries~2.9 and~6.6]{DGO}.
In particular, $C_\Gamma(s^m)\subseteq E_\Gamma(s)$ for every
$m\neq0$. 

For a fixed non-elementary acylindrical action
$\Gamma\curvearrowright S$, the intersection of $E_\Gamma(s)$ over
all loxodromic elements $s$ of this action is the maximal finite normal
subgroup of $\Gamma$ (see~\cite[Lemma~5.6]{AMS}). These elements
belong to $K_{\mathrm{lox}}$. Thus
$\bigcap_{s\in K_{\mathrm{lox}}}E_\Gamma(s)=\{e\}$ whenever $\Gamma$
has trivial amenable radical.

\section{Boomerang algebras, properties and associated characters}
\label{sec:prep}
We begin with basic properties of boomerang subalgebras and the positive-definite functions associated with them. We then study finite dimensional and amenable boomerang subalgebras and give conditions under which they are trivial.
\subsection{Boomerang subalgebras and their basic properties}
We now formally define the notion of a boomerang subalgebra.
\begin{definition}[Boomerang subalgebra]\label{def:boomerang-algebra}
A von Neumann subalgebra $\cA\leq L(\Gamma)$ is called boomerang if for every
$s\in\Gamma$ there is a strictly increasing sequence $(n_i)\subseteq\bN$ such that
$\lambda(s^{n_i})\cA\lambda(s^{n_i})^*\to\cA$ in the Effros--Mar\'echal topology.
\end{definition}
If $s$ has finite order $m$, taking $n_i=im$ gives
$\lambda(s^{n_i})\cA\lambda(s^{-n_i})=\cA$ for every $i$.
Thus the condition in Definition~\ref{def:boomerang-algebra} is
automatic for finite-order elements. Also, note that every invariant subalgebra is boomerang.

By Lemma~\ref{lem:conjugate-expectation}, the map $M\mapsto\E_M$ is
equivariant. By the equivalence of topologies in Section~\ref{sec:prelim},
$M$ is boomerang if and only if $\E_M$ is boomerang.

We first relate boomerang subgroups to boomerang subalgebras through the maps $L$ and $J$ as defined in the introduction. 
\begin{proposition}\label{prop:maps}
The maps
\[
\begin{tabular}{lllllll}
$\Sub(\Gamma)$ & $\stackrel{L}{\longrightarrow}$ & $\SA(L(\Gamma))$ & $\qquad\qquad$ & $\SA(L(\Gamma))$ & $\stackrel{J}\longrightarrow$ & $\Sub(\Gamma)$ \\
$\Delta$ & $\mapsto$ & $L(\Delta)$ & $\qquad\qquad$ & $M$ & $\mapsto$ &
$\Gamma_M:=\{\gamma:\lambda(\gamma)\in M\}$
\end{tabular}
\]
are measurable and $\Gamma$-equivariant, and satisfy $JL(\Delta)=\Delta$ and
$LJ(M)\leq M$. Moreover, $L$ is a topological embedding.
\end{proposition}
\begin{proof}
$J$ is Borel, since $\{M:\lambda(\gamma)\in M\}
=\{M:\|\E_M(\lambda(\gamma))\|_2=1\}$ and $M\mapsto \E_M(x)$ is Borel
for the Effros--Mar\'echal structure \cite{haagerup1998effros}.
It is $\Gamma$-equivariant since
$\Gamma_{\lambda(\gamma) M\lambda(\gamma)^*}=\gamma\Gamma_M\gamma^{-1}$.
Lemma~\ref{lem:diagram-inclusions} gives the equivariant topological
embedding $L$. The identities
$JL(\Delta)=\Delta$ and $LJ(M)\leq M$ follow from the definitions.
\end{proof}

The topological embedding in Proposition~\ref{prop:maps} identifies
boomerang subgroups with boomerang group von Neumann algebras.
\begin{proposition}\label{prop:boomerang-group-algebra}
For a subgroup $H\leq\Gamma$, the algebra $L(H)$ is boomerang if and
only if $H$ is a boomerang subgroup.
\end{proposition}
\begin{proof}
Apply the equivariant topological embedding of Proposition~\ref{prop:maps}
to the conjugates of $H$ by powers of each $s\in\Gamma$.
\end{proof}

The next example contrasts with the subgroup case. Every finite-index
subgroup is boomerang, since it has only finitely many conjugates.
However, a finite-index subalgebra need not be boomerang.
\begin{example}\label{prop:finite-index-not-boomerang}\label{prop:fi}\normalfont
There exist finite-index subalgebras which are not boomerang.
\end{example}
\begin{proof}
Let $\mathbb F_2=\langle a,b\rangle$, and let $\sigma$ be the
automorphism interchanging $a$ and $b$. We also write $\sigma$ for
the induced automorphism of $M=L(\mathbb F_2)$. Let $N=M^\sigma$.
The automorphism $\sigma$ is outer. Indeed, if
$\sigma=\operatorname{Ad}(v)$ for a unitary $v\in M$, then
$\lambda(b^{-1})v\lambda(a)=v$. The Fourier coefficients of $v$
would therefore be constant on every set
$\{b^{-m}ha^m:m\in\mathbb Z\}$. For fixed $h$, these elements
are pairwise distinct: their images in the abelianization
$\mathbb Z^2$ differ by $(m,-m)$. Square summability would force
all the Fourier coefficients of $v$ to be zero, a contradiction.
Thus the fixed-point inclusion for the outer order-two action
$\langle\sigma\rangle$ has $[M:N]=2$.

We now show that $N$ is not boomerang. Suppose otherwise. Applying
the definition to $b$, choose a strictly increasing sequence
$(n_i)\subseteq\mathbb N$ such that
$N_i=\lambda(b^{n_i})N\lambda(b^{-n_i})\to N$ in the
Effros--Mar\'echal topology. Since
$\mathbb E_N=(\operatorname{id}+\sigma)/2$, Lemma~\ref{lem:conjugate-expectation} gives
$\mathbb E_{N_i}(\lambda(a)) =\lambda(b^{n_i})\mathbb E_N (\lambda(b^{-n_i}ab^{n_i}))\lambda(b^{-n_i}) =\tfrac12\bigl(\lambda(a) +\lambda(b^{n_i}a^{-n_i}ba^{n_i}b^{-n_i})\bigr)$.
For every $n_i\geq1$, the reduced word
$b^{n_i}a^{-n_i}ba^{n_i}b^{-n_i}$ is different from $b$.
Since distinct group unitaries are orthogonal in $L^2(M,\tau_0)$,
$\|\mathbb E_{N_i}(\lambda(a))-\mathbb E_N(\lambda(a))\|_2^2 =\tfrac14\|\lambda(b^{n_i}a^{-n_i}ba^{n_i}b^{-n_i}) -\lambda(b)\|_2^2 =\tfrac12$.
This contradicts the pointwise $\|\cdot\|_2$ convergence of the
conditional expectations, which follows from $N_i\to N$.
\end{proof}

\subsection{Boomerang characters}
We first relate recurrence of UCP maps to recurrence of their
associated positive-definite functions.
\begin{proposition}\label{lem:boom_char}
Let $T=T_\psi\in\UCP_{\tau_0}(L(\Gamma))$ be as in
Proposition~\ref{prop:ira-ucp-correspondence}. Then $T$ is boomerang if
and only if for every $s\in\Gamma$ there is a strictly increasing
sequence $(n_i)\subseteq\bN$ such that
$\psi(s^{-n_i}gs^{n_i},s^{-n_i}hs^{n_i})\to\psi(g,h)$ for all $g,h\in\Gamma$.
In particular, if $\cA$ is a boomerang subalgebra, then $\phi_\cA$
is a boomerang positive-definite function.
\end{proposition}
\begin{proof}
The correspondence given by Proposition~\ref{prop:ira-ucp-correspondence}
is equivariant and is a homeomorphism. Consequently, $s^{n_i}\cdot T\to T$ point
ultraweakly if and only if $s^{n_i}\cdot\psi\to\psi$ pointwise. If $\cA$ is boomerang, then $\E_\cA$ is boomerang by
Lemma~\ref{lem:diagram-inclusions}. The function $\phi_\cA$ is the restriction
of $\psi_{\E_\cA}$ to the diagonal, so the first assertion applied
to $s^{-1}$ gives the recurrence in Definition~\ref{def:boomerang-function}.
\end{proof}
The converse to the assertion about subalgebras in
Proposition~\ref{lem:boom_char} fails even for torsion-free hyperbolic groups.
\begin{example}\label{rem:boomerang-character-converse}\normalfont
Let $\Gamma=\mathbb F_2=\langle a,b\rangle$, and let
$\alpha\in\operatorname{Aut}(\Gamma)$ be determined by
$\alpha(a)=a^{-1}$ and $\alpha(b)=b^{-1}$. We also denote by $\alpha$
the induced trace-preserving automorphism of $L(\Gamma)$. Let
$\mathcal M=L(\Gamma)^\alpha$. Since $\alpha^2=\operatorname{id}$,
$\mathbb E_{\mathcal M}=(\operatorname{id}+\alpha)/2$.

For every $g\neq e$, we have $\alpha(g)\neq g$. Hence, for $g\ne e$, we have
$$\phi_{\mathcal M}(g) =\tfrac12\tau_0\bigl((\lambda(g)+\lambda(\alpha(g)))\lambda(g)^*\bigr) =\tfrac12.$$
Moreover, $\phi_{\mathcal M}(e)=1$. Thus
$\phi_{\mathcal M}=(1+\delta_e)/2$ is conjugation invariant,
and in particular is boomerang.

But the same argument used in Example~\ref{prop:fi} shows that $\mathcal M$ is not boomerang.
In particular, $\mathcal M$ is not invariant.
Thus conjugation invariance of $\phi_{\mathcal M}$ does not imply that $\mathcal M$ is boomerang or invariant.
\end{example}

\subsection{Amenable boomerang subalgebras}
\label{sec:boomerang-amenability}
In \cite[Theorem~B]{amrutam2309amenable}, we showed that every amenable IRA of $\Gamma$ is supported on subalgebras of $L(\operatorname{Rad}(\Gamma))$.
In particular, when the amenable radical is trivial, every amenable IRA is trivial.
Here we study individual amenable boomerang subalgebras and give conditions under which they are trivial.

Before that, let us fix some notation. For a compact $\Gamma$-space $X$, fix $x_0\in X$ and represent $C(X)$ as multiplication operators on $\ell^2(\Gamma)$ by $f\delta_g=f(gx_0)\delta_g$.
When $X$ is minimal, this representation is faithful, and we identify $C(X)$ with its image in $\mathbb B(\ell^2(\Gamma))$; see \cite[Section~3]{amrutam2309amenable}.
We use this convention throughout, with $\lambda(g)f\lambda(g)^*=g\cdot f$, where $(g\cdot f)(x)=f(g^{-1}x)$.
\subsubsection{Finite dimensional boomerang algebras}
We begin by exploring finite dimensional boomerang subalgebras, which are natural examples of amenable subalgebras.

To deal with finite-dimensional invariant subalgebras inside $L(\Gamma)$ for i.c.c. groups $\Gamma$, Chifan-Das-Sun~\cite{CDS2023} used that every finite subset $F\subset\Gamma\setminus\{e\}$ admits $t\in\Gamma$ such that $tFt^{-1}\cap F=\emptyset$.
For boomerang subalgebras, recurrence is available along powers of a fixed element.
We therefore need separation for all sufficiently large powers of one element, so that it holds along any recurrence sequence.
This motivates the definition below.
\begin{definition}
Let $\Gamma$ be a countable discrete group. We say that $\Gamma$ has
property $ESD$ if, given finitely many infinite index subgroups
$\Lambda_i<\Gamma$ and finite subsets $F_i\subset\Gamma$, there are
$s\in\Gamma$ and $N\geq1$ such that
$s^n\notin\bigcup_i\bigcup_{t\in F_i}t\Lambda_i$ for all $n\geq N$.
We say that $\Gamma$ has property $ECI$ if, for every finite subset
$F\subset\Gamma\setminus\{e\}$, there are $s\in\Gamma$ and $N\geq1$
such that $s^nFs^{-n}\cap F=\emptyset$ for all $n\geq N$.
\end{definition}
See also \cite[Definition~2.10]{amrutam2026crossed} for a formulation of property $ECI$
requiring escape under every fixed nontrivial conjugating element,
for finite sets outside its centralizer. 

It turns out we only require powers of one element to escape finite unions of cosets of arbitrary infinite index subgroups and gives a sufficient condition for $ECI$.
\begin{proposition}\label{prop:esd-implies-eci}
Let $\Gamma$ be an icc group with property $ESD$. Then $\Gamma$ has
property $ECI$.
\end{proposition}
\begin{proof}
Fix a finite subset $F\subset\Gamma\setminus\{e\}$. For each conjugate
pair $g,h\in F$, choose $t_{g,h}\in\Gamma$ with
$t_{g,h}gt_{g,h}^{-1}=h$. Then
$$\{t\in\Gamma:tFt^{-1}\cap F\neq\emptyset\} =\bigcup_{\substack{g,h\in F\\g\text{ conjugate to }h}} t_{g,h}C_\Gamma(g).$$
Since $\Gamma$ is icc, every $C_\Gamma(g)$ in this union has infinite
index. Apply property $ESD$ to these finitely many cosets.
\end{proof}

\begin{lemma}\label{lem:eci-characterization}
A countable discrete group $\Gamma$ has property $ECI$ if and only if, for
every finite subset $F\subset\Gamma\setminus\{e\}$, there exists $t\in\Gamma$
such that $t^nf\neq ft^n$ for every $f\in F$ and every $n\geq1$.
\end{lemma}
\begin{proof}
Suppose that $\Gamma$ has property $ECI$, and choose $t$ for $F$ as in the
definition. If $t^n$ commuted with some $f\in F$, then
$t^{kn}ft^{-kn}=f$ for every $k\geq1$, a contradiction.

Conversely, choose $t$ satisfying the stated condition. For each
$f,h\in F$, there is at most one $n\geq1$ with $t^nft^{-n}=h$.
Indeed, two such exponents would give a nonzero power of $t$ commuting
with $f$. There are therefore only finitely many $n\geq1$ with
$t^nFt^{-n}\cap F\neq\emptyset$. This proves property $ECI$.
\end{proof}
To obtain examples of groups with property $ECI$, we use the properties $P_{\mathrm{ana}}$, $P_{\mathrm{geo}}$ and $P_{\mathrm{nai}}$ introduced in \cite{BCD}.

A group $\Gamma$ is said to have property $P_{\mathrm{ana}}$ (see \cite[Definition~1]{BCD}) if, for every finite subset $F\subset\Gamma\setminus\{e\}$, there
exist $s\in\Gamma$ and $C>0$ such that
$\left\|\sum_{j\geq1}a_j\lambda(s^{-j}xs^j)\right\| \leq C\|a\|_2$
for every $x\in F$ and every finitely supported
$a\in\ell^2(\mathbb N)$.

\begin{proposition}\label{prop:pana-implies-eci}
Every group with property $P_{\mathrm{ana}}$ has property $ECI$.
\end{proposition}
\begin{proof}
Fix a finite subset $F\subset\Gamma\setminus\{e\}$, and choose $s,C$ as
above. Suppose that $s^nx=xs^n$ for some $x\in F$ and $n\geq1$.
For $M\geq1$, set $a_j=M^{-1}$ when
$j\in\{n,2n,\ldots,M^2n\}$, and $a_j=0$ otherwise. Then
$\|a\|_2=1$, whereas the displayed inequality gives
$M\|\lambda(x)\|=M\leq C$, a contradiction. Apply
Lemma~\ref{lem:eci-characterization}.
\end{proof}
\begin{definition}[\cite{BCD}, Definition~3]\label{def:pgeo}
Let $\Gamma$ act continuously on a nonempty compact Hausdorff space $B$.
The pair $(\Gamma,B)$ has property $P_{\mathrm{geo}}$ if, for every finite subset $F\subseteq\Gamma\setminus\{e\}$, there exist $t\in\Gamma$, a finite set $\{b_s:s\in S\}=\operatorname{Fix}_B(t)$, and open neighbourhoods $V_s$ of $b_s$ such that:
\begin{enumerate}
\item for every $b\in B$, the sequence $t^nb$ converges to some $b_s$;
\item $gV_s\cap V_{s'}=\emptyset$ for every $g\in F$ and $s,s'\in S$;
\item if $b\in V_s$ and $t^nb\notin V_s$ for some $n\geq1$, then $t^{n+1}b\notin V_s$.
\end{enumerate}
\end{definition}
Let $G$ be a connected real semisimple Lie group without compact factors.
If $\Gamma\leq G$ has trivial centre and Zariski-dense image in the adjoint algebraic group of $G$, then its action on the Furstenberg boundary of $G$ has property $P_{\mathrm{geo}}$ \cite[Theorem~2]{BCD}.
In particular, this applies to lattices in such groups $G$ with trivial centre.
For a concrete example, $(\operatorname{PSL}_d(\mathbb Z),\mathbb{RP}^{d-1})$ has property $P_{\mathrm{geo}}$ for every $d\geq2$ \cite[Example~1]{BCD}.
Property $P_{\mathrm{geo}}$ implies $P_{\mathrm{ana}}$ by \cite[Lemmas~2.3 and~2.4]{BCD}.
Thus these groups also have property $ECI$ by Proposition~\ref{prop:pana-implies-eci}.

\begin{definition}[\cite{BCD}, Definition~4]\label{def:pnai}
A countable discrete group $\Gamma$ has property $P_{\mathrm{nai}}$
if, for every finite subset $F\subset\Gamma\setminus\{e\}$, there
exists an infinite order element $s\in\Gamma$ such that
$\langle s,t\rangle\cong\langle s\rangle*\langle t\rangle$ for
every $t\in F$, under the canonical homomorphism.
\end{definition}
Every acylindrically hyperbolic group with no nontrivial finite normal subgroup has property $P_{\mathrm{nai}}$ by \cite[Theorem~2.3]{abbott2019property}.
There are also examples beyond this class.
By \cite[Corollary~1.3]{KS2016}, a group acting properly and cocompactly on a finite dimensional essential $\operatorname{CAT}(0)$ cube complex with no Euclidean factor has property $P_{\mathrm{nai}}$, provided it has no nontrivial finite normal subgroup.
This includes the infinite simple Burger--Mozes lattices in products of trees; see \cite[p.~545]{KS2016}.
These groups are not acylindrically hyperbolic, since a countable acylindrically hyperbolic group has uncountably many normal subgroups \cite[Theorem~8.1(a) and the following paragraph]{osin2016acylindrically}.
\begin{proposition}\label{prop:pnai-implies-eci}
Every group with property $P_{\mathrm{nai}}$ has property $ECI$.
\end{proposition}
\begin{proof}
Fix a finite subset $F\subset\Gamma\setminus\{e\}$, and choose $s$
as in Definition~\ref{def:pnai}. The free product normal form shows
that $s^nt\neq ts^n$ for every $t\in F$ and $n\geq1$. Apply
Lemma~\ref{lem:eci-characterization}.
\end{proof}
Chifan--Das--Sun proved that every finite dimensional $\Gamma$-invariant von Neumann subalgebra of $L(\Gamma)$ is $\mathbb C1$ when $\Gamma$ is icc \cite[Proposition~3.2]{CDS2023}.
We prove the analogous statement for boomerang subalgebras under property $ECI$.
\begin{proposition}\label{prop:finite-dimensional-boomerang}
Assume that $\Gamma$ has property $ECI$. Then every finite dimensional
boomerang von Neumann subalgebra $\mathcal M\leq L(\Gamma)$ is
$\mathbb C1$.
\end{proposition}
\begin{proof}
Let $\mathcal M\leq L(\Gamma)$ be a finite dimensional boomerang
subalgebra. Suppose that $\mathcal M\neq\mathbb C1$. Choose
$\{b_1,\ldots,b_k\}$ to be an orthonormal basis of
$\mathcal M\ominus\mathbb C1$ with respect to $\tau_0$. Thus
$\|b_j\|_2=1$ and $\tau_0(b_j)=0$ for every $1\leq j\leq k$.
Let $\epsilon>0$. By truncating the Fourier expansions of the
$b_j$, choose a finite subset
$F_\epsilon\subset\Gamma\setminus\{e\}$ and elements
$b_j^\epsilon\in\mathbb C\Gamma$ such that
$\|b_j-b_j^\epsilon\|_2<\epsilon, \qquad\operatorname{supp}(b_j^\epsilon)\subseteq F_\epsilon \quad(1\leq j\leq k)$.
In particular, $\|b_j^\epsilon\|_2\leq1+\epsilon$.
By property $ECI$, choose $s\in\Gamma$ and $N\geq1$ such that
$s^nF_\epsilon s^{-n}\cap F_\epsilon=\emptyset \qquad\text{for every }n\geq N$.

For each $t\in\Gamma$, the element
$\mathbb E_{\mathcal M}(\lambda(t)^*b_1\lambda(t))$ has trace zero.
We may therefore write
$\mathbb E_{\mathcal M}(\lambda(t)^*b_1\lambda(t)) =\sum_{j=1}^k\alpha(t,j)b_j$.
Since conditional expectations are $\|\cdot\|_2$-contractive,
$\sum_{j=1}^k|\alpha(t,j)|^2 =\|\mathbb E_{\mathcal M}(\lambda(t)^*b_1\lambda(t))\|_2^2 \leq\|b_1\|_2^2=1$.
Thus $|\alpha(t,j)|\leq1$ for every $t$ and $j$. By
Lemma~\ref{lem:conjugate-expectation},
$\mathbb E_{\lambda(t)\mathcal M\lambda(t)^*}(b_1) =\sum_{j=1}^k\alpha(t,j)\lambda(t)b_j\lambda(t)^*$.

Since $\mathcal M$ is boomerang, choose a strictly increasing
sequence $(n_i)\subseteq\mathbb N$ such that
$\mathcal M_i=\lambda(s^{n_i})\mathcal M\lambda(s^{-n_i})\to\mathcal M$
in the Effros--Mar\'echal topology. Since
$\mathbb E_{\mathcal M}(b_1)=b_1$, choose $i$ large enough that
$n_i\geq N$ and
$\left\|\sum_{j=1}^k\alpha(s^{n_i},j) \lambda(s^{n_i})b_j\lambda(s^{-n_i})-b_1\right\|_2 =\|\mathbb E_{\mathcal M_i}(b_1)-b_1\|_2<\epsilon$.
For this $i$, the choice of the Fourier truncations gives
\begin{align*}
 &\left\|\mathbb E_{\mathcal M_i}(b_1)
 -\sum_{j=1}^k\alpha(s^{n_i},j)
       \lambda(s^{n_i})b_j^\epsilon\lambda(s^{-n_i})\right\|_2\\
 &\qquad\leq\sum_{j=1}^k|\alpha(s^{n_i},j)|
          \|b_j-b_j^\epsilon\|_2\leq k\epsilon,
\end{align*}
and
$\left\|\sum_{j=1}^k\alpha(s^{n_i},j) \lambda(s^{n_i})b_j^\epsilon\lambda(s^{-n_i})\right\|_2 \leq\sum_{j=1}^k|\alpha(s^{n_i},j)|\|b_j^\epsilon\|_2 \leq k(1+\epsilon)$.
The Fourier support of this last sum is contained in
$s^{n_i}F_\epsilon s^{-n_i}$, whereas the Fourier support of
$b_1^\epsilon$ is contained in $F_\epsilon$. By
the choice of $s$ and $N$, these supports are disjoint, so
$\left\langle b_1^\epsilon, \sum_{j=1}^k\alpha(s^{n_i},j) \lambda(s^{n_i})b_j^\epsilon\lambda(s^{-n_i})\right\rangle=0$.
Consequently, by Cauchy--Schwarz and the preceding estimates,
\begin{align*}
 1=\langle b_1,b_1\rangle
 &\leq\left|\langle b_1,b_1-\mathbb E_{\mathcal M_i}(b_1)\rangle\right|\\
 &\quad+\left|\left\langle b_1,
 \mathbb E_{\mathcal M_i}(b_1)
 -\sum_{j=1}^k\alpha(s^{n_i},j)
       \lambda(s^{n_i})b_j^\epsilon\lambda(s^{-n_i})
 \right\rangle\right|\\
 &\quad+\left|\left\langle b_1-b_1^\epsilon,
 \sum_{j=1}^k\alpha(s^{n_i},j)
       \lambda(s^{n_i})b_j^\epsilon\lambda(s^{-n_i})
 \right\rangle\right|\\
 &\leq\epsilon+k\epsilon+\epsilon k(1+\epsilon).
\end{align*}
Since $\epsilon>0$ was arbitrary, this is a contradiction.
\end{proof}
\subsubsection{Triviality of amenable boomerang subalgebras}
We now pass from finite dimensional to arbitrary amenable boomerang subalgebras.
Under the following hypotheses, the conclusion holds for every amenable boomerang subalgebra, strengthening the corresponding conclusion for amenable IRAs in \cite[Theorem~B]{amrutam2309amenable}.

For an amenable von Neumann subalgebra $\mathcal M\leq L(\Gamma)$,
let $\operatorname{Hype}_{\tau_0}(\mathcal M)$ denote the states
$\varphi$ on $\mathbb B(\ell^2(\Gamma))$ such that
$\varphi|_{L(\Gamma)}=\tau_0$ and
$\varphi(xT)=\varphi(Tx)$ for every $x\in\mathcal M$ and
$T\in\mathbb B(\ell^2(\Gamma))$. This set is nonempty by
\cite[Proposition~2.4]{amrutam2309amenable}. 

Recall that an action $\Gamma\curvearrowright X$ on a compact Hausdorff space is called \emph{extremely proximal} if, for every proper closed subset $C\subsetneq X$ and every nonempty open subset $U\subseteq X$, there is $g\in\Gamma$ such that $gC\subseteq U$.
\begin{theorem}\label{prop:amenable-boomerang}
Let $\Gamma$ be a countable group satisfying one of the following conditions:
\begin{enumerate}
\item $\Gamma$ has property $P_{\mathrm{nai}}$;
\item $\Gamma$ admits a faithful extremely proximal action on a compact Hausdorff space with more than two points;
\item $\Gamma$ acts continuously on a nonempty compact Hausdorff space $B$ such that $(\Gamma,B)$ has property $P_{\mathrm{geo}}$.
\end{enumerate}
Then every amenable boomerang von Neumann subalgebra $\mathcal M\leq L(\Gamma)$ is $\mathbb C1$.
\end{theorem}
\begin{proof}
Let $\mathcal M$ be an amenable boomerang subalgebra of $L(\Gamma)$.
Let us first assume that $\Gamma$ has property $P_{\mathrm{nai}}$.
We shall show that for every $a\in\mathcal M$, $\tau_0(a\lambda(h)^*)=0$ for every $h\ne e$.
Let $h\ne e$.
Since $\Gamma$ has property $P_{\mathrm{nai}}$, we can find $t\in\Gamma$ of infinite order such that $\Lambda=\langle h,t\rangle\cong\langle h\rangle\star\langle t\rangle$.
In particular, $\Lambda$ is a non-elementary hyperbolic group.
Let $X=\partial\Lambda$.
Choose $x_0\in X$ and one representative $r$ for each right coset
$\Lambda r$ in $\Gamma$. Represent $C(X)$ on $\ell^2(\Gamma)$ by
$f\delta_{\ell r}=f(\ell x_0)\delta_{\ell r}$ for $\ell\in\Lambda$.
Then $\lambda(g)f\lambda(g)^*=g\cdot f$ for every $g\in\Lambda$.
Since $\langle h,t\rangle\cong\langle h\rangle\star\langle t\rangle$, it follows that $h\{x_t^+,x_t^-\}\cap\{x_t^+,x_t^-\}=\emptyset$, where $x_t^+,x_t^-$ are the fixed points of $t$.
Since $\mathcal M$ is a boomerang subalgebra of $L(\Gamma)$, there exists an increasing sequence of distinct elements $\{n_i\}\subseteq\mathbb N$ such that
$\lambda(t)^{n_i}\mathcal M\lambda(t)^{-n_i}\to\mathcal M \text{ in EM-topology}$.

Since $\mathcal M$ is amenable, it follows from \cite[Proposition~2.4]{amrutam2309amenable} that $\operatorname{Hype}_{\tau_0}(\mathcal M)\ne\emptyset$.
Let $\varphi\in\operatorname{Hype}_{\tau_0}(\mathcal M)$.
We now observe that
$$\varphi_i(\cdot)=\varphi\left(\lambda(t)^{-n_i}\cdot\lambda(t)^{n_i}\right) \in\operatorname{Hype}_{\tau_0}(\lambda(t)^{n_i}\mathcal M\lambda(t)^{-n_i}).$$
Therefore, after passing to a subnet and using the proof of \cite[Proposition~4.7]{amrutam2309amenable}, we see that
$\psi:=\text{weak}^*-\lim_i\varphi_i \in\operatorname{Hype}_{\tau_0}(\mathcal M)$.
Moreover, since $t$ has north--south dynamics on $X$, it follows that $\psi|_{C(X)}=c\delta_{x_t^+}+(1-c)\delta_{x_t^-}$ for some $c\in[0,1]$.
Since $h\{x_t^+,x_t^-\}\cap\{x_t^+,x_t^-\}=\emptyset$, we can find a non-negative continuous function $f\in C(X)$ with $f|_{\{x_t^+,x_t^-\}}=1$ and $f|_{h\{x_t^+,x_t^-\}}=0$.
Since $\psi((1-f)^2)=0$, Cauchy--Schwarz gives
\begin{align*}
\tau_0(a\lambda(h)^*)
&=\psi(a\lambda(h)^*)\\
&=\psi\left(a\lambda(h)^*f\right)\\
&=\psi\left(a(h^{-1}\cdot f)\lambda(h)^*\right)\\
&=\psi\left(a\sqrt{h^{-1}\cdot f}\sqrt{h^{-1}\cdot f}\lambda(h)^*\right).
\end{align*}
Using the Cauchy--Schwarz inequality, we see that
\begin{align*}
|\psi(a\lambda(h)^*)|
&\leq\sqrt{\psi(a(h^{-1}\cdot f)a^*)}
\sqrt{\psi(\lambda(h)(h^{-1}\cdot f)\lambda(h)^*)}\\
&=\sqrt{\psi((h^{-1}\cdot f)a^*a)}\sqrt{\psi(f)},
\end{align*}
where the last equality uses the fact that $\psi$ is a $\mathcal M$-hypertrace.

Observe that $(h^{-1}\cdot f)(x_t^\pm)=f(hx_t^\pm)=0$, hence, another application of the Cauchy--Schwarz inequality gives
$$|\psi((h^{-1}\cdot f)a^*a)|^2 \leq\psi((h^{-1}\cdot f)^2)\psi((a^*a)^2) =0.$$
Thus $\tau_0(a\lambda(h)^*)=0$.
Since $h\ne e$ was arbitrary, $\mathcal M=\mathbb C1$.

Assume now that $\Gamma$ admits a faithful extremely proximal action on a space $X$ with more than two points.
Fix $g\in\Gamma\setminus\{e\}$.
As before, we show that $\tau_0(a\lambda(g)^*)=0$ for every $a\in\mathcal M$.
First, $X$ has no isolated points.
Indeed, if $z$ were isolated, extreme proximality would send the proper closed set $X\setminus\{z\}$ into $\{z\}$, contradicting injectivity since $|X|>2$.
By faithfulness, $X\setminus\operatorname{Fix}(g)$ is a nonempty open set.
Since $X$ has no isolated points, this set is infinite.
Choose $x_1\notin\operatorname{Fix}(g)$ and $x_2\notin\operatorname{Fix}(g)\cup\{x_1,gx_1,g^{-1}x_1\}$.
Thus $g\{x_1,x_2\}\cap\{x_1,x_2\}=\emptyset$.
Choose open neighbourhoods $U\ni x_1$ and $V\ni x_2$ such that
$\overline U\cap\overline V=\emptyset, \quad g(\overline U\cup\overline V)\cap(\overline U\cup\overline V)=\emptyset$.
By extreme proximality, choose $t\in\Gamma$ such that $t(X\setminus U)\subseteq V$.
Equivalently, $t^{-1}(X\setminus V)\subseteq U$.
Since $U\cap V=\emptyset$, one has $t(V)\subseteq V$.
Consequently, $t^n(X\setminus U)\subseteq V$ for every $n\geq1$.
Choose $\varphi\in\operatorname{Hype}_{\tau_0}(\mathcal M)$.
By the boomerang property, choose a strictly increasing sequence $(n_j)$ such that $\mathcal M_j=\lambda(t^{n_j})\mathcal M\lambda(t^{-n_j})\to\mathcal M$ in the Effros--Mar\'echal topology.
Let $\varphi_j=\varphi\circ\operatorname{Ad}(\lambda(t^{-n_j}))$, and take a weak$^*$ convergent subnet with limit $\psi$.
As in case $(1)$, it follows that $\psi\in\operatorname{Hype}_{\tau_0}(\mathcal M)$.
Let $\nu=\varphi|_{C(X)}$, regarded as a probability measure on $X$.
If $f\in C(X)$ vanishes on $\overline U\cup\overline V$, then $f(t^ny)\to0$ for every $y\in X$.
Dominated convergence gives
$\varphi_j(f)=\int_Xf(t^{n_j}y)\,d\nu(y)\longrightarrow0$.
Thus $\psi(f)=0$ for every such $f$.
Choose $0\leq f\leq1$ in $C(X)$ with $f|_{\overline U\cup\overline V}=1$ and $f|_{g(\overline U\cup\overline V)}=0$.
It follows that
$\psi((1-f)^2)=\psi((g^{-1}\cdot f)^2)=0$.
Arguing as in case $(1)$, using the Cauchy--Schwarz inequality and the $\mathcal M$-centrality of $\psi$, we obtain that $\tau_0(a\lambda(g)^*)=0$ for every $a\in\mathcal M$.
Since $g\neq e$ was arbitrary, $\mathcal M=\mathbb C1$.

Finally, assume that $(\Gamma,B)$ has property $P_{\mathrm{geo}}$.
Fix $g\in\Gamma\setminus\{e\}$.
Apply \cite[Definition~3]{BCD} with $F=\{g\}$.
There are $t\in\Gamma$, a finite set $\{b_s:s\in S\}\subseteq B$ of fixed points of $t$, and open neighbourhoods $V_s\ni b_s$ such that every $t^nb$ converges to some $b_s$, and $gV_s\cap V_{s'}=\emptyset$ for all $s,s'\in S$.
In particular,
$g\{b_s:s\in S\}\cap\{b_s:s\in S\}=\emptyset$.
Let $K_s=\{b\in B:\lim_{n\to\infty}t^nb=b_s\}$.
These sets form a Borel partition of $B$.
Choose $\varphi\in\operatorname{Hype}_{\tau_0}(\mathcal M)$ and a recurrence sequence $(n_j)$ for $\mathcal M$ under conjugation by $\lambda(t)$.
As in case $(1)$, a weak$^*$ convergent subnet of $\varphi\circ\operatorname{Ad}(\lambda(t^{-n_j}))$ has limit $\psi\in\operatorname{Hype}_{\tau_0}(\mathcal M)$.
Let $\nu=\varphi|_{C(B)}$, regarded as a probability measure on $B$.
For $f\in C(B)$, dominated convergence gives
$\lim_j\int_B f(t^{n_j}b)\,d\nu(b) =\sum_{s\in S}\nu(K_s)f(b_s)$, hence, $$\psi|_{C(B)}=\sum_{s\in S}\nu(K_s)\delta_{b_s}.$$
Choose $0\leq f\leq1$ in $C(B)$ with $f(b_s)=1$ and $f(gb_s)=0$ for every $s\in S$.
Then $\psi((1-f)^2)=\psi((g^{-1}\cdot f)^2)=0$.
Arguing as in case $(1)$, using the Cauchy--Schwarz inequality and the $\mathcal M$-centrality of $\psi$, we obtain $\tau_0(a\lambda(g)^*)=0$ for every $a\in\mathcal M$.
Since $g\neq e$ was arbitrary, $\mathcal M=\mathbb C1$.
\end{proof}
Of course, there are many natural classes of examples which admit extremely proximal faithful actions that are not topologically free. Let us consider the class of Thompson's groups first.
\begin{example}\normalfont
Every amenable boomerang von Neumann subalgebra of $L(T)$ is $\mathbb C1$, where $T$ is Thompson's group.
The natural action of Thompson's group $T$ on $\mathbb S^1$ is faithful but not topologically free; see \cite[Sections~1 and~5]{CFP1996}.
Indeed, $T$ contains nonidentity elements supported in a proper closed arc, which therefore fix the complementary open arc pointwise.
Let $C\subsetneq\mathbb S^1$ be closed and let $U\subseteq\mathbb S^1$ be nonempty and open.
Choose proper closed arcs $I,J$ with dyadic endpoints such that $C\subseteq I$ and $J\subseteq U$.
There is $t\in T$ mapping $I$ onto $J$, so $tC\subseteq U$.
Thus the action is extremely proximal, and Theorem~\ref{prop:amenable-boomerang} applies.
\end{example}
\begin{example}\normalfont
Let $\mathcal C$ be the binary Cantor set.
The natural action of Thompson's group $V$ on $\mathcal C$ is faithful,
extremely proximal and not topologically free; see the proof of
\cite[Theorem~4.6\textup{(iii)}]{LBMB2018}.
The same holds for every countable subgroup
$\Gamma\leq\operatorname{Homeo}(\mathcal C)$ containing $V$.
Consequently, every amenable boomerang von Neumann subalgebra of
$L(\Gamma)$ is $\mathbb C1$ by
Theorem~\ref{prop:amenable-boomerang}.
\end{example}
The examples above may suggest that our result applies only to
$C^*$-simple groups. The following example of Le Boudec~\cite{LeBoudec2017} dispels
this impression.
\begin{example}\normalfont
Let $\Omega$ be a finite set with $|\Omega|\geq3$, and let
$F<F'\leq\operatorname{Sym}(\Omega)$, where $F$ acts freely and
transitively on $\Omega$.
Consider Le Boudec's group $\Gamma=G(F,F')$ from
\cite{LeBoudec2017}.
Its action on the boundary of the $|\Omega|$-regular tree is faithful
and not topologically free; see the proof of
\cite[Theorem~C]{LeBoudec2017}.
It is extremely proximal by \cite[Proposition~4.26]{LBMB2018}.
Thus every amenable boomerang von Neumann subalgebra of $L(\Gamma)$
is $\mathbb C1$ by Theorem~\ref{prop:amenable-boomerang}.
These groups have trivial amenable radical but are not $C^*$-simple
\cite[Theorem~C]{LeBoudec2017}.
\end{example}

\section{Split higher rank lattices.}
\label{sec:higherrank}
We now prove Theorem~\ref{thm:ksplit} for the groups from Subsection~\ref{ksplit}. We first use Proposition~\ref{prop:support} to locate the Fourier support of the conditional expectation of a highest root element.
\begin{lemma}\label{lem:split-localization}
Let $\Gamma$ be as in Subsection \ref{ksplit} and $M \leq L(\Gamma)$ be a boomerang subalgebra. Then
\[\E_M(\lambda(x_{\theta}(t))) \in L(\bfU_{\theta}(K) \cap \Gamma), \qquad {\text{whenever }} x_{\theta}(t) \in \Gamma.\]
\end{lemma}

\begin{proof}
If $t=0$, the conclusion follows from $\E_M(1)=1$. Assume $t\neq0$.
Since $\bfU_\theta=Z(\bfU)$, every
$a\in\Gamma\cap\bfU(K)$ commutes with $x_\theta(t)$.
Moreover, $C_\Gamma(a^m)=C_\Gamma(a)$ for every $m\geq1$. Indeed, in a faithful linear representation, commuting with $a^m$ implies
commuting with $\log(a^m)=m\log a$, and hence with $a$.
Proposition~\ref{prop:support} therefore gives
$\supp \E_M(\lambda(x_\theta(t))) \subseteq R(x_\theta(t)) \subseteq C_\Gamma(\Gamma\cap\bfU(K))$.
The group $\Gamma\cap\bfU(K)$ is Zariski dense in $\bfU$.
Indeed, commensurability supplies a nontrivial root element in
$\Gamma$ for every positive root; its powers are Zariski dense in
that root group, and the positive root groups generate $\bfU$.

It remains to show that  $C_{\bfG}(\bfU)=\bfU_\theta$. To see this,
an element centralizing $\bfU$ belongs to
$N_{\bfG}(\bfU)=\mathbf B$. Its torus component acts
trivially on $\bfU/[\bfU,\bfU]$, hence all simple
roots take the value $1$ on it. Since $\bfG$ is adjoint, that
component is trivial. The element therefore belongs to
$Z(\bfU)=\bfU_\theta$. 
\end{proof}
We next construct suitable conjugates of the highest root group inside $\bfG_\theta$. The following properties will allow us to compare the Fourier coefficients of products of conditional expectations.
\begin{proposition} \label{prop:mat_conf}
Let $\Gamma_{\theta} = \Gamma \cap \bfG_\theta(K)$ and $\Upsilon_{\theta} = \Gamma \cap \bfU_\theta(K)$. Then there is some element $s \in \Upsilon_{\theta}$ of infinite order and an infinite sequence $\gamma_i \in \Gamma_{\theta}$; such that if we write $w_i := \gamma_i w \gamma_i^{-1}$ and define the product map
\[\Psi_{i,j}:\Upsilon_{\theta} \times \Upsilon_{\theta} \rightarrow \Gamma_{\theta}, \qquad \Psi_{i,j}(u,v) = u_iv_j = \gamma_i u \gamma_i^{-1} \gamma_j v \gamma_j^{-1}.\]
The following properties hold for every $i \ne j$:
\begin{enumerate}
\item \label{itm:injectiv} 
$\Psi_{i,j}$ is injective.
\item \label{itm:commutators}
If for $k,n \in \bZ \setminus \{0\}$ and $v,w \in \Upsilon_{\theta}$ we have
\[[s_i^k v_j,(w_i s_j^k)^n] =
[\gamma_i s^k \gamma_i^{-1}\gamma_j v \gamma_j^{-1}, (\gamma_i w \gamma_i^{-1} \gamma_j s^k \gamma_j^{-1})^n] = 1,\] then $v = w = s^k$. 
\end{enumerate}
\end{proposition}
\begin{proof}
In this proof we make the matrix calculations in $\SL_2(K)$ and apply $\iota$.
Since $\ker\iota\subseteq\{\pm I_2\}$, an equality between the images gives the corresponding equality in $\PGL_2(K)$.
The projective image of a root element determines its parameter.
In particular,
\[x_{\theta}(k) = \iota\begin{pmatrix}1&k\\0&1\end{pmatrix}, \quad x_{-\theta}(k) = \iota\begin{pmatrix}1&0\\k&1\end{pmatrix}, \quad \bfU_{\theta}(K) = \{x_{\theta}(k) \mid k \in K\} \cong K.\]
The groups $\Upsilon_\theta$ and $\bfU_\theta(\cO_S)$ are commensurable, and the same holds for the negative root groups.
Hence we can find $d \in \bN$ such that $x_\theta(d),x_{-\theta}(d)\in\Gamma$.
Let $J=\{t\in K:x_\theta(t)\in\Upsilon_\theta\}$, so that $J$ is commensurable with $\cO_S$.

Consider the set
$\cT_K= \{\zeta+\zeta^{-1}-2\in K\setminus\{0\}: \zeta\text{ is a root of unity}\}$.
Each such root of unity $\zeta$ satisfies a quadratic equation over $K$, so that $[K(\zeta):K]\leq2$ and hence
$[\mathbb Q(\zeta):\mathbb Q]\leq2[K:\mathbb Q]$. Thus the set $\cT_K$ is finite.

Choose a finite place $\mfp\notin S$ such that
$v_{\mfp}(a)=0$ for every $a\in\cT_K$ and $v_{\mfp}(r)\geq0$ for every $r\in J$.
Such a place exists: $\cT_K$ is a finite set of nonzero algebraic integers, and $mJ\subseteq\cO_S$ for some positive integer $m$, so we also take $\mfp\nmid m$.

Choose a positive integer $q\in d\mathbb Z$ with $v_{\mfp}(q)>0$, and set $s = x_{\theta}(d)$ and $\gamma_i=x_{-\theta}(qi)$. Thus if $v = x_{\theta}(t) \in \Upsilon_{\theta}$ then
\[v_i := \gamma_i v \gamma_i^{-1} = \iota(I_2+tA_i), \qquad {\text{where }}
A_i=\begin{pmatrix}-qi&1\\-q^2i^2&qi\end{pmatrix}.\]
Note that $A_i$ are nilpotent of order $2$ and that $(I_2+t_1A_i)(I_2+t_2A_i) = I_2 + (t_1+t_2)A_i$ for all $t_1,t_2 \in J$. We proceed to show that these choices give us the desired properties. 

Note that the nilpotent matrices $A_i,A_j$ have distinct kernels, namely
$K(1,qi)^{\mathsf T}$ and $K(1,qj)^{\mathsf T}$. Therefore the
one-parameter groups $I_2+KA_i$ and $I_2+KA_j$ have trivial
intersection even after passage to $\mathrm{PGL}_2(K)$.
Rearranging an equality of two products proves
(\ref{itm:injectiv}).

For (\ref{itm:commutators}) take 
$v = x_{\theta}(m)$, $w = x_{\theta}(r)$ for some $m,r \in J$ and put
$X = (I_2+dkA_i)(I_2+mA_j) \qquad T=(I_2+rA_i)(I_2+dkA_j)$. The hypothesis gives $[\overline X,\overline T^{\,n}]=1$ in $\PGL_2(K)$. Our goal is to show that $r=m=dk$.
If $r=0$, then $\overline T^{\,n}$ has the unique fixed point $\ker A_j$, whereas $\overline X$ does not fix $\ker A_j$, a contradiction. Hence $r\neq0$.

Since $\Tr(A_iA_j)=-q^2(i-j)^2$,
$\Tr(T)-2=-rkd q^2(i-j)^2\ne0, \qquad v_{\mfp}(\Tr(T)-2)>0$.
It follows that no eigenvalue of $T$ is a root of unity. Otherwise,
writing its eigenvalues as $\zeta,\zeta^{-1}$ would give
$\Tr(T)-2\in\cT_K$, a contradiction.
In particular, its eigenvalues $\lambda,\lambda^{-1}$ are distinct
and $\lambda^2$ has infinite order. For every $n\ne0$, the
centralizer of $\overline{T}^n$ in $\mathrm{PGL}_2(\overline K)$ fixes
each of the two eigenlines of $T$.

Choose a basis of $K^2$ consisting of vectors in the kernels of
$A_i$ and $A_j$. In this basis, for some $a,b\in K^\times$,
$I_2+tA_i=\begin{pmatrix}1&at\\0&1\end{pmatrix}, \qquad I_2+zA_j=\begin{pmatrix}1&0\\bz&1\end{pmatrix}$.
For $t,z\ne0$, the projective fixed points of their product are the
roots of $Q(Z)=Z^2-atZ-\frac{at}{bz}$. If the roots $\omega_+$ and
$\omega_-$ are distinct, then $\omega_++\omega_-=at$ and
$\omega_+\omega_-=-\frac{at}{bz}$. Thus the unordered pair
$\{\omega_+,\omega_-\}$ determines $t$ and $z$.

In particular, since $\overline X$ commutes with $\overline T^{\,n}$, it fixes both eigenlines of $T$ individually.
If $m=0$, then $X=I_2+dkA_i$ is nontrivial unipotent and has only one fixed point, a contradiction. Thus $m\neq0$.
Moreover, $\overline{X}\ne1$, by the trivial intersection of the two
unipotent groups. Hence $\overline{X}$ has precisely these two fixed points.
Comparing the polynomial $Q(Z)$ for $(t,z)=(kd,m)$
and $(t,z)=(r,kd)$ gives the desired equalities $r=kd$ and $m=kd$.
\end{proof}
As promised, we now return to our general $K$-split lattice $\Gamma$ from Subsection \ref{ksplit} and apply Proposition \ref{prop:mat_conf} to $\Gamma_{\theta} = \bfG_{\theta}(K) \cap \Gamma$. We show that, when $\Gamma_M=\{e\}$, at most one of the conditional expectations $\E_M(\lambda(s_i^k))$ can be nonzero for each fixed $k\neq0$.
\begin{proposition} \label{prop:exclusion}
Let $\Gamma$ be as in Subsection \ref{ksplit} and apply Proposition \ref{prop:mat_conf} to the subgroup $\Gamma_{\theta}$, retaining all the notation there. Let $M \in \SA(L(\Gamma))$ be a boomerang algebra with the property that $\Gamma_{M} = \{\gamma \ | \ \lambda(\gamma) \in M\} =  \trivgp$. Then for every $0 \ne k \in \bZ$ and every $i \ne j \in \bN$ 
either $\E_{M}(\lambda(s_i^k)) = 0$ or $\E_{M}(\lambda(s_j^k)) = 0$.
\end{proposition}
\begin{proof}
Suppose both expectations are nonzero. Applying Lemma~\ref{lem:split-localization} to the boomerang subalgebras $\lambda(\gamma_i)^*M\lambda(\gamma_i)$ and $\lambda(\gamma_j)^*M\lambda(\gamma_j)$, we write
$$A:=\E_M(\lambda(s_i^k))=\sum_{v\in\Upsilon_\theta}c_v\lambda(v_i),\text{ and }
B:=\E_M(\lambda(s_j^k))=\sum_{w\in\Upsilon_\theta}d_w\lambda(w_j).$$
Since $\E_M$ is the orthogonal projection onto $L^2(M)$, we have $c_{s^k}=\|A\|_2^2>0$ and $d_{s^k}=\|B\|_2^2>0$. Also, $c_e=\tau_0(A)=0$ and $d_e=\tau_0(B)=0$.

By the injectivity in Proposition~\ref{prop:mat_conf}(\ref{itm:injectiv}), the Fourier coefficient of $AB$ at $v_iw_j$ is $c_vd_w$. Indeed, this holds for finite partial sums of $A$, and the general statement follows by $L^2$-approximation, since multiplication on the right by $B$ is continuous on $L^2$.
Since $A\in M$, bimodularity gives
$AB=\E_M(A\lambda(s_j^k))
=\sum_{v\in\Upsilon_\theta}c_v\E_M(\lambda(v_is_j^k))$,
where the series converges in $L^2$.
Suppose that $d_w\neq0$ for some $w\neq s^k$. Since $d_e=0$, $w\neq e$. The Fourier coefficient of $AB$ at $s_i^kw_j$ is $c_{s^k}d_w\neq0$. Thus there is some $v\in\Upsilon_\theta$ with $c_v\neq0$ such that $s_i^kw_j\in\supp\E_M(\lambda(v_is_j^k))$.
Since $c_e=0$, $v\neq e$. Also, $v_is_j^k\neq e$ by Proposition~\ref{prop:mat_conf}(\ref{itm:injectiv}). Proposition~\ref{prop:support} therefore gives $n\geq1$ such that $[s_i^kw_j,(v_is_j^k)^n]=1$.
By Proposition~\ref{prop:mat_conf}(\ref{itm:commutators}), $w=s^k$, a contradiction.
Therefore $B=d_{s^k}\lambda(s_j^k)$. Applying $\E_M$ again gives $d_{s^k}=d_{s^k}^2$. Since $d_{s^k}>0$, $d_{s^k}=1$, so $\lambda(s_j^k)\in M$. But $s_j^k\neq e$, contradicting $\Gamma_M=\{e\}$.
\end{proof}
We now combine the IRS classification in Theorem~\ref{thm:NSZ} with Proposition~\ref{prop:exclusion}. In the case where $J_*\mu=\delta_{\{e\}}$, averaging over the conjugates of $s^k$ gives the vanishing required in Proposition~\ref{prop:onedim}.
\begin{proof}[Proof of Theorem~\ref{thm:ksplit}]
Let $\mu$ be an ergodic IRA on $\SA(L(\Gamma))$ and push it forward to an ergodic IRS $\nu = J_*\mu$ on $\Gamma$. 
We may split the proof into two cases by Theorem~\ref{thm:NSZ}.

\medskip
\noindent\emph{Case 1: $\nu=\nu_H$ for a finite-index subgroup $H<\Gamma$.} 
In this case,  $\mu\bigl(\{M:\Gamma_M\in\{\gamma H\gamma^{-1}:\gamma\in\Gamma\}\}\bigr)=1$.  Let
$H_0:=\bigcap_{\gamma\in\Gamma}\gamma H\gamma^{-1}$
be the normal core of $H$; since $[\Gamma:H]<\infty$, $H$ has finitely many conjugates, so $H_0\lhd\Gamma$ and $[\Gamma:H_0]<\infty$.  For $\mu$-a.e.\ $M$ we then have $L(H_0)\subseteq L(\Gamma_M)\subseteq M\subseteq L(\Gamma)$. By Borel density, every finite-index subgroup of $\Gamma$ is Zariski dense in $\bfG$. Hence an element of $\Gamma$ with finite conjugacy class belongs to $Z(\bfG)=\{e\}$, so $\Gamma$ is icc. Since $\Gamma$ is icc and $[\Gamma:H_0]<\infty$, the subgroup $H_0$ is relatively icc in $\Gamma$: for $\gamma\neq e$,
$[H_0:C_{H_0}(\gamma)] \geq\frac{[\Gamma:C_\Gamma(\gamma)]}{[\Gamma:H_0]}=\infty$.
By \cite[Lemma 3.5]{jiang-scm}, every intermediate von Neumann algebra $L(H_0)\subseteq M\subseteq L(\Gamma)$ is of the form $M=L(K)$ for a subgroup $H_0\leq K\leq\Gamma$.  By maximality, $K=\Gamma_M$, so $M=L(\Gamma_M)=L(J(M))$ for $\mu$-a.e.\ $M$, whence
$\mu=(L\circ J)_*\mu=L_*\nu$.

\medskip
\noindent\emph{Case 2: $\nu=\delta_{\{e\}}$.} 
Our goal is to prove that $\mu=\delta_{\mathbb{C}}$. By Proposition~\ref{prop:onedim} it suffices to find $s\in\Gamma$ of infinite order with 
$\E_M(\lambda(s^k))=0$ for all $k\neq0$ and $\mu$-a.e.~ $M$.
Apply Proposition \ref{prop:mat_conf} to the group $\Gamma_{\theta} = \Gamma \cap \bfG_{\theta}(K)$ and $\Upsilon_{\theta} = \Gamma \cap \bfU_{\theta}(K)$. We keep all the notation of that proposition and proceed to show that the element $s \in \Upsilon_{\theta}$ provided there satisfies this vanishing property. For this let us fix $0 \ne k \in \bZ$.

The two assumptions of Proposition \ref{prop:exclusion} hold for $\mu$-almost every $M$: by our standing Case~2 assumption, $\Gamma_M=\{e\}$, and $M$ is boomerang by Poincar\'e recurrence, as noted in the introduction. Hence, by Proposition~\ref{prop:exclusion}, almost surely, at most one of the $N$ terms $\{\E_M(\lambda(s_i^k))\}_{i\leq N}$ is nonzero, and since $\|\E_M(\lambda(s_i^k))\|_2\leq1$,
$\sum_{i=1}^N\|\E_M(\lambda(s_i^k))\|_2^2\leq1$ for $\mu$-a.e.\ $M$. Let $s_i=\gamma_is\gamma_i^{-1}$.
Integrating and using conjugation-invariance of $\mu$ together with $$\|\E_M(\lambda(\gamma)x\lambda(\gamma)^*)\|_2=\|\E_{\lambda(\gamma)^*M\lambda(\gamma)}(x)\|_2,$$ we get that
\begin{align*}
    1\geq\int\sum_{i=1}^N\|\E_M(\lambda(s_i^k))\|_2^2\,d\mu(M)
        &=\sum_{i=1}^N\int\|\E_M(\lambda(s^k))\|_2^2\,d\mu(M)
        \\&=N\int\|\E_M(\lambda(s^k))\|_2^2\,d\mu(M).
\end{align*}
As $N$ is arbitrary, $\int\|\E_M(\lambda(s^k))\|_2^2\,d\mu(M)=0$, i.e.\ $\E_M(\lambda(s^k))=0$ for $\mu$-a.e.\ $M$.  Discarding a countable union of null sets, for $\mu$-a.e.\ $M$ we have $\E_M(\lambda(s^k))=0$ for all $k\geq1$, and taking adjoints, $\E_M(\lambda(s^{-k}))=\E_M(\lambda(s^k)^*)=\E_M(\lambda(s^k))^*=0$ as well; hence $\E_M(\lambda(s^k))=0$ for all $k\neq0$.  Proposition~\ref{prop:onedim} now gives $\mu=\delta_{\mathbb{C}}=L_*\delta_{\{e\}}=L_*\nu$.
\end{proof}

\section{Boomerang subalgebras of acylindrically hyperbolic groups}
\label{sec:boomerang-hyperbolic}
In this section, we prove Theorems~\ref{thm:hyperbolic_tf}
and~\ref{thm:hyperbolic_tf_boom} from the introduction. We first show
that boomerang UCP maps of acylindrically hyperbolic groups with
trivial amenable radical and of their finite products are Fourier
multipliers. Applying this to conditional expectations gives
Theorem~\ref{thm:hyperbolic_tf_boom}. We then use the same result
together with Poincar\'e recurrence to prove
Theorem~\ref{thm:hyperbolic_tf}.

We use the terminology and notation from
Subsection~\ref{subsec:acylindrically-hyperbolic}. We begin with the following lemma which supplies the thick set used
in Theorem~\ref{thm:boomerang-ucp-ah}.
\begin{lemma}\label{prop:loxodromic-thickness}
Let $\Gamma$ be acylindrically hyperbolic with trivial amenable radical
and set $K=\{(g,h):g\in K_{\mathrm{lox}}\text{ or }h\in K_{\mathrm{lox}}\}$. Then the following hold true. 
\begin{enumerate}
    \item Given $s\in K_{\mathrm{lox}}$ and $u\not\in E_{\Gamma}(s)$, there exists $N\ge 1$ such that $s^{-k}u^{-1}s^ku\in K_{\mathrm{lox}}$
for all $k\geq N$.
\item For every $(g,h)\ne(e,e)$ there exist $s\in K_{\mathrm{lox}}$
and $N\geq1$ such that, for every increasing sequence
$(n_i)\subseteq\bN$ with $n_{i+1}-n_i\geq N$, the conjugates
$x_i=(s^{-n_i}gs^{n_i},s^{-n_i}hs^{n_i})$ satisfy
$x_j^{-1}x_i\in K$ for all $i\ne j$. If $g\ne h$, then
$x_j^{-1}x_i$ is off the diagonal for all $i\ne j$.
\item $K$ is $\psi$-thick for every normalized
positive-definite function $\psi$ on $\Gamma\times\Gamma$ which is
boomerang for diagonal conjugation.
\end{enumerate}
\end{lemma}
\begin{proof}
Let $s\in K_{\mathrm{lox}}$ and $u\notin E_\Gamma(s)$. Using \cite[Theorem~1.4]{osin2016acylindrically}, choose an acylindrical
action on a hyperbolic Cayley graph in which $s$ is loxodromic.
The fixed-point pairs
of $s$ and $u^{-1}su$ are disjoint. Indeed, if
they shared a point, then $\langle s,u^{-1}su\rangle$ would fix that
point and hence contain no independent loxodromic elements. By
\cite[Theorem~1.1]{osin2016acylindrically}, this subgroup would be
virtually cyclic. The infinite cyclic subgroups generated by $s$ and
$u^{-1}su$ would then have an infinite intersection, so these elements
would have the same fixed-point pair. Since the stabilizer of the
fixed-point pair of $s$ is $E_\Gamma(s)$, this would give
$u\in E_\Gamma(s)$, a contradiction. Apply
\cite[Lemma~1.2]{abbott2019property} to $s^{-1}$ and $u^{-1}su$ to
obtain $N\geq1$ such that $s^{-k}u^{-1}s^ku\in K_{\mathrm{lox}}$
for all $k\geq N$. This proves $(1)$.

Now fix $(g,h)\ne(e,e)$. By
Subsection~\ref{subsec:acylindrically-hyperbolic},
$\bigcap_{s\in K_{\mathrm{lox}}}E_\Gamma(s)=\{e\}$.
If $g\ne h$, choose $s\in K_{\mathrm{lox}}$ such that
$hg^{-1}\notin E_\Gamma(s)$. Then at least one of $g,h$ lies outside
$E_\Gamma(s)$. If $g=h$, choose $s\in K_{\mathrm{lox}}$ such that
$g\notin E_\Gamma(s)$. In either case, let $u$ be a coordinate outside
$E_\Gamma(s)$ and choose $N$ as in $(1)$. Let $(n_i)$ satisfy $n_{i+1}-n_i\geq N$ and let $x_i$ be as defined in $(2)$. For $j>i$, let $k=n_j-n_i$. The coordinate
of $x_j^{-1}x_i$ corresponding to $u$ is
$s^{-n_i}(s^{-k}u^{-1}s^ku)s^{n_i}\in K_{\mathrm{lox}}$.
Thus $x_j^{-1}x_i\in K$ for $j>i$ and hence for $i\ne j$, since $K$ is
closed under inversion. If $g\ne h$ and $x_j^{-1}x_i$ is diagonal for some $j>i$, its coordinates
are equal, which gives $g^{-1}s^kg=h^{-1}s^kh$. Thus
$hg^{-1}\in C_\Gamma(s^k)\subseteq E_\Gamma(s)$, contrary to the choice
of $s$. The differences are therefore off the diagonal. This proves $(2)$.

For the thickness assertion, choose $(n_i)$ to be a recurrence sequence
for $\psi$ under $s$ and pass to a subsequence with $n_{i+1}-n_i\geq N$.
Then $\psi(x_i)\to\psi(g,h)$, as required. This proves $(3)$.
\end{proof}

To show that a boomerang UCP map $T_\psi$ is a Fourier multiplier, we
must show that $\psi$ vanishes off the diagonal. For a single factor, if
$g\in K_{\mathrm{lox}}$ and $\psi(g,h)\ne0$,
Proposition~\ref{prop:support} gives $h\in E_\Gamma(g)$. When $h\ne g$,
the free product normal form shows that for the pairwise differences of
suitable conjugates of $(g,h)$, the second entry commutes with no positive
power of the first. The same proposition then gives the vanishing
needed for us to appeal to Lemma~\ref{lem:thickness-vanishing}.
\begin{theorem}\label{thm:boomerang-ucp-ah}
Let $n\geq1$ and let $\Gamma_i$ be countable acylindrically hyperbolic groups with trivial amenable radical. Let $\Gamma=\oplus_{1\leq i\leq n}\Gamma_i$. Let $T\in\UCP_{\tau_0}(L(\Gamma))$ be boomerang. Then, it is a Fourier multiplier associated with a Boomerang character on $\Gamma$.
\end{theorem}

\begin{proof}
Let $T$ be a boomerang UCP map. By Proposition~\ref{prop:ira-ucp-correspondence},
we can write $T=T_\psi$, where $\psi\in\operatorname{PD}_1^{\mathrm R}(\Gamma\times\Gamma)$
is given by $\psi(g,h)=\tau_0(T(\lambda(g))\lambda(h)^*)$.

First assume that $n=1$. We first show that $\psi(g,h)=0$ for $g\in K_{\mathrm{lox}}$ and $h\ne g$. Note that for every $s\in K_{\mathrm{lox}}$,
Proposition~\ref{prop:support} gives
$\supp(T(\lambda(s)))\subseteq\bigcup_{k\geq1}C_\Gamma(s^k)\subseteq E_\Gamma(s)$. By this support inclusion and the regularity of $\psi$, it suffices
to consider $h\in E_\Gamma(g)\setminus\{e,g\}$.

Since $E_\Gamma(g)$ is hyperbolically embedded in $\Gamma$ and $\Gamma$
has no nontrivial finite normal subgroup, we can choose a generalized
loxodromic element $t$ such that
$\langle E_\Gamma(g),t\rangle\cong E_\Gamma(g)*\langle t\rangle$
under the canonical homomorphism
\cite[Theorem~2.5 and the proof of Theorem~2.2]{abbott2019property}.
The free product normal form gives $g\notin E_\Gamma(t)$.
It follows from the first bit of Lemma~\ref{prop:loxodromic-thickness} that there is $N\geq1$
such that $t^{-k}g^{-1}t^kg\in K_{\mathrm{lox}}$ for all $k\geq N$.

By Proposition~\ref{lem:boom_char}, choose a recurrence sequence
$(n_i)$ for $\psi$ under $t$, so that $t^{n_i}\cdot\psi\to\psi$
pointwise. Pass to a subsequence with $n_{i+1}-n_i\geq N$.
Let $x_i=(t^{-n_i}gt^{n_i},t^{-n_i}ht^{n_i})$, so that
$\psi(x_i)\to\psi(g,h)$. For $i<j$, we have
\begin{align*}
x_i^{-1}x_j
 &= \bigl(t^{-n_i}g^{-1}t^{n_i-n_j}gt^{n_j},
          t^{-n_i}h^{-1}t^{n_i-n_j}ht^{n_j}\bigr)\\
 &= \bigl(t^{-n_i}(g^{-1}t^{n_i-n_j}gt^{n_j-n_i})t^{n_i},
          t^{-n_i}(h^{-1}t^{n_i-n_j}ht^{n_j-n_i})t^{n_i}\bigr).
\end{align*}
It follows from Proposition~\ref{prop:ira-ucp-correspondence} that
\begin{align*}
\psi(x_i^{-1}x_j)
=\tau_0\bigl(T(\lambda(t^{-n_i}g^{-1}t^{n_i-n_j}gt^{n_j}))
              \lambda(t^{-n_i}h^{-1}t^{n_i-n_j}ht^{n_j})^*\bigr).
\end{align*}
Let $a_{ij}=t^{-n_i}g^{-1}t^{n_i-n_j}gt^{n_j}$ and $b_{ij}=t^{-n_i}h^{-1}t^{n_i-n_j}ht^{n_j}$.
Observe that $a_{ij}\in K_{\mathrm{lox}}$ since it is
conjugate to the inverse of $t^{-(n_j-n_i)}g^{-1}t^{n_j-n_i}g$. We claim that $b_{ij}\notin E_\Gamma(a_{ij})$. After conjugation by
$t^{n_i}$, both words are cyclically reduced in $E_\Gamma(g)*\langle t\rangle$.
For every $k\geq1$, the reduced words of the conjugates of
$b_{ij}a_{ij}^k$ and $a_{ij}^kb_{ij}$ start with $h^{-1}$ and $g^{-1}$,
respectively. Thus $b_{ij}$ commutes with no positive power of $a_{ij}$.
Also, $b_{ij}$ has infinite order, so if it belonged to the virtually
cyclic group $E_\Gamma(a_{ij})$, a nonzero power of $b_{ij}$ would be
a positive power of $a_{ij}$. This proves the claim. Since
$\supp(T(\lambda(a_{ij})))\subseteq E_\Gamma(a_{ij})$(cf~
Proposition~\ref{prop:support}), it follows that $\psi(x_i^{-1}x_j)=0$. For $i>j$, the same follows
from $\psi(y^{-1})=\overline{\psi(y)}$. Since $\psi(x_i)\to\psi(g,h)$,
Lemma~\ref{lem:thickness-vanishing} gives $\psi(g,h)=0$.
The coordinate swap preserves positive definiteness, regular marginals
and recurrence under diagonal conjugation. Apply the preceding argument
to $(g,h)\mapsto\psi(h,g)$ to obtain $\psi(g,h)=0$ whenever $g\ne h$
and at least one of $g,h$ belongs to $K_{\mathrm{lox}}$.

Now fix $g\ne h$. Choose $s$ and $N$ by
Lemma~\ref{prop:loxodromic-thickness}(2), and take a recurrence sequence
$(n_i)$ for $\psi$ under $s$ with $n_{i+1}-n_i\geq N$. The conjugates
$x_i$ satisfy $\psi(x_i)\to\psi(g,h)$, and $x_j^{-1}x_i$ lies in $K$
and is off the diagonal for $i\ne j$. By the preceding argument,
$\psi(x_j^{-1}x_i)=0$ for $i\ne j$. Lemma~\ref{lem:thickness-vanishing}
gives $\psi(g,h)=0$.

Now suppose that $n\geq2$. Fix $g=(g_i)$ and $h=(h_i)$ with $g\ne h$,
and choose $j$ such that $g_j\ne h_j$. Identify $\Gamma_j$ with the
$j$th factor of $\Gamma$ and set
$\psi_j=\psi|_{\Gamma_j\times\Gamma_j}\in
\operatorname{PD}_1^{\mathrm R}(\Gamma_j\times\Gamma_j)$.
For every $s\in\Gamma_j$, a recurrence sequence for $\psi$ under
$s$ is also a recurrence sequence for $\psi_j$, since conjugation
by $s$ preserves $\Gamma_j\times\Gamma_j$. Thus the UCP map associated
to $\psi_j$ is boomerang by Proposition~\ref{lem:boom_char}.
By the case $n=1$, $\psi_j$ vanishes off the diagonal.

Apply Lemma~\ref{prop:loxodromic-thickness}(2) to $(g_j,h_j)$ to choose
$s\in\Gamma_j$ and $N\geq1$. Choose a recurrence sequence $(n_i)$ for
$\psi$ under $s$ and pass to a subsequence with $n_{i+1}-n_i\geq N$.
Let $x_i=(s^{-n_i}gs^{n_i},s^{-n_i}hs^{n_i})$ and
$y_i=(s^{-n_i}g_js^{n_i},s^{-n_i}h_js^{n_i})$.

Fix $i\ne k$. For every $r\ne j$, the $r$th pair of coordinates of
$x_k^{-1}x_i$ is $(g_r^{-1}g_r,h_r^{-1}h_r)=(e,e)$. Its $j$th pair
of coordinates is $y_k^{-1}y_i$, which is off the diagonal by
Lemma~\ref{prop:loxodromic-thickness}(2). Consequently,
$\psi(x_k^{-1}x_i)=\psi_j(y_k^{-1}y_i)=0$. Since $\psi(x_i)=\psi(s^{-n_i}gs^{n_i},s^{-n_i}hs^{n_i})
=(s^{n_i}\cdot\psi)(g,h)\to\psi(g,h)$ and $\psi(x_k^{-1}x_i)=0$
for all $i\ne k$, Lemma~\ref{lem:thickness-vanishing} gives
$\psi(g,h)=0$.

Thus $\psi$ vanishes off the diagonal. Its restriction
$\phi(g)=\psi(g,g)$ is normalized and positive definite, and
$T=T_\phi$ by Proposition~\ref{prop:ira-ucp-correspondence}.
For each $s\in\Gamma$, applying Proposition~\ref{lem:boom_char} to
$s^{-1}$ and restricting to the diagonal gives a recurrence sequence
for $\phi$ as in Definition~\ref{def:boomerang-function}.
Thus $\phi$ is a Boomerang character.
\end{proof}

To pass from maps to subalgebras, we use the idempotence of the
conditional expectation.
\begin{theorem}\label{boomerangfromsubgroup}\label{thm:boomerang-products-ah}
Let $n\geq1$ and let $\Gamma_i$ be countable acylindrically hyperbolic
groups with trivial amenable radical. Let
$\Gamma=\oplus_{1\leq i\leq n}\Gamma_i$. Then every boomerang von
Neumann subalgebra $\mathcal M\leq L(\Gamma)$ is $L(H)$ for a unique
boomerang subgroup $H\leq\Gamma$.
\end{theorem}

\begin{proof}
By Lemma~\ref{lem:diagram-inclusions}, $\E_{\mathcal M}$ is boomerang.
Apply Theorem~\ref{thm:boomerang-ucp-ah} to write $\E_{\mathcal M}=T_\phi$
for a Boomerang character $\phi$. Since $\E_{\mathcal M}$ is idempotent,
we see that $\phi(g)^2=\phi(g)$ for every $g\in\Gamma$. Let
$(\pi,\mathcal H,\xi)$ be the GNS representation of $\phi$. Since
$\|\pi(g)\xi-\xi\|^2=2-2\operatorname{Re}\phi(g)$, the set
$H=\{g:\phi(g)=1\}$ is the stabilizer of $\xi$ and is therefore a subgroup.
Thus $\phi=1_H$, so $\E_{\mathcal M}$ and $\E_{L(H)}$ agree on every
group unitary. They agree on $L(\Gamma)$ by normality. Consequently,
$\mathcal M=L(H)$ with $H=\{g:\lambda(g)\in\mathcal M\}$.
By Proposition~\ref{prop:boomerang-group-algebra}, $H$ is boomerang.
\end{proof}

This proves Theorem~\ref{thm:hyperbolic_tf_boom}. We now use Poincar\'e recurrence to pass from boomerang UCP maps to IRMs.
\begin{proof}[Proof of Theorem~\ref{thm:hyperbolic_tf}]
Let $\mu$ be an IRM of $\Gamma$. As noted in the introduction,
Poincar\'e recurrence gives that $\mu$-almost every map is boomerang.
By Theorem~\ref{thm:boomerang-ucp-ah}, these maps are Fourier multipliers.
Thus $\Gamma$ is IRM rigid. By the implication proved in the
introduction, every IRA $\mu$ of $\Gamma$ is of the form $L_*\nu$
for a unique IRS $\nu$ of $\Gamma$.
\end{proof}
\section{Invariant random maps of mixed identity-free groups}
\label{sec:ira-mif}

In this section, we first prove Theorems~\ref{thm:ira-mif-intro}
and~\ref{thm:ira-boundaries-trees-intro} from the introduction.
For the first theorem, we use mixed identity-freeness to transfer an
IRM of $\oplus_{i=1}^n\Gamma_i$ to an IRM of
$\oplus_{i=1}^n(\Gamma_i*\langle t_i\rangle)$. We apply
Theorem~\ref{thm:hyperbolic_tf} to the latter group and then restrict
the associated positive-definite functions to recover the IRM of
$\oplus_{i=1}^n\Gamma_i$. We also
apply Theorem~\ref{thm:ira-mif-intro} to linear groups
(Corollary~\ref{cor:linear-irm}) and use character rigidity to describe
the IRMs of IRM rigid groups with property $\mathrm{CR}_{\mathrm{fin}}$
(Theorem~\ref{prop:irm-char} and Corollary~\ref{cor:irm-pgl}).

We first record the continuity and
equivariance properties of the pullback and restriction maps. Let $\Gamma\leq\Lambda$ be countable groups and let
$\rho\colon\Lambda\to\Gamma$ be a homomorphism such that
$\rho|_\Gamma=\operatorname{id}_\Gamma$. On $\operatorname{PD}_1(\Lambda\times\Lambda)$, let $\Lambda$ act by
$(s\cdot\theta)(a,b)=\theta(s^{-1}as,s^{-1}bs)$. On
$\operatorname{PD}_1^{\mathrm R}(\Gamma\times\Gamma)$, let $\Lambda$
act through $\rho$, so that
$(s\cdot\psi)(g,h)=\psi(\rho(s)^{-1}g\rho(s),\rho(s)^{-1}h\rho(s))$.
When we say $\Lambda$ equivariance, we mean in the sense defined above.
\begin{lemma}\label{lem:pd-pullback}
Let $\Gamma\leq\Lambda$ be countable groups and let
$\rho\colon\Lambda\to\Gamma$ be a homomorphism such that
$\rho|_\Gamma=\operatorname{id}_\Gamma$. Then the following hold true.
\begin{enumerate}
\item The map
$\mathfrak{P}_\rho\colon\operatorname{PD}_1^{\mathrm R}(\Gamma\times\Gamma)
\to\operatorname{PD}_1(\Lambda\times\Lambda)$ given by
$\mathfrak{P}_\rho(\psi)=\psi\circ(\rho\times\rho)$ is continuous
for the pointwise topologies and $\Lambda$-equivariant for the
actions above. In particular, it is $\Gamma$-equivariant for the
diagonal conjugation actions. For every
$\psi\in\operatorname{PD}_1^{\mathrm R}(\Gamma\times\Gamma)$ and
$a\in\Lambda$, we have
$\mathfrak{P}_\rho(\psi)(a,e)=\mathfrak{P}_\rho(\psi)(e,a)
=\delta_{\rho(a),e}$.

\item The restriction map
$\operatorname{Res}_\Gamma\colon\operatorname{PD}_1(\Lambda\times\Lambda)
\to\operatorname{PD}_1(\Gamma\times\Gamma)$ given by
$\operatorname{Res}_\Gamma(\theta)=\theta|_{\Gamma\times\Gamma}$
is continuous for the pointwise topologies and $\Gamma$-equivariant
for the diagonal conjugation actions. Moreover,
$\operatorname{Res}_\Gamma(\mathfrak{P}_\rho(\psi))=\psi$ for every
$\psi\in\operatorname{PD}_1^{\mathrm R}(\Gamma\times\Gamma)$.
If $\theta(a,b)=0$ for all distinct $a,b\in\Lambda$, then
$\operatorname{Res}_\Gamma(\theta)(g,h)=0$ for all distinct
$g,h\in\Gamma$.
\end{enumerate}
\end{lemma}

\begin{proof}
\emph{(1).}
Since $\rho\times\rho$ is a homomorphism, the positive-definiteness
inequality for $\mathfrak{P}_\rho(\psi)$ at a finite list of pairs
$(a_k,b_k)$ is the positive-definiteness inequality for $\psi$ at
the pairs $(\rho(a_k),\rho(b_k))$. Also,
$\mathfrak{P}_\rho(\psi)(e,e)=\psi(e,e)=1$. For each $a,b\in\Lambda$,
evaluation of $\mathfrak{P}_\rho(\psi)$ at $(a,b)$ is evaluation of
$\psi$ at $(\rho(a),\rho(b))$. This gives continuity in the pointwise
topologies.

For $s,a,b\in\Lambda$, we have
\begin{align*}
(s\cdot\mathfrak{P}_\rho(\psi))(a,b)
&=\psi\bigl(\rho(s^{-1}as),\rho(s^{-1}bs)\bigr)\\
&=\psi\bigl(\rho(s)^{-1}\rho(a)\rho(s),
             \rho(s)^{-1}\rho(b)\rho(s)\bigr)\\
&=\mathfrak{P}_\rho(s\cdot\psi)(a,b).
\end{align*}
Since $\rho$ fixes $\Gamma$ pointwise, the action on the domain
restricts to diagonal conjugation by $\Gamma$.
The marginal identities follow from
$\psi(\rho(a),e)=\psi(e,\rho(a))=\delta_{\rho(a),e}$.

\emph{(2).}
The restriction of $\theta$ to $\Gamma\times\Gamma$ is normalized
and positive definite, since the defining conditions on
$\Gamma\times\Gamma$ are among those on $\Lambda\times\Lambda$.
For $g,h\in\Gamma$, we have
$\operatorname{Res}_\Gamma(\theta)(g,h)=\theta(g,h)$, which gives
continuity in the pointwise topologies and proves the vanishing
assertion. For $s,g,h\in\Gamma$, we have
\[
\operatorname{Res}_\Gamma(s\cdot\theta)(g,h)
=\theta(s^{-1}gs,s^{-1}hs)
=(s\cdot\operatorname{Res}_\Gamma(\theta))(g,h).
\]
Thus $\operatorname{Res}_\Gamma$ is $\Gamma$-equivariant.
Finally, since $\rho$ fixes $\Gamma$ pointwise, for $g,h\in\Gamma$
we have
$\operatorname{Res}_\Gamma(\mathfrak{P}_\rho(\psi))(g,h)
=\psi(\rho(g),\rho(h))=\psi(g,h)$.
\end{proof}
\begin{proof}[Proof of Theorem~\ref{thm:ira-mif-intro}]
Let $\Gamma=\oplus_{1\leq i\leq n}\Gamma_i$ and let $\mu$ be an IRM
of $\Gamma$. By Proposition~\ref{prop:ira-ucp-correspondence}, we regard
$\mu$ as a $\Gamma$-invariant probability measure on
$\operatorname{PD}_1^{\mathrm R}(\Gamma\times\Gamma)$.

Let $\Lambda_i=\Gamma_i*\langle t_i\rangle$, where
$\langle t_i\rangle\cong\Z$, and
$\Lambda=\oplus_{1\leq i\leq n}\Lambda_i$. For each $i$, enumerate
$\Lambda_i\setminus\{e\}$ by $h_1^{(i)},h_2^{(i)},\ldots$.
By Hull and Osin~\cite[Proposition~5.3]{hull-osin2016}, choose
homomorphisms $\rho_j^{(i)}\colon\Lambda_i\to\Gamma_i$ fixing
$\Gamma_i$ pointwise such that $\rho_j^{(i)}(h_k^{(i)})\neq e$ for
$k\leq j$. Let $\rho_j=\oplus_{1\leq i\leq n}\rho_j^{(i)}$. Then
$\rho_j|_\Gamma=\operatorname{id}_\Gamma$.
If $g\in\Lambda\setminus\{e\}$, choose a nonidentity coordinate
$g_i=h_k^{(i)}$. For every $j\geq k$, we have
$\rho_j^{(i)}(g_i)\neq e$, so $\rho_j(g)\neq e$.

For each $j$, let $\eta_j=(\mathfrak{P}_{\rho_j})_*\mu$.
Since $\mu$ is $\Gamma$-invariant, it is invariant under the action
of $\Lambda$ through $\rho_j$. Thus Lemma~\ref{lem:pd-pullback}
gives that $\eta_j$ is a $\Lambda$-invariant probability measure on
$\operatorname{PD}_1(\Lambda\times\Lambda)$. Since this space is
compact metrizable, the space of probability measures on it is
compact metrizable in the weak topology. Pass to a subsequence such
that $\eta_j\to\eta$ weakly. For each $s\in\Lambda$, continuity of
the action gives $s_*\eta_j\to s_*\eta$ weakly. Since $s_*\eta_j=\eta_j$,
we get $s_*\eta=\eta$.

We claim that
$\eta(\operatorname{PD}_1^{\mathrm R}(\Lambda\times\Lambda))=1$.
Fix $g\in\Lambda\setminus\{e\}$. The set
$C_g=\{\theta\in\operatorname{PD}_1(\Lambda\times\Lambda):
\theta(g,e)=\theta(e,g)=0\}$ is closed in the pointwise topology.
Choose $j_0$ such that $\rho_j(g)\neq e$ for every $j\geq j_0$.
By Lemma~\ref{lem:pd-pullback}, both marginals of
$\mathfrak{P}_{\rho_j}(\psi)$ at $g$ are $\delta_{\rho_j(g),e}=0$
for every $\psi\in\operatorname{PD}_1^{\mathrm R}(\Gamma\times\Gamma)$
and every $j\geq j_0$. Hence $\eta_j(C_g)=1$ for every $j\geq j_0$.
Since $C_g$ is closed and $\eta_j\to\eta$ weakly,
\cite[Theorem~17.20(iii)]{Kechris1995} gives
$1=\limsup_j\eta_j(C_g)\leq\eta(C_g)$. By countability of $\Lambda$,
$\eta$ gives full measure to $\bigcap_{g\in\Lambda\setminus\{e\}}C_g$.
Since $\theta(e,e)=1$ for every
$\theta\in\operatorname{PD}_1(\Lambda\times\Lambda)$, this
intersection is precisely
$\operatorname{PD}_1^{\mathrm R}(\Lambda\times\Lambda)$.

Regard $\eta$ as a $\Lambda$-invariant probability measure on
$\operatorname{PD}_1^{\mathrm R}(\Lambda\times\Lambda)$. By
Proposition~\ref{prop:ira-ucp-correspondence},
$(\mathfrak T_\Lambda)_*\eta$ is an IRM of $\Lambda$. As noted in the
introduction, each $\Lambda_i$ is acylindrically hyperbolic with
trivial amenable radical. Thus Theorem~\ref{thm:hyperbolic_tf}
implies that $T_\theta$ is a Fourier multiplier for $\eta$-almost
every $\theta$. For each such $\theta$ and every pair of distinct
$a,b\in\Lambda$, we have
$\theta(a,b)=\tau_0(T_\theta(\lambda(a))\lambda(b)^*)=0$.

By Lemma~\ref{lem:pd-pullback}(2),
$\operatorname{Res}_\Gamma\circ\mathfrak P_{\rho_j}$ is the identity
on $\operatorname{PD}_1^{\mathrm R}(\Gamma\times\Gamma)$. Since
$\eta_j=(\mathfrak P_{\rho_j})_*\mu$, we have
$(\operatorname{Res}_\Gamma)_*\eta_j=\mu$, where $\mu$ is viewed
as a measure on $\operatorname{PD}_1(\Gamma\times\Gamma)$.
Since $\operatorname{Res}_\Gamma$ is continuous by
Lemma~\ref{lem:pd-pullback}(2) and $\eta_j\to\eta$ weakly, we have
$(\operatorname{Res}_\Gamma)_*\eta=\mu$.

Let $D_\Gamma=\{\psi\in\operatorname{PD}_1(\Gamma\times\Gamma): \psi(g,h)=0~\forall g\ne h\in\Gamma\}$. This set is closed in the pointwise topology and is contained in
$\operatorname{PD}_1^{\mathrm R}(\Gamma\times\Gamma)$.
By Lemma~\ref{lem:pd-pullback}(2),
$\operatorname{Res}_\Gamma(\theta)\in D_\Gamma$ for $\eta$-almost every
$\theta$. Therefore
$\mu(D_\Gamma)=\eta(\operatorname{Res}_\Gamma^{-1}(D_\Gamma))=1$.
By Proposition~\ref{prop:ira-ucp-correspondence}, the maps
corresponding to functions in $D_\Gamma$ are Fourier multipliers.
Thus $\Gamma$ is IRM rigid.

The assertion for IRAs follows from the
implication from IRM rigidity to IRA rigidity in the introduction.
\end{proof}
The following example shows that the conclusion of
Theorem~\ref{thm:hyperbolic_tf_boom} can fail for mixed identity-free
groups.
\begin{example}\label{ex:mif-boomerang}\normalfont
By \cite[Theorem~1.3]{J}, there exists a countable nontrivial locally
finite $3$-group $\Gamma$ which is mixed identity-free.
Thus $\Gamma$ is IRA rigid by Theorem~\ref{thm:ira-mif-intro}.
Every von Neumann subalgebra of $L(\Gamma)$ is boomerang, since
conjugation by each $\gamma\in\Gamma$ has finite order.
Choose $g\in\Gamma$ of order $3$ and set
$p=(1+\lambda(g)+\lambda(g)^2)/3$. Then $p$ is a projection with
$\tau(p)=1/3$, so $M=\bC p\oplus\bC(1-p)\subseteq L(\langle g\rangle)$
is a two-dimensional unital von Neumann subalgebra.
If $M=L(\Delta)$ for some subgroup $\Delta\leq\Gamma$, then
$|\Delta|=2$, which is impossible since $\Gamma$ is a $3$-group.
\end{example}
We now use Theorem~\ref{thm:ira-mif-intro} to prove
Theorem~\ref{thm:ira-boundaries-trees-intro}.
\begin{proof}[Proof of Theorem~\ref{thm:ira-boundaries-trees-intro}]
Each $\Gamma_i$ is mixed identity-free by \cite[Proposition~3.4 and Theorem~3.6]{ABGM2026} in case $(1)$ and \cite[Theorem~B]{FLMMS2022} in case $(2)$.
Apply Theorem~\ref{thm:ira-mif-intro}.
The Baumslag--Solitar case follows from \cite[Theorem~B and Lemma~8.6]{FLMMS2022}.
\end{proof}
Amenable groups with the ISR property were first obtained by
Jiang--Zhou~\cite{jiang2024example} for the group of finitely supported
permutations of $\mathbb N$. More examples of such amenable groups
can be found in \cite{dudko2024character,amrutam2025invariant}.
We can now obtain the following examples of IRM and IRA rigid groups.

\begin{corollary}\label{cor:jacobson}
The following hold true.
\begin{enumerate}
\renewcommand{\labelenumi}{\textup{(\roman{enumi})}}
\item There exists a two-generated infinite elementary amenable
group which is IRM and IRA rigid.
\item For each prime $p$, there exists an infinite locally finite
$p$-group which is IRM and IRA rigid.
\item There exist continuum many pairwise nonisomorphic finitely generated
nonamenable groups which are IRM and IRA rigid, contain no nonabelian free
subgroup, and have nontrivial amenable radical.
\end{enumerate}
\end{corollary}
\begin{proof}
The mixed identity-free groups in $(i)$ and $(ii)$ are provided by
\cite[Theorems~1.2 and~1.3]{J}, respectively, and those in $(iii)$ by
\cite[Corollary~1.3]{HL}. Apply Theorem~\ref{thm:ira-mif-intro}.
\end{proof}

We can also embed any countable amenable group into a finitely
generated amenable group that is IRM and IRA rigid.
\begin{corollary}\label{cor:embedding}
Let $\Gamma$ be a countable amenable group. Then $\Gamma$ embeds into a
finitely generated amenable group that is IRM and IRA rigid.
If the given group is elementary amenable, the overgroup can be
chosen elementary amenable.
\end{corollary}
\begin{proof}
By \cite[Corollary~1.2\textup{(1)}]{HL}, every countable
(elementary) amenable group embeds into a finitely generated
(elementary) amenable mixed identity-free group.
Apply Theorem~\ref{thm:ira-mif-intro}.
\end{proof}
We finish this section with some examples, where we can classify the IRMs completely. 
We first apply Theorem~\ref{thm:ira-mif-intro} to linear groups.
\begin{corollary}\label{cor:linear-irm}
Let $n\geq2$, let $k$ be an algebraically closed field and let $\Gamma\leq\PGL_n(k)$ be a countable Zariski dense subgroup. Then $\Gamma$ is IRM rigid. In particular, it is IRA rigid. The same holds for finite products of such groups.
\end{corollary}
\begin{proof}
Let $n\geq2$ and let $k$ be an algebraically
closed field. It follows from Tomanov~\cite[Theorem~5]{Tomanov} that
$\PGL_n(k)$ is mixed identity-free. The zero set of a mixed word is
Zariski closed, so every Zariski dense subgroup of $\PGL_n(k)$ is
mixed identity-free.
Since $\Gamma$ is mixed identity-free, apply Theorem~\ref{thm:ira-mif-intro}.
\end{proof}
For example, $\PSL_n(\bZ)$ is Zariski dense in $\PGL_n(\mathbb C)$, so it is IRM rigid for every $n\geq2$. More generally, let $\Gamma$ be as in Subsection~\ref{ksplit} with $\bfG=\PGL_n$. By Borel density, $\Gamma$ is Zariski dense in $\PGL_n(\overline K)$, where $\overline K$ is an algebraic closure of $K$, so $\Gamma$ is IRM rigid. For these groups, this strengthens the conclusion of Theorem~\ref{thm:ksplit} from IRAs to IRMs. However, not every group in Theorem~\ref{thm:ksplit} is mixed identity-free as the following example demonstrates.
\begin{example}\normalfont
Let $n\geq2$ and let $\bfG$ be the special orthogonal group of the bilinear form $b(x,y)=x_0y_0+\sum_{i=1}^n(x_iy_{n+i}+x_{n+i}y_i)$ on $\bQ^{2n+1}$, with the standard basis given by $e_0,\ldots,e_{2n}$. It is connected, absolutely simple, adjoint and $\bQ$-split of type $B_n$. Let $\sigma=\diag(1,-1,\ldots,-1)$ and let $u=1+N$, where $N(x)=x_{n+2}e_1-x_{n+1}e_2$. Both are nonidentity elements of $\bfG(\bZ)$. Fix $t\in\bfG(\bZ)$ and set $\sigma'=t\sigma t^{-1}$. Then $\sigma'$ acts by $-1$ on a hyperplane $H$, and $\sigma'u\sigma'^{-1}=1+N'$ with $N'=\sigma'N\sigma'^{-1}$. The images $P$ and $P'=\sigma'P$ of $N$ and $N'$ are totally isotropic planes, and the nonzero subspace $P\cap H$ is contained in $P\cap P'$. Let $\ell$ be a line in $P\cap P'$. Both $u$ and $\sigma'u\sigma'^{-1}$ act trivially on $\ell$ and on $\ell^\perp/\ell$. Hence they belong to the unipotent radical of the stabilizer of $\ell$ in $\bfG$, which is abelian. The resulting identity is a special case of Gordeev~\cite[Theorem~3(b)]{Gordeev:identities}. Thus the nonidentity element $[u,t\sigma t^{-1}ut\sigma t^{-1}]$ of $\bfG(\bZ)*\langle t\rangle$ is trivial under every homomorphism to $\bfG(\bZ)$ fixing $\bfG(\bZ)$ pointwise, and $\bfG(\bZ)$ is not mixed identity-free.
\end{example}

We now use character rigidity to describe the IRMs of an IRM rigid
group. By Lemma~\ref{lem:fourier-multiplier-map}, every IRM of such a
group is the pushforward under $\mathfrak{M}_\Gamma$ of a
$\Gamma$-invariant probability measure on $\operatorname{PD}_1(\Gamma)$.
Alekseev and Brugger~\cite[Theorem~5.1]{AB:irpdf} showed that, for an
icc group, such a measure with average $\delta_e$ is concentrated on
$\delta_e$.
As in Proposition~\ref{prop:onedim}, we use character rigidity to exclude
the finite-dimensional terms in an averaged character.
\begin{theorem}\label{prop:irm-char}
Let $\Gamma$ be a countable IRM rigid group with property
$\mathrm{CR}_{\mathrm{fin}}$. Let $\mu$ be an IRM of $\Gamma$. Then for
$\mu$-almost every $T$ there are $a\in[0,1]$ and
$\phi\in\operatorname{PD}_1(\Gamma)$ such that
$$T(x)=aT_\phi(x)+(1-a)\tau_0(x)1,\quad x\in L(\Gamma),$$
and $\phi$ is a uniform limit of positive-definite functions factoring
through finite quotients of $\Gamma$. 
\end{theorem}
\begin{proof}
Fix an IRM $\mu$ of $\Gamma$. Since $\Gamma$ is IRM rigid, $\mu$
gives full measure to the Fourier multipliers. By
Lemma~\ref{lem:fourier-multiplier-map}, the pushforward
$\nu=(\mathfrak{M}_\Gamma^{-1})_*\mu$ is a $\Gamma$-invariant
probability measure on $\operatorname{PD}_1(\Gamma)$, and
$(\mathfrak{M}_\Gamma)_*\nu=\mu$.

Since $\Gamma\in\mathrm{CR}_{\mathrm{fin}}$, we may list the irreducible
finite-dimensional unitary representations of $\Gamma$, up to unitary
equivalence, as $(\pi_i)_{i\geq1}$, with repetitions if there are only
finitely many. Let $\chi_i$ be the normalized character of $\pi_i$ and
let $N_j=\bigcap_{i\leq j}\ker\pi_i$. Since each $\pi_i$ factors through
a finite quotient, $(N_j)$ is a decreasing sequence of finite-index
normal subgroups. Every finite-dimensional unitary representation of
$\Gamma$ is trivial on some $N_j$, since it is a finite direct sum of
irreducible representations. We use the notations from Proposition~\ref{prop:finite-quotient-part}.
For each $\phi\in\operatorname{PD}_1(\Gamma)$, $\phi_\infty=a(\phi)F_Q(\phi)$ when $a(\phi)>0$.
The same identity holds when $a(\phi)=0$, since
$a(\phi)=\|P\xi\|^2=0$ implies $P\xi=0$. By
Proposition~\ref{prop:finite-quotient-part}(2) and (3), the map
$\phi\mapsto\phi_\infty$ is Borel and $\Gamma$-equivariant. Define
\[
\theta(g)=\int|\phi(g)-\phi_\infty(g)|^2\,d\nu(\phi).
\]
By Proposition~\ref{prop:finite-quotient-part}(1), $\theta$ is positive definite.
Using the same proposition and invariance of $\nu$, it follows
that $\theta$ is conjugation invariant. We claim that $\theta(g)=0$ for every $g\ne e$. 

If $\theta(e)=0$, then
$\theta=0$. Suppose $\theta(e)>0$. Then $\theta/\theta(e)$ is a
character. Since $\Gamma\in\mathrm{CR}_{\mathrm{fin}}$, as in the proof of
Proposition~\ref{prop:onedim}, we may write
$\theta=\theta(e)(c_0\delta_e+\sum_{i\geq1}c_i\chi_i)$, where
$c_i\geq0$ and $\sum_{i\geq0}c_i=1$.

Fix $i\geq1$ and choose a symmetric probability measure $\eta$ on
$N_i$ with full support. For every $n\in N_i$, we have
$n\in\ker\pi_i$, so $\chi_i(n)=1$. By the character expansion,
$\theta-\theta(e)c_i\chi_i$ is positive definite. Let
$(\sigma,\mathcal K,v)$ be a GNS triple for its restriction to $N_i$.
Then $\langle\sigma(n)v,v\rangle=\theta(n)-\theta(e)c_i$ for every
$n\in N_i$.

Since $\eta$ is symmetric, the operator
$\sigma(\eta)=\sum_{n\in N_i}\eta(n)\sigma(n)$ is self-adjoint.
Also, $\sigma(\eta^{*2k})=\sigma(\eta)^{2k}$. Since
$\sum_{n\in N_i}\eta^{*2k}(n)=1$, for every $k\geq1$ we have
\[
\begin{aligned}
\sum_{n\in N_i}\eta^{*2k}(n)\theta(n)-\theta(e)c_i
&=\sum_{n\in N_i}\eta^{*2k}(n)\bigl(\theta(n)-\theta(e)c_i\bigr)\\
&=\sum_{n\in N_i}\eta^{*2k}(n)\langle\sigma(n)v,v\rangle\\
&=\langle\sigma(\eta^{*2k})v,v\rangle\\
&=\|\sigma(\eta)^k v\|^2\geq0.
\end{aligned}
\]
Thus $\sum_{n\in N_i}\eta^{*2k}(n)\theta(n)\geq\theta(e)c_i$.

Now, fix
$\phi\in\operatorname{PD}_1(\Gamma)$ with GNS triple
$(\rho,\mathcal H,\xi)$. Let $\zeta=(1-P)\xi$. Observe that (see~\cite[proof of Proposition~F.3.2]{BHV:propertyT})
\[
|\phi(g)-\phi_\infty(g)|^2
=\langle(\rho\otimes\bar\rho)(g)(\zeta\otimes\bar\zeta),
          \zeta\otimes\bar\zeta\rangle.
\]
An $N_i$-invariant vector in
$(1-P)\mathcal H\otimes\overline{(1-P)\mathcal H}$ corresponds to a
Hilbert--Schmidt operator $S$ on $(1-P)\mathcal H$ commuting with
$\rho(N_i)$ (see~\cite[Proposition~A.1.12]{BHV:propertyT}). If $S\ne0$,
an eigenspace of $S^*S$ for a nonzero eigenvalue gives a nonzero
finite-dimensional $\rho(N_i)$-invariant subspace of $(1-P)\mathcal H$. The space $P\mathcal H$ contains every
finite-dimensional $\rho(\Gamma)$-invariant subspace, since the
representation on any such subspace is trivial on some $N_j$.
If $V\subseteq(1-P)\mathcal H$ is finite dimensional and
$\rho(N_i)$-invariant, the sum of its translates over representatives
of $\Gamma/N_i$ is finite dimensional and $\rho(\Gamma)$-invariant.
This sum lies in both $P\mathcal H$ and $(1-P)\mathcal H$, so $V=0$. Thus the projection onto the
$N_i$-invariant vectors in this tensor product is zero. Applying
Lemma~\ref{lem:averaging-operator} to the displayed coefficient gives
\[
\lim_{k\to\infty}\sum_{n\in N_i}\eta^{*2k}(n)
|\phi(n)-\phi_\infty(n)|^2=0.
\]
The averages are bounded by $1$, since
$|\phi(n)-\phi_\infty(n)|\leq\|\zeta\|^2\leq1$. Integrating over
$\nu$ and applying dominated convergence gives
$\lim_{k\to\infty}\sum_{n\in N_i}\eta^{*2k}(n)\theta(n)=0$.
Hence $\theta(e)c_i=0$. Since $i$ was arbitrary and $\theta(e)>0$,
we obtain $\theta=\theta(e)\delta_e$, proving the claim.

Since $\Gamma$ is countable, for $\nu$-almost every $\phi$ we have
$\phi(g)=\phi_\infty(g)$ for all $g\ne e$. Since
$\nu=(\mathfrak{M}_\Gamma^{-1})_*\mu$, for $\mu$-almost every $T$, $\phi=\mathfrak{M}_\Gamma^{-1}(T)$ satisfies these equalities.
Fix such a $T$. Note that $\phi(e)=1$ and
$\phi_\infty(e)=a(\phi)$, so
$\phi=a(\phi)F_Q(\phi)+(1-a(\phi))\delta_e$.
Since $T=\mathfrak{M}_\Gamma(\phi)=T_\phi$, it follows that
$T=a(\phi)T_{F_Q(\phi)}+(1-a(\phi))\tau_0(\cdot)1$.
By Proposition~\ref{prop:finite-quotient-part}(4), $F_Q(\phi)$ is a
uniform limit of normalized positive-definite functions factoring
through finite quotients. This finishes the proof.
\end{proof}
By Theorem~\ref{thm:char}, every group $\Gamma$ as in
Subsection~\ref{ksplit} has property $\mathrm{CR}_{\mathrm{fin}}$.
Thus Theorem~\ref{prop:irm-char} applies to such $\Gamma$ whenever it is
IRM rigid, for instance when it is mixed identity-free. Together with
Corollary~\ref{cor:linear-irm}, this gives the following.
\begin{corollary}\label{cor:irm-pgl}
Let $\Gamma$ be as in Subsection~\ref{ksplit} with $\bfG=\PGL_n$, and let $\mu$ be an ergodic IRM of $\Gamma$. Then there is $a\in[0,1]$ such that $\mu$-almost every $T$ is of the form $aT_\phi+(1-a)\tau_0(\cdot)1$, where $\phi\in\operatorname{PD}_1(\Gamma)$ is a uniform limit of positive-definite functions which factor via finite quotients of $\Gamma$.
\end{corollary}
\begin{proof}
By Borel density, $\Gamma$ is Zariski dense in $\PGL_n(\overline K)$. Thus $\Gamma$ is IRM rigid by Corollary~\ref{cor:linear-irm}. Apply Theorems~\ref{thm:char} and~\ref{prop:irm-char}.
\end{proof}
In particular, Corollary~\ref{cor:irm-pgl} applies when $n\geq3$ and
$\Gamma\leq\PGL_n(\bQ)$ is commensurable with $\PSL_n(\bZ)$. Let us make this explicit.
\begin{example}\normalfont
Let $n\geq3$ and let $\Gamma\leq\PGL_n(\bQ)$ be commensurable with
$\PSL_n(\bZ)$. We claim that the ergodic IRAs of $\Gamma$ are $\delta_{\bC1}$ and the
uniform measures on the conjugacy orbits of $L(H)$, where $H$ ranges
over the finite-index subgroups of $\Gamma$. Let $\mu$ be an ergodic IRA of $\Gamma$. Since $\Gamma$ is IRA rigid (by Theorem~\ref{thm:ksplit}),
for $\mu$-almost every $M$ we have $M=L(H)$ with $H=J(M)$, so
that $\E_M=T_{1_H}$. Using Lemma~\ref{lem:diagram-inclusions}, we obtain that the pushforward of $\mu$ under $M\mapsto\E_M$ is an
ergodic IRM. Applying Corollary~\ref{cor:irm-pgl}, we can find $a\in[0,1]$ such that for
$\mu$-almost every $M$ there is $\phi\in\operatorname{PD}_1(\Gamma)$ (a uniform limit of positive-definite functions factoring via finite
quotients) such that $T_{1_H}=aT_\phi+(1-a)\tau_0(\cdot)1$. Consequently, for each $g\in\Gamma$, we have
$1_H(g)=a\phi(g)+(1-a)\delta_e(g)$. Fix such an $M$. If $H=\{e\}$, then $M=\bC1$. Otherwise, take
$h\in H\setminus\{e\}$. Then $1=a\phi(h)$, and since $|\phi(h)|\leq1$,
we have $a=1$ and $\phi=1_H$. Choose a positive-definite function $f$
factoring via a finite quotient $\Gamma/N$ such that
$\|1_H-f\|_\infty<1/3$. For every $g\in N$ we have $f(g)=f(e)$, so
$|1_H(g)-1|\leq |1_H(g)-f(g)|+|f(e)-1|<2/3$. Since $1_H(g)$ is
either $0$ or $1$, we get $g\in H$. Thus $N\leq H$, and $H$ has
finite index in $\Gamma$. Since $\Gamma$ is commensurable with the finitely generated group
$\PSL_n(\bZ)$, it is finitely generated. Hence it has only countably
many finite-index subgroups. Each has a finite conjugacy orbit, so
$\mu$ gives full measure to a countable union of finite invariant
sets, including $\{\bC1\}$. By ergodicity, one of these orbits has
full measure, and by invariance, $\mu$ is uniform on that orbit.
Conversely, the uniform measure on each such orbit is invariant
and ergodic.
\end{example}
\bibliographystyle{abbrv}
\begingroup
\color{black}
\hypersetup{linkcolor=black,citecolor=black,urlcolor=black}
\bibliography{refs}

@article {Effros,
    AUTHOR = {Effros, Edward G.},
     TITLE = {The {B}orel space of von {N}eumann algebras on a separable
              {H}ilbert space},
   JOURNAL = {Pacific J. Math.},
  FJOURNAL = {Pacific Journal of Mathematics},
    VOLUME = {15},
      YEAR = {1965},
     PAGES = {1153--1164},}

@misc{ABGM2026,
  author = {Amrutam, Tattwamasi and Banerjee, Arka and Gulbrandsen, Daniel L. and Mishra, Pratyush},
  title = {Boundary Dynamics, Cubical Actions, and Infinite Girth},
  year = {2026},
  eprint = {2608.31124},
  archivePrefix = {arXiv},
  primaryClass = {math.GR},
  url = {https://arxiv.org/abs/2608.31124},
  note = {arXiv:2608.31124},
}

@misc{AFJ:uniformly-recurrent,
  author = {Amrutam, Tattwamasi and Fima, Pierre and Jiang, Yongle},
  title = {Uniformly recurrent subalgebras in finite von {Neumann} algebras},
  year = {2026},
  eprint = {2606.09673},
  archivePrefix = {arXiv},
  primaryClass = {math.OA},
  url = {https://arxiv.org/abs/2606.09673},
  note = {arXiv:2606.09673},
}

@article{AMS,
  author = {Antol\'in, Yago and Minasyan, Ashot and Sisto, Alessandro},
  title = {Commensurating endomorphisms of acylindrically hyperbolic groups and applications},
  journal = {Groups Geom. Dyn.},
  volume = {10},
  pages = {1149--1210},
  year = {2016},
  doi = {10.4171/GGD/379},
}

@misc{AP:II1,
  author = {Anantharaman, Claire and Popa, Sorin},
  title = {An introduction to {$\mathrm{II}_1$} factors},
  year = {2017},
  note = {Draft},
  url = {https://www.math.ucla.edu/~popa/Books/IIunV15.pdf},
}

@article{BBHP,
  author = {Bader, Uri and Boutonnet, R\'emi and Houdayer, Cyril and
              Peterson, Jesse},
  title = {Charmenability of arithmetic groups of product type},
  journal = {Invent. Math.},
  fjournal = {Inventiones Mathematicae},
  volume = {229},
  year = {2022},
  number = {3},
  pages = {929--985},
  issn = {0020-9910,1432-1297},
  mrclass = {22D10 (20C15 22D25 22E40 37A55 43A35 46L10)},
  mrnumber = {4462622},
  mrreviewer = {Robert\ S.\ Doran},
  doi = {10.1007/s00222-022-01117-w},
  url = {https://doi.org/10.1007/s00222-022-01117-w},
}

@article{BCD,
  author = {Bekka, M. and Cowling, M. and de la Harpe, P.},
  title = {Some groups whose reduced {$C^*$}-algebra is simple},
  journal = {Publ. Math. Inst. Hautes \'Etudes Sci.},
  volume = {80},
  pages = {117--134},
  year = {1994},
  doi = {10.1007/BF02698898},
  url = {https://www.numdam.org/articles/10.1007/BF02698898/},
}

@article{Bekka,
  author = {Bachir Bekka},
  title = {Operator-algebraic superrigidity for {$\mathrm{SL}_n(\mathbb Z)$}, $n\geq3$},
  journal = {Inventiones Mathematicae},
  volume = {169},
  number = {2},
  pages = {401--425},
  year = {2007},
  doi = {10.1007/s00222-007-0050-5},
  eprint = {math/0609102},
  archiveprefix = {arXiv},
  note = {arXiv:math/0609102},
}

@book{bekka-de-la-harpe2020,
  author = {Bekka, Bachir and de la Harpe, Pierre},
  title = {Unitary Representations of Groups, Duals, and Characters},
  series = {Mathematical Surveys and Monographs},
  volume = {250},
  publisher = {American Mathematical Society},
  address = {Providence, RI},
  year = {2020},
  doi = {10.1090/surv/250},
  eprint = {1912.07262},
  archiveprefix = {arXiv},
}

@article{CDS2023,
  author = {Chifan, Ionu\c{t} and Das, Sayan and Sun, Bin},
  title = {Invariant subalgebras of von {Neumann} algebras arising from negatively curved groups},
  journal = {J. Funct. Anal.},
  volume = {285},
  number = {9},
  pages = {110098},
  year = {2023},
  doi = {10.1016/j.jfa.2023.110098}
}

@misc{DGGHL:char_rigid,
  title = {Non-uniform higher-rank lattices are character rigid},
  author = {Alon Dogon and Michael Glasner and Yuval Gorfine and Liam Hanany and Arie Levit},
  year = {2025},
  eprint = {2507.21862},
  archiveprefix = {arXiv},
  primaryclass = {math.GR},
  url = {https://arxiv.org/abs/2507.21862},
  note = {arXiv:2507.21862},
}

@article{DGO,
  author = {Dahmani, Fran\c{c}ois and Guirardel, Vincent and Osin, Denis},
  title = {Hyperbolically embedded subgroups and rotating families in groups acting on hyperbolic spaces},
  journal = {Mem. Amer. Math. Soc.},
  volume = {245},
  number = {1156},
  year = {2017},
  doi = {10.1090/memo/1156},
  eprint = {1111.7048},
  archiveprefix = {arXiv},
  note = {arXiv:1111.7048},
}

@article{FLMMS2022,
  author = {Fima, Pierre and {Le Ma{\^\i}tre}, Fran{\c c}ois and Moon, Soyoung and Stalder, Yves},
  title = {A characterization of high transitivity for groups acting on trees},
  journal = {Discrete Anal.},
  year = {2022},
  pages = {Paper No. 8, 63 pp.},
  doi = {10.19086/da.37645},
  eprint = {2003.11116},
  archivePrefix = {arXiv},
  primaryClass = {math.GR},
  url = {https://arxiv.org/abs/2003.11116},
  note = {arXiv:2003.11116},
}

@misc{FLMP:ucp,
  author = {Fima, Pierre and {Le Ma{\^\i}tre}, Fran{\c c}ois and Mukherjee, Kunal and Patri, Issan},
  title = {Spaces of {UCP} maps and subalgebras of von {Neumann} algebras},
  year = {2026},
  eprint = {2607.14358},
  archivePrefix = {arXiv},
  primaryClass = {math.OA},
  url = {https://arxiv.org/abs/2607.14358},
  note = {arXiv:2607.14358},
}

@article{G:lin_IRS,
  author = {Glasner, Yair},
  title = {Invariant random subgroups of linear groups},
  journal = {Israel J. Math.},
  volume = {219},
  number = {1},
  pages = {215--270},
  year = {2017},
  doi = {10.1007/s11856-017-1479-x},
  eprint = {1407.2872},
  archiveprefix = {arXiv},
  note = {With an appendix by Tsachik Gelander and Yair Glasner; arXiv:1407.2872}
}

@misc{GHL:ec,
  author = {Glasner, Yair and Hartnick, Tobias and Lederle, Waltraud},
  title = {Elementwise conservative actions and new constructions of boomerang subgroups},
  year = {2025},
  eprint = {2512.20596},
  archiveprefix = {arXiv},
  primaryclass = {math.GR},
  url = {https://arxiv.org/abs/2512.20596},
  note = {arXiv:2512.20596}
}

@misc{JL,
  author = {Yongle Jiang and Hongyi Li},
  title = {Classification of Invariant Subalgebras in a Class of Factors with Property {(T)}},
  year = {2026},
  eprint = {2601.06353},
  archiveprefix = {arXiv},
  primaryclass = {math.OA},
  note = {arXiv:2601.06353},
}

@article{KS2016,
  author = {Kar, Aditi and Sageev, Michah},
  title = {Ping Pong on {CAT}(0) cube complexes},
  journal = {Comment. Math. Helv.},
  volume = {91},
  number = {3},
  pages = {543--561},
  year = {2016},
  doi = {10.4171/CMH/395}
}

@book{Kechris1995,
  author = {Kechris, Alexander S.},
  title = {Classical Descriptive Set Theory},
  series = {Graduate Texts in Mathematics},
  volume = {156},
  publisher = {Springer-Verlag},
  address = {New York},
  year = {1995},
  doi = {10.1007/978-1-4612-4190-4},
  isbn = {978-0-387-94374-9},
}

@misc{Lederle2026,
  author = {Lederle, Waltraud},
  title = {Critical exponents of boomerang subgroups in the free group},
  year = {2025},
  eprint = {2512.20767},
  archiveprefix = {arXiv},
  primaryclass = {math.GR},
  url = {https://arxiv.org/abs/2512.20767},
  note = {arXiv:2512.20767}
}

@article{NSZ:intermediate_factor,
  author = {Nevo, Amos and Zimmer, Robert J.},
  title = {A generalization of the intermediate factors theorem},
  journal = {J. Anal. Math.},
  fjournal = {Journal d'Analyse Math\'ematique},
  volume = {86},
  year = {2002},
  pages = {93--104},
  issn = {0021-7670,1565-8538},
  mrclass = {22E46},
  mrnumber = {1894478},
  mrreviewer = {Tuong Ton That},
  doi = {10.1007/BF02786645},
  url = {https://doi.org/10.1007/BF02786645},
}

@article{PT:char_rigidity,
  author = {Peterson, Jesse and Thom, Andreas},
  title = {Character rigidity for special linear groups},
  journal = {J. Reine Angew. Math.},
  fjournal = {Journal f\"ur die Reine und Angewandte Mathematik. [Crelle's
              Journal]},
  volume = {716},
  year = {2016},
  pages = {207--228},
  issn = {0075-4102,1435-5345},
  mrclass = {22D20 (20G30 22D10 22F10)},
  mrnumber = {3518376},
  mrreviewer = {Anton\ Deitmar},
  doi = {10.1515/crelle-2014-0009},
  url = {https://doi.org/10.1515/crelle-2014-0009},
}

@article{abbott2019property,
  author = {Abbott, Carolyn R. and Dahmani, Fran\c{c}ois},
  title = {Property {$P_{\mathrm{naive}}$} for acylindrically hyperbolic groups},
  journal = {Math. Z.},
  volume = {291},
  pages = {555--568},
  year = {2019},
  doi = {10.1007/s00209-018-2094-1},
  eprint = {1610.04143},
  archiveprefix = {arXiv},
  note = {arXiv:1610.04143},
}

@misc{amrutam-relative-isr,
  author = {Amrutam, Tattwamasi},
  title = {On Relative Invariant Subalgebra Rigidity Property},
  year = {2026},
  eprint = {2604.04835},
  archiveprefix = {arXiv},
  primaryclass = {math.OA},
  url = {https://arxiv.org/abs/2604.04835},
  note = {arXiv:2604.04835},
}

@article{amrutam2023invariant,
  author = {Amrutam, T. and Jiang, Y.},
  title = {On invariant von {N}eumann subalgebras rigidity property},
  journal = {J. Funct. Anal.},
  volume = {284},
  number = {5},
  year = {2023},
  pages = {Paper No. 109804},
}

@misc{amrutam2026confined,
  author = {Amrutam, Tattwamasi and Jiang, Yongle},
  title = {{$C^*$}-simplicity, confined subalgebras and operator algebraic uniform recurrence},
  year = {2026},
  eprint = {2604.18458},
  archiveprefix = {arXiv},
  primaryclass = {math.OA},
  url = {https://arxiv.org/abs/2604.18458},
  note = {arXiv:2604.18458},
}

@misc{amrutam2026crossed,
  author = {Amrutam, Tattwamasi and Dudko, Artem and Jiang, Yongle and Skalski, Adam},
  title = {Invariant von {Neumann} subalgebras of crossed products and tensor products},
  year = {2026},
  eprint = {2610.19023},
  archiveprefix = {arXiv},
  primaryclass = {math.OA},
  url = {https://arxiv.org/abs/2610.19023},
  note = {arXiv:2610.19023},
}

@article{amrutam2309amenable,
  author = {Amrutam, Tattwamasi and Hartman, Yair and Oppelmayer, Hanna},
  title = {On the amenable subalgebras of group von {Neumann} algebras},
  journal = {J. Funct. Anal.},
  volume = {288},
  number = {2},
  pages = {110718},
  year = {2025},
  doi = {10.1016/j.jfa.2024.110718},
  eprint = {2309.10494},
  archiveprefix = {arXiv},
  primaryclass = {math.OA},
  note = {arXiv:2309.10494},
}

@book{brown2008cstar,
  author = {Nathanial P. Brown and Narutaka Ozawa},
  title = {{$C^*$}-Algebras and Finite-Dimensional Approximations},
  series = {Graduate Studies in Mathematics},
  volume = {88},
  publisher = {American Mathematical Society},
  address = {Providence, RI},
  year = {2008},
  doi = {10.1090/gsm/088},
}

@article{dudko2024character,
  title = {A character approach to the {ISR} property},
  author = {Dudko, Artem and Jiang, Yongle},
  journal = {arXiv preprint arXiv:2410.14517},
  year = {2024},
  note = {arXiv:2410.14517},
}

@article{glasner2024boomerang,
  author = {Yair Glasner and Waltraud Lederle},
  title = {Strong subgroup recurrence and the {N}evo--{S}tuck--{Z}immer theorem},
  journal = {Proc. Lond. Math. Soc. (3)},
  year = {2025},
  volume = {130},
  number = {6},
  pages = {e70053},
  doi = {10.1112/plms.70053},
  eprint = {2208.14758},
  archiveprefix = {arXiv},
  primaryclass = {math.GR},
  note = {arXiv:2208.14758},
}

@article{haagerup1998effros,
  author = {Uffe Haagerup and Carl Winsl{\o}w},
  title = {The {Effros--Mar\'echal} topology in the space of von {Neumann} algebras},
  journal = {American Journal of Mathematics},
  volume = {120},
  number = {3},
  pages = {567--617},
  year = {1998},
  doi = {10.1353/ajm.1998.0022},
}

@article{hull-osin2016,
  author = {Hull, Michael and Osin, Denis},
  title = {Transitivity degrees of countable groups and acylindrical hyperbolicity},
  journal = {Israel J. Math.},
  volume = {216},
  number = {1},
  pages = {307--353},
  year = {2016},
  doi = {10.1007/s11856-016-1411-9},
  eprint = {1501.04182},
  archivePrefix = {arXiv},
  primaryClass = {math.GR},
  url = {https://arxiv.org/abs/1501.04182},
  note = {arXiv:1501.04182},
}

@article{jiang-scm,
  author = {Yongle Jiang},
  title = {Maximal {von Neumann} subalgebras arising from maximal subgroups},
  journal = {Sci. China Math.},
  volume = {64},
  number = {10},
  pages = {2295--2312},
  year = {2021},
  doi = {10.1007/s11425-020-1671-9},
  eprint = {1911.10483},
  archiveprefix = {arXiv},
  primaryclass = {math.OA},
  note = {arXiv:1911.10483},
}

@article{jiang2024example,
  author = {Jiang, Y. and Zhou, X.},
  title = {An example of an infinite amenable group with the {ISR} property},
  journal = {Math. Z.},
  VOLUME = {307},
      YEAR = {2024},
    NUMBER = {2},
     PAGES = {Paper No. 23, 11 pp.},
}

@article{osin2016acylindrically,
  author = {Osin, Denis},
  title = {Acylindrically hyperbolic groups},
  journal = {Trans. Amer. Math. Soc.},
  volume = {368},
  pages = {851--888},
  year = {2016},
  eprint = {1304.1246},
  archiveprefix = {arXiv},
  note = {arXiv:1304.1246},
}

@article{stuck1994stabilizers,
  title = {Stabilizers for ergodic actions of higher rank semisimple groups},
  author = {Stuck, Garrett and Zimmer, Robert J},
  journal = {Annals of Mathematics},
  pages = {723--747},
  year = {1994},
  publisher = {JSTOR},
}

@article{kalantar2023invariant,
  author = {Kalantar, M. and Panagopoulos, N.},
  title = {On invariant subalgebras of group {$C^*$} and von {N}eumann algebras},
  journal = {Ergodic Theory Dynam. Systems},
  volume = {43},
  number = {10},
  year = {2023},
  pages = {3341--3353},
  doi = {10.1017/etds.2022.76},
  url = {https://doi.org/10.1017/etds.2022.76},
}

@article{abert2014kesten,
  author = {Ab\'ert, Mikl\'os and Glasner, Yair and Vir\'ag, B\'alint},
  title = {Kesten's theorem for invariant random subgroups},
  journal = {Duke Math. J.},
  volume = {163},
  number = {3},
  pages = {465--488},
  year = {2014},
  doi = {10.1215/00127094-2410064},
  url = {https://doi.org/10.1215/00127094-2410064},
}

@article{marechal1973topologie,
  author = {Mar\'echal, Odile},
  title = {Topologie et structure bor\'elienne sur l'ensemble des alg\`ebres de von {Neumann}},
  journal = {C. R. Acad. Sci. Paris S\'er. A-B},
  volume = {276},
  pages = {A847--A850},
  year = {1973},
}

@misc{FLMP:marechal,
  author = {Fima, Pierre and {Le Ma{\^\i}tre}, Fran{\c c}ois and Mukherjee, Kunal and Patri, Issan},
  title = {Michael's selection theorem and applications to the {Mar\'echal} topology},
  year = {2024},
  eprint = {2407.05776},
  archivePrefix = {arXiv},
  primaryClass = {math.OA},
  url = {https://arxiv.org/abs/2407.05776},
  note = {arXiv:2407.05776},
}

@article{minasyan-osin2015trees,
  author = {Minasyan, Ashot and Osin, Denis},
  title = {Acylindrical hyperbolicity of groups acting on trees},
  journal = {Math. Ann.},
  volume = {362},
  number = {3--4},
  pages = {1055--1105},
  year = {2015},
  doi = {10.1007/s00208-014-1138-z},
  eprint = {1310.6289},
  archivePrefix = {arXiv},
  primaryClass = {math.GR},
  url = {https://arxiv.org/abs/1310.6289},
  note = {arXiv:1310.6289},
}

@misc{FOT2026,
  author        = {Fima, Pierre and Oppelmayer, Hanna and Tao, Leyan},
  title         = {Random {$C^*$}-subalgebras},
  year          = {2026},
  eprint        = {2610.28295},
  archivePrefix = {arXiv},
  primaryClass  = {math.OA},
  doi           = {10.48550/arXiv.2610.28295},
  url           = {https://arxiv.org/abs/2610.28295},
  note = {arXiv:2610.28295},
}

@misc{AKP:inner-selfless,
  author = {Amrutam, Tattwamasi and Klisse, Mario and Pagliero, Matteo},
  title = {Inner {$C^*$}-selfless groups},
  year = {2026},
  eprint        = {2610.40198},
  archivePrefix = {arXiv},
  primaryClass  = {math.OA},
  doi           = {10.48550/arXiv.2610.40198},
  url           = {https://arxiv.org/pdf/2610.40198},
  note = {arXiv:2610.40198},
}

@article{LBMB2018,
  author = {Le Boudec, Adrien and Matte Bon, Nicol{\'a}s},
  title = {Subgroup dynamics and {$C^\ast$}-simplicity of groups
           of homeomorphisms},
  journal = {Ann. Sci. {\'E}c. Norm. Sup{\'e}r. (4)},
  volume = {51},
  number = {3},
  pages = {557--602},
  year = {2018},
  doi = {10.24033/asens.2361}
}

@article{LeBoudec2017,
  author = {Le Boudec, Adrien},
  title = {{$C^*$}-simplicity and the amenable radical},
  journal = {Inventiones Mathematicae},
  volume = {209},
  number = {1},
  pages = {159--174},
  year = {2017},
  doi = {10.1007/s00222-016-0706-0}
}

@article{breuillard2017c,
  title={{C}{*}-simplicity and the unique trace property for discrete groups},
  author={Breuillard, Emmanuel and Kalantar, Mehrdad and Kennedy, Matthew and Ozawa, Narutaka},
  journal={Publications math{\'e}matiques de l'IH{\'E}S},
  volume={126},
  pages={35--71},
  year={2017}
}

@article{J,
  author = {Jacobson, Bryan},
  title = {A mixed identity-free elementary amenable group},
  journal = {Communications in Algebra},
  volume = {49},
  number = {1},
  pages = {235--241},
  year = {2021},
  doi = {10.1080/00927872.2020.1797073}
}

@article{amrutam2025invariant,
  title={Invariant subalgebras rigidity for von Neumann algebras of groups arising as certain semidirect products},
  author={Amrutam, Tattwamasi and Dudko, Artem and Jiang, Yongle and Skalski, Adam},
  journal={arXiv preprint arXiv:2507.12824},
  year={2025},
  note = {arXiv:2507.12824},
}

@misc{HL,
  author = {Hyde, James and Lodha, Yash},
  title = {Embeddings into highly transitive and mixed identity free groups},
  year = {2025},
  eprint = {2509.09788},
  archivePrefix = {arXiv},
  primaryClass = {math.GR},
  url = {https://arxiv.org/abs/2509.09788},
  note = {arXiv:2509.09788},
}

@article{abert2017growth,
  author = {Ab{\'e}rt, Mikl{\'o}s and Bergeron, Nicolas and Biringer, Ian
            and Gelander, Tsachik and Nikolov, Nikolay and Raimbault, Jean
            and Samet, Iddo},
  title = {On the growth of {$L^2$}-invariants for sequences of lattices
           in {Lie} groups},
  journal = {Ann. of Math. (2)},
  volume = {185},
  number = {3},
  pages = {711--790},
  year = {2017},
  doi = {10.4007/annals.2017.185.3.1},
}

@article{duchesne2015geometric,
  author = {Duchesne, Bruno and Glasner, Yair and Lazarovich, Nir
            and L{\'e}cureux, Jean},
  title = {Geometric density for invariant random subgroups of groups
           acting on {CAT(0)} spaces},
  journal = {Geom. Dedicata},
  volume = {175},
  number = {1},
  pages = {249--256},
  year = {2015},
  doi = {10.1007/s10711-014-0038-4},
}

@article{creutz2017stabilizers,
  author = {Creutz, Darren and Peterson, Jesse},
  title = {Stabilizers of ergodic actions of lattices and commensurators},
  journal = {Trans. Amer. Math. Soc.},
  volume = {369},
  number = {6},
  pages = {4119--4166},
  year = {2017},
  url = {https://arxiv.org/abs/1303.3949},
}

@article{bowen2015irs,
  author = {Bowen, Lewis},
  title = {Invariant random subgroups of the free group},
  journal = {Groups Geom. Dyn.},
  volume = {9},
  number = {3},
  pages = {891--916},
  year = {2015},
  doi = {10.4171/GGD/331},
}

@article{bader2016amenable,
  author = {Bader, Uri and Duchesne, Bruno and L{\'e}cureux, Jean},
  title = {Amenable invariant random subgroups},
  journal = {Israel J. Math.},
  volume = {213},
  number = {1},
  pages = {399--422},
  year = {2016},
  doi = {10.1007/s11856-016-1324-7},
  note = {With an appendix by Phillip Wesolek},
}

@article{becker2019stability,
  author = {Becker, Oren and Lubotzky, Alexander and Thom, Andreas},
  title = {Stability and invariant random subgroups},
  journal = {Duke Math. J.},
  volume = {168},
  number = {12},
  pages = {2207--2234},
  year = {2019},
  doi = {10.1215/00127094-2019-0024},
}

@misc{zhou2024noncommutative,
  author = {Zhou, Shuoxing},
  title = {Noncommutative topological boundaries and amenable invariant
           random intermediate subalgebras},
  year = {2024},
  eprint = {2407.10905},
  archivePrefix = {arXiv},
  primaryClass = {math.OA},
  url = {https://arxiv.org/abs/2407.10905},
  note = {arXiv:2407.10905. With an appendix by Tattwamasi Amrutam
          and Yongle Jiang},
}

@article{CFP1996,
  author = {Cannon, James W. and Floyd, William J. and Parry, Walter R.},
  title = {Introductory notes on {Richard Thompson}'s groups},
  journal = {Enseign. Math. (2)},
  volume = {42},
  number = {3--4},
  pages = {215--256},
  year = {1996},
  doi = {10.5169/seals-87877},
  url = {https://doi.org/10.5169/seals-87877},
}

@misc{AB:irpdf,
  author = {Alekseev, Vadim and Brugger, Rahel},
  title = {On invariant random positive definite functions},
  year = {2018},
  note = {arXiv:1804.10471}
}

@book{BHV:propertyT,
  author = {Bekka, Bachir and de la Harpe, Pierre and Valette, Alain},
  title = {Kazhdan's Property ({T})},
  series = {New Mathematical Monographs},
  volume = {11},
  publisher = {Cambridge University Press},
  year = {2008},
  doi = {10.1017/CBO9780511542749}
}

@article{Gordeev:identities,
  author = {Gordeev, N. L.},
  title = {Freedom in conjugacy classes of simple algebraic groups and identities with constants},
  journal = {St. Petersburg Math. J.},
  volume = {9},
  number = {4},
  pages = {709--723},
  year = {1998}
}

@article{Tomanov,
  author = {Tomanov, G. M.},
  title = {Generalized group identities in linear groups},
  journal = {Math. USSR-Sb.},
  volume = {51},
  number = {1},
  pages = {33--46},
  year = {1985},
  doi = {10.1070/SM1985v051n01ABEH002845}
}
\endgroup
\end{document}